\documentclass{amsart}
\usepackage{tikz-cd}
\usepackage{enumitem}
\usepackage[a4paper,margin= 3 cm]{geometry}
\usepackage[hyphens]{url}
\usepackage{amsthm,amsmath,amssymb,amsfonts}
\usepackage[hidelinks]{hyperref}
\usepackage[indentafter]{titlesec}
\usepackage{mathrsfs}
\usepackage{bm}
\numberwithin{equation}{section}
\titleformat{name=\section}{}{\thetitle.}{0.8em}{\centering\scshape}
\titleformat{\subsection}
  {\normalfont\bfseries}
  {\thesubsection}
  {1em}
  {}
\theoremstyle{plain}
\newtheorem*{theorem*}{Theorem}
\newtheorem{theorem}{\textbf{Theorem}}[section]
\newtheorem{lemma}[theorem]{\textbf{Lemma}}
\newtheorem{proposition}[theorem]{\textbf{Proposition}}
\newtheorem{corollary}[theorem]{\textbf{Corollary}}

\theoremstyle{definition}
\newtheorem{definition}[theorem]{\textbf{Definition}}
\newtheorem{example}[theorem]{\textbf{Example}}
\newtheorem{remark}[theorem]{\textbf{Remark}}

\title[Bruhat decompositions of operator algebras]
{Bruhat decompositions of operator algebras}
\author{Thibaut Lescure}
\address{LMNO\\ 6 Boulevard Maréchal Juin\\ 14000 Caen \\ FRANCE}
\email{thibaut.lescure@etu.unicaen.fr}
\begin{document}

\begin{abstract}
We introduce and study a notion of decomposition of a C$^*$-algebra over a Coxeter system based on Tits' definition of a $W$-distance. When such a decomposition comes with suitable conditional expectations, we build an associated Fock Hilbert module and reduced C$^*$-algebra $A^r$. This unifies constructions of Voiculescu \cite{Voi85} and Caspers--Fima \cite{CF17} and provides a noncommutative analogue of the situation of a discrete group $G$ acting on a building, in which case $A^r \cong C^*_r(G)$. 

We construct covariance C$^*$-algebras $\mathscr{C}(i)\supset A^r$ as a noncommutative analogue of the crossed product $C(\Omega)\rtimes_r G \supset C^*_r(G) $ where $\Omega$ is Caprace and Lécureux's minimal combinatorial compactification \cite{CL11} of the locally finite building. Following the approach of Hasegawa \cite{Has17}  and Klisse \cite{Kli25}, we prove a universal property for the covariance algebras. This structural result yields different approximation properties, some of which are new even for the group case.
\end{abstract}
\maketitle

\section{Introduction}

Let $X$ be a building of type $(W,S)$. Let $G$ be a discrete group. The axioms of a $W$-distance tell us that $W$-isometric actions of $G$ on $X$ coincide with the data of partitions of $G$ of the form \[G = \biguplus_{w\in W}C_{i,j}(w),\] for all $i,j \in \mathscr{I} = G\backslash X$, satisfying a number of axioms coming from a $G$-equivariant identification $X \cong \coprod_{i\in \mathscr{I}}G/H_i$, with $H_i = C_{i,i}(e)$ (see Theorem \ref{CTB}). We call such a decomposition a $W$-decomposition of $G$. Bruhat decompositions of $G$ (for example coming from a BN-pair) correspond to the particular case of a chamber-transitive action (see \cite[Chapter 6]{AB08}). Passing to the reduced C$^*$-algebra, we obtain a direct sum decomposition 
\[C^*_r(G) = \overline{\bigoplus_{w\in W}L_{i,j}(w)},\] where each $L_{i,j}(w)$ is spanned by the unitaries coming from elements of $C_{i,j}(w)$. Translating the group-theoretic axioms, we get purely operator-theoretic notions of $W$-decomposition over an index set $\mathscr{I}$ and of Bruhat decomposition (see Definition \ref{OB}).  The first feature of this framework is that it unifies different, apparently unrelated constructions: Iwahori-Hecke C$^*$-algebras with positive real multiparameter, Voiculescu's amalgamated free products and graph product C$^*$-algebras are all perfectly natural examples of operator algebras with Bruhat decompositions. We also ask for the existence of conditional expectations to the base (or Borel) C$^*$-subalgebras $E_i: A \rightarrow B_i= \overline{L_{i,i}(e)}$ for all $i\in \mathscr{I}$ satisfying the condition \[\forall w\in W\setminus \{e\}, E_i(L_{i,i}(w)) = 0.\] With these conditional expectations, we are able to build Fock modules with natural direct sum decompositions $X_j = \bigoplus_{w\in W}X_{i,j}(w) = L^2(A,E_j)$. Putting all these right Hilbert $B_j$-modules together into a single right Hilbert module over $\mathcal{P} =\prod_{j\in \mathscr{I}}B_j$ which we call $\mathcal{X}$, we define a reduced C$^*$-algebra $A^r \subset \mathcal{L}_\mathcal{P}(\mathcal{X})$ (see Definition \ref{red}) which generalizes the reduced C$^*$-algebra constructed by Voiculescu and Caspers-Fima \cite{CF17}. The reduced C$^*$-algebra inherits a similar expected $W$-decomposition $A^r = \overline{\sum_{w\in W}F_{i,j}(w)}$. In the particular case of a Bruhat decomposition of some C$^*$-algebra $A$, we show that a conditional expectation $E : A \rightarrow B$ exists if we assume the existence of a coherent family of smaller conditional expectations $E_I : A_I \rightarrow B$ defined on parabolic C$^*$-subalgebras associated to spherical subsets $I\subset S$ of cardinality at most $3$ (see Theorem \ref{VV}).

We also claim that this framework is adapted to the study of operator algebraic properties of groups acting on buildings, as done in many different articles, e.g. \cite{Rob00, RS99, MRV24, KS91, Lec09}. Inspired by the work of Klisse on boundaries of graphs \cite{Kli23}, graph products C$^*$-algebras \cite{Kli25} and the work of Hasegawa on amalgamated free products C$^*$-algebras \cite{Has19}, we define what we call the covariance C$^*$-algebra based at $i$, denoted by $\mathscr{C}(i)$ for all $i\in \mathscr{I}$. For all $i\in \mathscr{I},$ $\mathscr{C}(i)$ is defined as the closure of the span of the path operators from $i$ to $i$. For any $i,j \in \mathscr{I},$ a path operator from $i$ to $j$ is a finite product of creation, annihilation, or elementary diagonal operators associated with elements $a_1,\ldots a_n$ such that for all $1 \leq k \leq n$ we have $a_k\in F_{i_{k-1},i_k}(s_k)$ for some $i_0,\ldots i_n \in \mathscr{I}, s_1, \ldots s_n \in S$ such that $i_0 = i$ and $i_n = j.$ We also define the global covariance C$^*$-algebra $\mathscr{C}$ as the C$^*$-subalgebra of $\mathcal{L}_{\mathcal{P}}(\mathcal{X})$ generated by arbitrary path operators. As each element of one of the subspaces of the form $F_{i,j}(s)$ decomposes as a sum of a creation, an annihilation and an elementary diagonal operator we have $A^r \subset \mathscr{C}(i) \subset \mathcal{L}_\mathcal{P}(\mathcal{X})$. In the case when the C$^*$-algebra comes from a discrete group $G$ acting on a locally finite building, we show that $\mathscr{C}=\mathscr{C}(i)\cong C(\Omega)\rtimes_r G$ (see Theorem \ref{convcrossed}), where $\Omega$ is Klisse's bordification of the building $X=\coprod_{i\in \mathscr{I}}G/H_i$ seen as a graph. In this particular case, $\mathscr{C}(i)$ is independent of the choice of $i.$ It turns out (see Theorem \ref{capraceklisse}) that $\Omega$ also coincides with Caprace and Lécureux's combinatorial compactification of the building \cite{CL11}. 

We introduce a universal counterpart of our covariance C$^*$-algebra denoted by $\mathscr{C}^{\max}$. It is the universal C$^*$-algebra generated by the class of so-called covariant representations of the decomposition. Our main result (see Theorem \ref{bigbig}) is that if $S$ is finite, the decomposition is of finite type (see Proposition \ref{finito}) and for all $i\in \mathscr{I}$ and $s\in S$, the left action of $B_i$ on $\mathcal{X}_i(s)$ is faithful, then the canonical surjection $\mathscr{C}^{\max} \twoheadrightarrow \mathscr{C}$ induces an isomorphism $\mathscr{C}^{\max}(i) \cong\mathscr{C}(i)$ between the corresponding C$^*$-subalgebras for every $i \in \mathscr{I}$. The proof, carried out essentially in section 4, goes as follows: \begin{itemize} \item Decompose the C$^*$-algebras $\mathscr{C}(i)$ and $\mathscr{C}^{\max}(i)$ as the closure of a sum of subspaces $\mathscr{C}_w(i)$ and $\mathscr{C}^{\max}_w(i)$ satisfying the axioms of a Fell bundle over $W$ (e.g. $\mathscr{C}_u(i) \mathscr{C}_v(i) \subset \mathscr{C}_{uv}(i)$ for all $u,v \in W$).
    \item Prove that if the decomposition is of finite type then the diagonal subalgebras $\mathcal{D}^{\max}(i)$ and $\mathcal{D}(i)$ are spanned by compact operators.
    \item Adapt Katsura's proof of the gauge invariance uniqueness theorem \cite{Kat04} to prove that the canonical map $\mathcal{D}^{\max}(i) \twoheadrightarrow \mathcal{D}(i)$ is injective.
    \item Identify the C$^*$-subalgebra $\Delta_0(i) \subset \mathcal{D}^{\max}(i)$ generated by the projections $(p_{i,w})_{w\in W}$ with that of continuous functions $C(\overline{(W,S)})$ on the bordification of the Cayley graph of $W$ in the sense of \cite{Kli23}. Then use the F\o{}lner sequence $a_{\lambda} : W \rightarrow C(\overline{(W,S)})$ built by \cite{DJ99, Kli23} to prove that the $W$-grading of $\mathscr{C}^{\max}(i)$ satisfies Exel's approximation property \cite{Exe97}.
\end{itemize}

Once this theorem is established and under some technical assumptions, we deduce the fact that $\mathscr{C}(i)$ satisfies any approximation property among nuclearity, exactness, LLP, WEP if and only if the Borel C$^*$-subalgebras $(B_j)_{j\in \mathscr{I}}$ all satisfy the same approximation property. In the case of groups, asking for the decomposition to be of finite type is actually equivalent to asking that the building $X = \coprod_{i\in \mathscr{I}}G/H_i$ is locally finite (see Proposition \ref{finitogroup}). As a corollary we get \begin{theorem*}[Corollary~\ref{exactG}]
Let $G$ be a discrete group acting on a locally finite
building. Assume that the stabilizers of the chambers are exact. Then $G$ is
exact.
\end{theorem*}
For comparison, \cite{Lec09} establishes exactness of a locally compact group acting properly on a not necessarily locally finite building. A crucial step of our proof of the nuclearity/WEP/LLP  of $C(\Omega)\rtimes_r G$ (see Theorem \ref{nucnucexex}  and Corollary \ref{wepllp}) is to use the amenability of the action of $W$ on the bordification of its Cayley graph $(W,S)$ established in \cite{Lec09,Kli23}. In order to prove the exactness of $G$ we only have to use the exactness result of \cite{DJ99}.

The paper is organized as follows. In section \ref{section2} we recall the axioms of a $W$-distance, and introduce $W$-decompositions of C$^*$-algebras. Afterward, we assume the existence of suitable conditional expectations and introduce expected $W$-decompositions of C$^*$-algebras. We then construct the associated Fock module and reduced C$^*$-algebra and give a sufficient condition for a conditional expectation to exist. In section \ref{section3} we recall Klisse's definition of the bordification of a rooted graph, then we define the covariance C$^*$-algebras associated to an expected $W$-decomposition. When the decomposition comes from a group acting on a building, we show that the covariance C$^*$-algebras are all isomorphic to the reduced crossed product of $G$ with the C$^*$-algebra of continuous functions on the minimal combinatorial compactification of the building. In section \ref{section4} we introduce the notion of covariant representation of an expected $W$-decomposition. We then carry out the proof of our main theorem, and use it to establish different approximation properties.

\section{Actions on buildings and Bruhat decompositions}\label{section2}

\subsection{\texorpdfstring{$W$}{W}-decompositions of a discrete group}

We refer the reader to \cite{BB05} for fundamental properties of Coxeter groups and to \cite{AB08} for the fundamentals of the theory of buildings. We fix a Coxeter system $(W,S)$, which we shall often refer to simply as $W$.
\begin{definition}
    A building of type $(W,S)$ is the data of a set $X$, whose elements are called chambers and a map $\delta : X\times X \rightarrow W$ (called a $W$-distance) such that 
    \begin{enumerate}[label=\textbf{(Bui\arabic*)}]
    \item For all $x,y \in X$ we have $\delta(x,y) = e $ if and only if $x=y$.
    \item Let $x,y \in X$, define $\delta(x,y) = w$. For all $z\in X$ such that $\delta(z,x) = s \in S$ we have $\delta(z,y) \in \{sw,w\}$. Moreover if $l(sw) > l(w)$ then $\delta(z,y) = sw$.
    \item Let $x,y \in X$, define $\delta(x,y) = w$. For all $s\in S$ there exists a $z\in X$ such that $\delta(z,x) = s$ and $\delta(z,y) = sw$.
    \end{enumerate}
\end{definition}
The building $(X,\delta)$ is called locally finite if for every chamber $x\in X$ and $s\in S$ the number of $y\in X$ such that $\delta(x,y) = s$ is finite. The following definition should be natural to readers familiar with buildings. To the best of our knowledge, this precise system of axioms does not appear in the existing literature. The closest related notion is that of a $W$-groupoid, introduced by Norledge in \cite{Nor17,Nor21}. Let $G$ be a discrete group. For any subsets $A,B$ of $G$, we write $A\uplus B$ for their union in the case when they are disjoint and we write $A\cdot B = \{ab |a\in A, b\in B\}$.
\begin{definition}
     A $W$-decomposition of $G$ is the data of a nonempty index set $\mathscr{I}$ together with a family $(C_{i,j}(w))_{i,j\in \mathscr{I},w\in W}$ of subsets of $G$ such that
    \begin{enumerate}[label=\textbf{(Wdec\arabic*)}]
    \item For any fixed $i,j\in \mathscr{I}$, the subsets $(C_{i,j}(w))_{w\in W}$ are pairwise disjoint and $G = \biguplus_{w\in W} C_{i,j}(w)$.
    \item For any $i\neq j\in \mathscr{I}$, $C_{i,j}(e) = \emptyset$, for all $i\in \mathscr{I}, C_{i,i}(e) = H_i$ is a subgroup of $G$, called the $i$-th Borel subgroup of $G$ and for all $i,j\in \mathscr{I},w\in W$ we have $H_i\cdot C_{i,j}(w)\cdot H_j = C_{i,j}(w)$.
    \item For all $i\in \mathscr{I}$ and $s\in S$, there exists a $j \in \mathscr{I}$ such that $C_{i,j}(s)$ is nonempty.
    \item For all $s\in S, w\in W, i,j,k\in \mathscr{I},$ we have $C_{i,j}(s)\cdot C_{j,k}(w) \subset C_{i,k}(sw)\uplus C_{i,k}(w)$. Moreover if $l(sw) > l(w)$ then $C_{i,k}(sw) = \bigcup_{j\in\mathscr{I}}C_{i,j}(s)\cdot C_{j,k}(w)$.    
    \item For all $i,j \in \mathscr{I},w\in W$, $C_{i,j}(w)^{-1} = C_{j,i}(w^{-1})$.
    \end{enumerate}
\end{definition}
A Bruhat decomposition of $G$ of type $W$ is a $W$-decomposition of $G$ whose index set $\mathscr{I}$ is a singleton. In other words it is the data of a partition $G = \biguplus_{w\in W} C(w)$ of $G$ by a family $(C(w))_{w\in W}$ of nonempty subsets such that
\begin{itemize}
\item $C(e) = H$ is a subgroup of $G$ and for all $w\in W$ we have $H\cdot C(w)\cdot H = C(w)$.
\item For all $s\in S$ and $w\in W$ we have $C(s)\cdot C(w) \subset C(sw)\uplus C(w)$. Moreover if $l(sw) > l(w)$ then $C(sw) = C(s)\cdot C(w)$.    \item For all $w\in W$, $C(w)^{-1} = C(w^{-1})$.
\end{itemize}
    
We will see that $W$-decompositions of $G$ correspond exactly to actions of $G$ on a building by $W$-isometries. Let $G$ be a discrete group. Assume that $G$ admits a $W$-decomposition $G = \biguplus_{w\in W} C_{i,j}(w)$ with index set $\mathscr{I}$. For every $g\in G$ and $i,j \in \mathscr{I}$ define $[g]_{i,j}$ to be the unique element of $W$ such that $g \in C_{i,j}([g]_{i,j})$. Set $X = \coprod_{i\in \mathscr{I}}G/H_i$ and define a map $\delta : X\times X \rightarrow W$ by $\delta(g_1H_i,g_2H_j) = [g_1^{-1}g_2]_{i,j}$. $G$ acts on $X$ by left translations and for all $x_1,x_2\in X, g\in G$ we have $\delta(gx_1,gx_2) = \delta(x_1,x_2)$. Assume that $G$ admits an action on a building $X$ with a $W$-distance $\delta$. Let $\mathscr{I} = G \backslash X$ be the quotient set. For all $i\in \mathscr{I}$, let $x_i \in X$ be any representative of the class $i$. For all $i,j \in \mathscr{I},w\in W$, set $C_{i,j}(w) = \{g\in G | \delta(x_i,g\cdot x_j) = w\}$.
\begin{theorem}\label{CTB}
If $G$ admits a $W$-decomposition then $X = \coprod_{i\in \mathscr{I}}G/H_i$ with the $W$-distance defined above is a building on which $G$ acts by $W$-isometries. Reciprocally if $G$ admits a $W$-isometric action on a building $X$ then the family of subsets $C_{i,j}(w)$ over the index set $\mathscr{I} = G \backslash X$ is a $W$-decomposition of $G$. Moreover these two constructions are inverses of each other up to conjugation of the subsets $C_{i,j}(w)$ by elements of $G$.
\end{theorem}
\begin{proof}
Assume that $G$ admits a $W$-decomposition $G = \biguplus_{w\in W} C_{i,j}(w)$ with index set $\mathscr{I}$. If $\delta(g_1H_i,g_2H_j) = e$ then $[g_1^{-1}g_2]_{i,j} = e$ hence $g_1^{-1}g_2 \in C_{i,j}(e)$ which implies $i = j$ and $g_1H_i = g_2H_j$ by (Wdec2). This proves (Bui1). Let $i,j,k \in \mathscr{I}$ and $g_1,g_2,g_3 \in G$ be such that $\delta(g_3H_k,g_1H_i) = s \in S$. Define $\delta(g_1H_i,g_2H_j) = w$. We have $g_3^{-1}g_1 \in C_{k,i}(s)$ and $g_1^{-1}g_2 \in C_{i,j}(w)$. By (Wdec4) we get $g_3^{-1}g_2 \in C_{k,j}(sw)\uplus C_{k,j}(w)$ thus $[g_3^{-1}g_2]_{k,j} = \delta(g_3H_k,g_2H_j) \in \{sw,w\}$. Moreover if $l(sw) > l(w)$ then $g_3^{-1}g_2 \in C_{k,j}(sw)$ thus  $\delta(g_3H_k,g_2H_j) = sw$. This proves (Bui2). Let $i,k \in \mathscr{I}$, let $g_1,g_2 \in G$. Define $\delta(g_1H_i,g_2H_k) = w$. Let $s\in S$. Assume that $l(sw) > l(w)$. Use (WDec3) to take a $g_3\in C_{i,j}(s)$ for some $j \in \mathscr{I}$. We have $\delta(g_1g_3H_j,g_1H_i) = [g_3^{-1}]_{j,i} = s$ by (Wdec5). Moreover, $\delta(g_1g_3H_j,g_2H_k) = [g_3^{-1}g_1^{-1}g_2]_{j,k} = sw$ by (Wdec4) and (Wdec5) because $l(sw) > l(w)$. Assume now that $w = s\cdot v$ for $l(v) = l(w) - 1$. By (Wdec4), there exists a $j\in \mathscr{I}$, a $g_3 \in C_{i,j}(s)$ and a $g_4\in C_{j,k}(v)$ such that $g_1^{-1}g_2 = g_3g_4$.  We have $\delta(g_1g_3H_j,g_1H_i) = [g_3^{-1}]_{j,i} = s$ by (Wdec5). We have $\delta(g_1g_3H_j,g_2H_k) = [g_3^{-1}g_1^{-1}g_2]_{j,k} = [g_4]_{j,k} = sw$. This proves (Bui3).

Let $(X,\delta)$ be a building of type $(W,S)$ on which $G$ acts by $W$-isometries. (Wdec1) is obvious. Define $H_i = \text{stab}_G(x_i)$ for all $i \in\mathscr{I}$. By (Bui1) we know that $C_{i,i}(e) = H_i$ is a subgroup of $G$. For all $i,j \in \mathscr{I}$, $w\in W, g\in C_{i,j}(w)$ and $h_1 \in H_i, h_2 \in H_j$ we have $\delta(x_i,h_1gh_2x_j) = \delta(x_i, h_1gx_j) = \delta(h_1^{-1}x_i,gx_j) = \delta(x_i,gx_j) = w$ as $G$ acts by $\delta$-isometries. This proves (Wdec2). Now identify $X = \biguplus_{i\in \mathscr{I}} Gx_i$ with $\coprod_{i\in \mathscr{I}}G/H_i$. (Wdec5) is a direct consequence of the formula $\delta(x,y)^{-1} = \delta(y,x)$ for all $x,y\in X$ (see \cite[Corollary 5.17]{AB08}). Let $i\in \mathscr{I}$ and $s\in S$. By (Bui3), there exists a $z \in X$ such that $\delta(x_i,z) = s$. This proves (Wdec3). Let $s\in S, w\in W$ and $i,j,k \in \mathscr{I}$. Let $g_1\in C_{i,j}(s)$ and $g_2 \in C_{j,k}(w)$. We have $\delta(g_1^{-1}x_i,x_j) = \delta(x_i, g_1x_j) = s$ and $\delta(x_j, g_2x_k) = w$ thus by (Bui2) we have $\delta(x_i,g_1g_2x_k) = \delta(g_1^{-1}x_i,g_2x_k) \in \{sw,w\}$ and $\delta(x_i,g_1g_2x_k) = sw$ if $l(sw) > l(w)$ by (Bui2). Assume that $l(sw) > l(w)$ and take $g_1\in C_{i,k}(sw)$. We have $\delta(x_i,g_1x_k) = sw$ thus by (Bui3) there exists a $j\in \mathscr{I}$ and a $g_2\in G$ such that $\delta(g_2x_j,x_i) = s$ and $\delta(g_2x_j,g_1x_k) = w$. In other words, we have $g_1 = g_2(g_2^{-1}g_1)$ with $g_2^{-1}g_1 \in C_{j,k}(w)$, $g_2^{-1}\in C_{j,i}(s)$ and hence $g_2\in C_{i,j}(s)$. This proves (Wdec4).
\end{proof}
From this theorem, it is easy to see that chamber-transitive actions of $G$ on a building coincide with Bruhat decompositions of $G$. In this case, the action of $G = \biguplus_{w\in W}C(w)$ on $X = G/H$ is Weyl-transitive in the sense of \cite[Chapter 6]{AB08} if and only if the double cosets $H \backslash C(w) / H$ are all singletons. Typical examples of discrete groups acting on buildings are lattices in reductive algebraic groups over non-Archimedean local fields \cite{BT72,BT84}. Notice that in these cases, the action is very rarely chamber-transitive \cite{KLT87}. 

\begin{example}
The expression $W = \biguplus_{w\in W}\{w\}$ is a Bruhat decomposition of type $W$. The associated building is just $W$ equipped with the $W$-distance defined by $\delta(u,v) = u^{-1}v$ for all $u,v\in W$.
\end{example}
\begin{example}
   Let $D_\infty=\langle s,t \mid s^2=t^2=e\rangle$ be the infinite dihedral group. Let $G_s$ and $G_t$ be two groups containing a common subgroup $H$. Let $G = G_s\star_H G_t$ be the associated amalgamated free product. Reduced words of $G$ are either the elements of $H$ or the elements of the form $g =g_1\ldots g_n$ with $g_i \in G_{s_i}\setminus{H}$, $s_i \in \{s,t\}$ so that for all $1\leq i < n$ we have $s_{i+1} \neq s_i$. In other words, reduced words of $G$ seen as an amalgamated free product are exactly the words whose indices in $S$ form a reduced word of the Coxeter group $W$. Set $C(e) = H$. For every $w\in D_\infty\setminus\{e\}$, denote by $C(w)$ the set of all reduced words $g$ of this form such that $s_1\ldots s_n = w$. Then $G = \biguplus_{w\in W} C(w)$ is a Bruhat decomposition of type $D_\infty$. The associated building has $X = G/H$ as a set of chambers, which coincides with the set of edges of the Bass-Serre tree of the amalgamated free product \cite{Ser77}. The preceding theorem generalizes the following classical result: splittings of a group $G$ as an amalgamated free product of subgroups correspond to actions of the group on trees $T$ that are without inversions and whose quotient graph $T/G$ is a single edge with two distinct vertices.
\end{example}
Let us give a few elementary properties of $W$-decompositions. Let $G = \biguplus_{w\in W} C_{i,j}(w)$ be a $W$-decomposition of $G$ with index set $\mathscr{I}$.
\begin{proposition}
    Let $s\in S$ and $w\in W$ be such that $l(sw) > l(w)$. Let $i,k \in \mathscr{I}$. The subsets $(C_{i,j}(s)C_{j,k}(w))_{j\in \mathscr{I}}$ are disjoint: $C_{i,k}(sw) = \biguplus_{j\in\mathscr{I}}C_{i,j}(s)\cdot C_{j,k}(w)$.
\end{proposition}
\begin{proof}
    Let $j, j'\in \mathscr{I}$. Let $g_1 \in C_{i,j}(s), g_2 \in C_{j,k}(w)$ and $g_1' \in C_{i,j'}(s), g_2' \in C_{j',k}(w)$ be such that $g_1g_2=g_1'g_2'$. We have $g_2 = (g_1^{-1}g_1')g_2'$ and $g_1^{-1}g_1' \in C_{j,j'}(s)\uplus C_{j,j'}(e)$. Assume that $j\neq j'$. (Wdec2) gives $g_1^{-1}g_1' \in C_{j,j'}(s)$ and hence $g_2 \in C_{j,j'}(s)\cdot C_{j',k}(w) = C_{j,k}(sw)$ as $l(sw) > l(w)$. But $g_2 \in C_{j,k}(w)$ so there is a contradiction.
\end{proof}
Hence for every $i,j \in \mathscr{I}$ and every reduced expression $w = s_1\ldots s_n$ of an element of $W$, we have \[C_{i,j}(w) = \biguplus_{i\in \mathscr{I}^{n-1}}C_{i,i_1}(s_1)\ldots C_{i_{n-1},j}(s_n).\]

Now we assume that $G$ actually has a Bruhat decomposition $G = \biguplus_{w\in W}C(w)$ of type $W$. For all $I\subset S,$ consider $G_I = \biguplus_{w\in W_I}C(w)$. One easily proves that $G_I$ is a subgroup of $G$ and that the family $(C(w))_{w\in W_I}$ forms a Bruhat decomposition of type $W_I$ of $G_I$. We call $G_I$ the parabolic subgroup of $G$ of type $I$. Notice that $G_\emptyset = H$. It turns out that $G$ has a universal property coming from its Bruhat decomposition.
\begin{proposition}\label{univgroup}
Let $\mathcal{R}_2 = \left\{ I\subseteq S \;\middle|\; |I|\leq 2 \text{ and } m_{s,t}<\infty \text{ whenever } I=\{s,t\} \right\}$ which is ordered by inclusion.
Let $G' = \bigstar_{I\in \mathcal{R}_2}G_I$ be the colimit (amalgamated sum) of the $G_I$ for $I\in \mathcal{R}_2$ together with the inclusions $G_I \subset G_J$. The canonical group homomorphism $\phi : G' \rightarrow G$ is an isomorphism.
\end{proposition}
\begin{proof}
For all $I\in \mathcal{R}_2,$ write $j_I : G_I \hookrightarrow G'$ for the canonical embedding. For all $s\in S$ let $C'(s) = j_{\{s\}}(C(s))$. Let $H' = j_\emptyset(H)$. First notice that for all $s,t \in S$, such that $m_{s,t} <\infty$, we have \[\begin{aligned} C'(s)C'(t)\ldots = j_{s,t}(C(s)C(t)\ldots) \\ = j_{s,t}(C(st\ldots))\\ =  j_{s,t}(C(ts\ldots)) \\ = j_{s,t}(C(t)C(s)\ldots) \\ = C'(t)C'(s)\ldots  \end{aligned}\]
where each product has $m_{s,t}$ terms. 

For all $n\geq 0$ we let $\mathcal{E}_n \subset G'$ be the subset containing $H'$ and arbitrary products of at most $n$ elements of some $C'(s)$ for $s\in S$. We now prove by induction on $n$ that $\ker\phi\cap \mathcal{E}_n = \{e\}$. The result is true for $n=0,1$ by definition of $G'$. Let $n\geq 1$ be such that the result holds. Let $g\in \ker \phi \cap \mathcal{E}_{n+1}$. There exist $s_1,\ldots, s_{n+1} \in S$ and $g_i \in C'(s_i)$ for $1\leq i \leq n+1$ such that $g = g_1\ldots g_{n+1}$. We have $\phi(g_1)\ldots \phi(g_{n+1}) = e$ with $\phi(g_i) \in C(s_i)$ for all $1 \leq i\leq n+1$. The expression $s_1\ldots s_{n+1}$ cannot be a reduced word of $W$. Hence, by Matsumoto's theorem \cite[Theorem 3.3.1]{BB05}, there is a sequence of braid moves in $W$ which leads to a situation of the form $s_i = s_{i+1}$. By the last remark, this implies that $g \in \mathcal{E}_n$.
\end{proof}
\begin{remark}\label{groupoids}
    It seems that there is no way of establishing a similar universal property in the more general case of a group $G$ with an arbitrary $W$-decomposition. In fact, it seems that in this case there is no satisfying notion of parabolic subgroups besides the Borel subgroups $H_i$ for $i\in \mathscr{I}$. These problems disappear if one uses Norledge's formalism of $W$-groupoids \cite{Nor17} instead of $W$-decompositions. A $W$-groupoid is essentially just the data of a discrete groupoid $\mathcal{G}$ with object set $\mathcal{G}^{(0)} = \mathscr{I}$ together with partitions of each Hom-set $\mathcal{G}(i,j) = \biguplus_{w\in W}C_{i,j}(w)$ satisfying axioms similar to those of a $W$-decomposition. For every $I\subset S$ the parabolic groupoid of type $I$ is given by the same object set as $\mathcal{G}$ and hom-sets given by $\mathcal{G}_I(i,j) = \biguplus_{w\in W_I}C_{i,j}(w)$ for all $i,j\in \mathscr{I}$. The last proof can be directly adapted to express $\mathcal{G}$ as the colimit of the family of parabolics $(\mathcal{G}_I)_{I\in \mathcal{R}_2}$ in the category of groupoids. As there is a naturally associated reduced and a maximal C$^*$-category to each groupoid, this article could be entirely formulated using the formalism of $W$-groupoids and C$^*$-categories.
\end{remark}

\subsection{\texorpdfstring{$W$}{W}-decompositions of a \texorpdfstring{C$^*$-algebra}{C*-algebra}}
We fix a Coxeter system $(W,S)$ and a unital C$^*$-algebra $A$. 
For every $X,Y \subset A$ write $X\cdot Y = \operatorname{span}\{xy |x\in X,y\in Y\}$. 
The following is the C$^*$-algebraic analogue of the notion of $W$-decomposition of a discrete group.
\begin{definition}\label{OB}
     A $W$-decomposition of $A$ is the data of a nonempty index set $\mathscr{I}$ together with a family $(L_{i,j}(w))_{i,j\in \mathscr{I},w\in W}$ of subspaces of $A$ such that
    \begin{enumerate}[label=\textbf{(OWD\arabic*)}]
    \item For any fixed $i,j\in \mathscr{I}$, we have $A = \overline{\sum_{w\in W} L_{i,j}(w)}$.
    \item For any $i\neq j\in \mathscr{I}$, $L_{i,j}(e) = 0$. Moreover, for all $i\in \mathscr{I},$ $L_{i,i}(e)$ is a $*$-subalgebra of $A$ with the same unit such that for all $i,j\in \mathscr{I},w\in W$ we have $L_{i,i}(e)\cdot L_{i,j}(w)\cdot L_{j,j}(e) = L_{i,j}(w)$.
    \item For all $s\in S, w\in W, i,j,k\in \mathscr{I},$ we have $L_{i,j}(s)\cdot L_{j,k}(w) \subset L_{i,k}(sw) + L_{i,k}(w)$. Moreover if $l(sw) > l(w)$ then $L_{i,k}(sw) = \sum_{j\in\mathscr{I}}L_{i,j}(s)\cdot L_{j,k}(w)$.    
    \item For all $i,j \in \mathscr{I}, w\in W$, $L_{i,j}(w)^{*} = L_{j,i}(w^{-1})$.
    \end{enumerate}
For all $i\in \mathscr{I}$ we call $B_i = \overline{L_{i,i}(e)}$ the $i$-th Borel C$^*$-subalgebra of $A$. A Bruhat decomposition of $A$ of type $W$ is a $W$-decomposition of $A$ whose index set is a singleton.
\end{definition}
A direct consequence of (OWD3) is that for every $i,j \in \mathscr{I}$ and every reduced expression $w = s_1\ldots s_n$ of an element of $W$, we have \[L_{i,j}(w) = \sum_{i\in \mathscr{I}^{n-1}}L_{i,i_1}(s_1)\ldots L_{i_{n-1},j}(s_n).\]
\begin{example}\label{Bruhathecke}
    Let $q = (q_s)_{s\in S}$ be a family of positive real numbers indexed by $S$ such that if $s,t\in S$ are conjugate then $q_s = q_t$. Let $C^*_{r,q}(W)$ be the C$^*$-subalgebra of $B(l^2W)$ generated by the operators $T_s$, $s\in S$ defined on the canonical basis $(\delta_w)_{w\in W}$ of $l^2W$ by \[T_s(\delta_w) = \begin{cases}
\delta_{sw} & \text{if } l(sw) > l(w), \\
\delta_{sw} + p_s(q) \delta_w & \text{if } l(sw) < l(w),
\end{cases}\]
where $p_s(q) = q_s^{-1/2}(q_s-1)$. For all $w\in W$ and reduced expression $w = s_1 \ldots s_n$ let $T_w = T_{s_1}\ldots T_{s_n}$. One can prove that $T_w \in B(l^2W)$ is a well-defined operator \cite[Theorem 3.1.1]{Kli22}. Then $C^*_{r,q}(W) = \overline{\bigoplus_{w\in W}\mathbb{C}\cdot T_w}$ is a Bruhat decomposition of type $W$ of $C^*_{r,q}(W)$ of Borel C$^*$-algebra $\mathbb{C}$. These C$^*$-algebras have been extensively studied in \cite{Kli22, Sol07}.
\end{example}
We now want to define a suitable notion of a reduced C$^*$-algebra associated to a $W$-decomposition.
\begin{definition}
    An expected $W$-decomposition of $A$ is the data of a $W$-decomposition $A = \overline{\sum_{w\in W}L_{i,j}(w)}$ with index set $\mathscr{I}$ together with a family of conditional expectations $(E_i : A \rightarrow B_i)_{i\in \mathscr{I}}$ such that for all $i\in \mathscr{I}$ and $w \in W\setminus \{e\}$ we have $E_i(L_{i,i}(w)) = 0$.
\end{definition}
Notice that if such a family of conditional expectations exists, then it is unique. We fix an expected $W$-decomposition of $A$ with such conditional expectations $E_i : A \rightarrow B_i$. The following is the key proposition enabling us to do computations with the $E_i$.
\begin{proposition}\label{calcul}
    Let $i,j \in \mathscr{I}$. Let $u,v \in W$. Let $x \in L_{i,j}(u)$ and $y \in L_{i,j}(v)$. If $u\neq v$ then $E_{j}(x^*y) = 0$. Assume $u = v\neq e$ and take $u = s_1\ldots s_n$ a reduced expression. Assume that $x$ and $y$ are of the form $x = x_1\ldots x_n$ and $y = y_1\ldots y_n$ with $i_0,\ldots i_n ,j_0\ldots j_n \in \mathscr{I}$ such that $i_0 = j_0 = i$, $i_n=j_n=j$ and for all $1\leq k\leq n$, $x_k\in L_{i_{k-1},i_{k}}(s_k)$ and $y_k\in L_{j_{k-1},j_{k}}(s_k)$. If there exists a $1\leq k\leq n$ such that $i_k \neq j_k$ then $E_j(x^*y) = 0$, otherwise we have \[E_j(x^*y) = E_{i_n}(x_n^*(E_{i_{n-1}}(x_{n-1}^*\ldots E_{i_1}(x_1^*y_1)\ldots y_{n-1})y_n).\]
\end{proposition}
\begin{proof}
(OWD3) allows us to decompose $x^*y$ as a finite linear combination of elements of some $L_{j,j}(w)$ for some $w\in W$ such that $l(w) \geq |l(u)-l(v)|$. Hence, if $l(u) \neq l(v)$ then $E_j(x^*y)=0$ and we will treat the case when $l(u) = l(v)$. Assume that there is an $s\in S$ such that $l(su) < l(u)$ but $l(sv)> l(v)$. For all $k\in \mathscr{I}$, $a \in L_{i,k}(s)$ and $x'\in L_{k,j}(su)$, we have $E_j((ax')^*y) = E_j(x'^*(a^*y))$ but as $l(sv)> l(v) > l(su)$, we have $E_j((ax')^*y) = 0$. By (OWD3) this implies $E_j(x^*y) = 0$ and we will treat the case when $u$ and $v$ begin with the same set of elements of $S$.

Now we proceed to prove the proposition by induction on $l(u) = l(v)$. The result is obvious for $l(u)=l(v)=1$. Let $n\geq 1$. Assume the proposition holds at $n-1$. Let $s \in S$ be a common prefix of $u$ and $v$. Assume that $x = ax', y = by'$ for some $k_1,k_2\in \mathscr{I}$, $a \in L_{i,k_1}(s), x'\in L_{k_1,j}(su)$ and $b \in L_{i,k_2}(s), y'\in L_{k_2,j}(sv)$. We have $E_j(x^*y) = E_j(x'^*a^*by')$. Moreover $a^*b \in L_{k_1,k_2}(s)+L_{k_1,k_2}(e)$. If $k_1\neq k_2$ then (OWD2) gives $L_{k_1,k_2}(e) = 0$ hence $a^*b \in L_{k_1,k_2}(s)$ and $E_j(x'^*(a^*by')) = 0$ as $a^*by' \in L_{k_1,j}(v)$ and $l(v) > l(su)$. Now we assume that $k_1 = k_2 =k$. Write \[E_j(x^*y) = E_j(x'^*a^*by') = E_j(x'^*E_k(a^*b)y')+E_j(x'^*(a^*b-E_k(a^*b))y').\] We know that $(a^*b-E_k(a^*b))y' \in L_{k,j}(v)$, as $l(v) > l(su)$ we have $E_j(x'^*(a^*b-E_k(a^*b))y') = 0$. This finishes the proof as we can now apply the induction hypothesis to compute $E_j(x^*y) = E_j(x'^*z)$ where $z = E_k(a^*b)y'\in L_{k,j}(sv)$.
\end{proof}
For all $i\in \mathscr{I}$, the $i$-th Fock module is defined as the right $B_i$ Hilbert module obtained by GNS-construction $X_i = L^2(A,E_i)$. We denote by $\eta_i$ its cyclic vector so that $X_i = \overline{A\eta_iB_i}$. The following proposition is a direct consequence of the last one.
\begin{proposition}\label{fock}
    For all $i,j \in \mathscr{I}$ and $w\in W$, we define $X_{i,j}(w) = \overline{L_{i,j}(w)\cdot \eta_j}\subset  X_j$. Each $X_{i,j}(w)$ is a right Hilbert $B_j$-submodule of $X_j$. For arbitrary $i,j \in \mathscr{I}$, the right $B_j$-module $X_j$ decomposes as a direct sum \[X_j =  \bigoplus_{w\in W}X_{i,j}(w).\]
    Moreover for every reduced decomposition $w= s_1\ldots s_n,$ there is a right $B_j$-linear unitary isomorphism \[X_{i,j}(w) \cong \bigoplus_{(i_1 \ldots i_{n-1}) \in \mathscr{I}^{n-1}}X_{i,i_1}(s_1)\otimes_{B_{i_1}}\ldots \otimes_{B_{i_{n-1}}}X_{i_{n-1},j}(s_n),\] implemented by \[a_1\eta_{i_1}\otimes a_2\eta_{i_2} \ldots \otimes  a_n \eta_{j} \mapsto a_1\ldots a_n \eta_j.\]
\end{proposition}
We now introduce the Fock module on which the reduced C$^*$-algebra will be defined and give two elementary lemmas which will be useful many times in the rest of this article. Let $\mathcal{P} = \prod_{i\in \mathscr{I}}B_i$. For all $i\in \mathscr{I}$, write $1_i\in \mathcal{P}$ for the vector whose $i$-th coordinate is $1$ and all other are zero. Embed each $B_i$ into $\mathcal{P}$ by the map $b \mapsto b\cdot1_i$. We define the following right $\mathcal{P}$-module \[
\mathcal{X}= \bigoplus_{i\in \mathscr{I}}X_i\otimes_{B_i}\mathcal{P}.
\]
\begin{lemma}\label{ZERO}
    Let $C,D$ be two unital C$^*$-algebras. Let $H$ be a right Hilbert $C$-module and $K$ be a $C$-$D$ correspondence. Assume that the left action $C \rightarrow \mathcal{L}_{D}(K)$ is faithful. $\mathcal{L}_C(H)$ embeds into $\mathcal{L}_D(H\otimes_CK)$ by $f \mapsto f\otimes \mathrm{id}_K$. Moreover, for all $f\in \mathcal{L}_C(H)$, if $f\otimes \mathrm{id}_K \in \mathcal{L}_D(H\otimes_C K)$ is compact then so is $f$.
\end{lemma}
\begin{proof}
    Denote by $\iota : \mathcal{L}_C(H) \rightarrow \mathcal{L}_D(H\otimes_CK)$ the $*$-homomorphism given by $\iota(f) = f\otimes \mathrm{id}_K.$ The injectivity of $\iota$ is easy. We assume that $f\in \mathcal{L}_C(H)$ is such that $\iota(f)$ is compact. Let $(e_\lambda)_{\lambda\in I}$ be an approximate identity for $\mathcal{K}_C(H)$. As $e_\lambda x \rightarrow x $ in norm for all $x\in H$, we know that, if $f\otimes \mathrm{id}_K$ is compact, then $e_\lambda f \otimes \mathrm{id}_K$ converges in norm to $f\otimes \mathrm{id}_K$. The injectivity of $\iota$ thus implies that $e_\lambda f$ converges in norm to $f$.
\end{proof}
\begin{lemma}\label{box}
There is a $*$-isomorphism $\mathcal{L}_{\mathcal{P}}(\mathcal{X}) \cong \prod_{i\in \mathscr{I}}\mathcal{L}_{B_i}(X_i)$ which sends every $f\in \mathcal{L}_{B_i}(X_i)$ to $f\otimes \mathrm{id}_\mathcal{P}$. Moreover this isomorphism induces at the level of compact operators the isomorphism $\mathcal{K}_{\mathcal{P}}(\mathcal{X}) \cong \bigoplus^{c_0}_{i\in \mathscr{I}}\mathcal{K}_{B_i}(X_i).$
\end{lemma}
\begin{proof}
The last lemma gives the injectivity of the $*$-homomorphism $\prod_{i\in \mathscr{I}}\mathcal{L}_{B_i}(X_i) \rightarrow \mathcal{L}_{\mathcal{P}}(\mathcal{X})$. Notice that for all $i,j \in \mathscr{I}$ and $b\in B_i$, we have $b\otimes 1_j = b1_i\otimes 1_j = \delta_{i,j}b\otimes 1_\mathcal{P}.$ Hence any $f\in \mathcal{L}_\mathcal{P}(\mathcal{X})$ satisfies $f(X_i\otimes_{B_i}\mathcal{P}) \subset X_i\otimes_{B_i}\mathcal{P}$ and its restriction to $X_i\otimes_{B_i}\mathcal{P}$ is just given by an element of $\mathcal{L}_{B_i}(X_i)$. The second part of the last lemma gives the second identification.
\end{proof}
As $A$ acts on each $X_i$ on the left by adjointable operators, $\mathcal{X}$ is actually an $A$-$\mathcal{P}$ C$^*$-correspondence with a left action we denote by $\lambda : A \rightarrow \mathcal{L}_\mathcal{P}(X)$.
\begin{definition}\label{red}
    The reduced C$^*$-algebra $A^r \subset \mathcal{L}_{\mathcal{P}}(\mathcal{X})$ associated to the expected $W$-decomposition $A = \overline{\sum_{w\in W}L_{i,j}(w)}$ is the image of the $*$-homomorphism $\lambda$.
\end{definition}
For all $i,j \in \mathscr{I}$ and $w\in W$, we write $F_{i,j}(w) = \lambda(L_{i,j}(w)) \subset A^r$. The following proposition is easy to prove.
\begin{proposition}
$A^r = \overline{\sum_{w\in W} F_{i,j}(w)}$ is a $W$-decomposition of $A^r$ with index set $\mathscr{I}$. For all $i\in \mathscr{I}$, the Borel C$^*$-subalgebra $\overline{F_{i,i}(e)}$ of $A^r$ is $*$-isomorphic to $B_i$.     
\end{proposition}
\begin{remark}
    The decomposition $A = \overline{\sum_{w\in W} L_{i,j}(w)}$ will be of little use in practice. It is only used in this first section to introduce the formal properties of $W$-decompositions. The decomposition of the reduced C$^*$-algebra $A^r = \overline{\sum_{w\in W} F_{i,j}(w)}$ is the one we shall focus on in the next sections.
\end{remark}
\begin{example}\label{GOP}
    Let $G = \biguplus_{w\in W}C_{i,j}(w)$ be a discrete group $G$ with a $W$-decomposition of index set $\mathscr{I}$ and Borel subgroups $H_i = C_{i,i}(e)$. Let $A = C^*_r(G)$ be the reduced C$^*$-algebra of $G$, which is generated by the unitaries $\lambda(g)\in B(l^2G)$ implementing the action of $G$ on itself by left translation. For all $i,j \in \mathscr{I}$ and $w\in W$ let $F_{i,j}(w) = \operatorname{span}\{\lambda(g)|g\in C_{i,j}(w)\} \subset A$. For all $i\in \mathscr{I}$, let $B_i = \overline{F_{i,i}(e)} = C^*_r(H_i)$. $A = \overline{\bigoplus_{w\in W}F_{i,j}(w)}$ is a $W$-decomposition of index set $\mathscr{I}$ of $A$. Let $E_i : C^*_r(G) \rightarrow C^*_r(H_i)$ be the usual conditional expectations. This $W$-decomposition together with the family of conditional expectations $(E_i)_{i\in \mathscr{I}}$ form an expected $W$-decomposition of $C^*_r(G)$. As each $E_i$ is faithful the reduced C$^*$-algebra of this decomposition is isomorphic to $C^*_r(G)$.
    \end{example}
For all $i\in \mathscr{I}$ and $w\in W$ we define \[\mathcal{X}_{i}(w) = \bigoplus_{j\in \mathscr{I}}X_{i,j}(w)\otimes_{B_j}\mathcal{P},\] so that we have $\mathcal{X} = \bigoplus_{w\in W} \mathcal{X}_i(w)$ for all $i\in \mathscr{I}$.
\begin{lemma}\label{xxxxxxx}
For all $i,k \in \mathscr{I}$ and $u,v \in W$ such that $l(uv) =l(u)+l(v)$, there are unitary isomorphisms of right Hilbert $B_k$-modules and right Hilbert $\mathcal{P}$-modules \[X_{i,k}(uv) \cong \bigoplus_{j\in \mathscr{I}} X_{i,j}(u)\otimes_{B_j}X_{j,k}(v) \text{ and } \mathcal{X}_i(uv) \cong \bigoplus_{j\in \mathscr{I}}X_{i,j}(u)\otimes_{B_j}\mathcal{X}_j(v).\]
\end{lemma}
\begin{proof}
    The first assertion follows directly from Proposition \ref{fock}. Let us prove the second one. Let $u,v\in W$ be such that $l(uv) = l(u) + l(v)$. Compute \[
\begin{aligned}
\mathcal{X}_i(uv)
&=
\bigoplus_{j\in\mathscr{I}}
X_{i,j}(uv)\otimes_{B_{j}}\mathcal{P}
\\
&\cong
\bigoplus_{j,j'\in\mathscr{I}} X_{i,j'}(u)
\otimes_{B_{j'}}
X_{j',j}(v)\otimes_{B_{j}}\mathcal{P}
\\
&\cong
\bigoplus_{j'\in\mathscr{I}}
X_{i,j'}(u)\otimes_{B_{j'}}
\mathcal{X}_{j'}(v).
\end{aligned}
\]
\end{proof}
Recall that a right Hilbert module $\mathcal{H}$ over a unital C$^*$-algebra $\mathcal{A}$ is called finitely generated if it is finitely generated as a right $\mathcal{A}$-module and that this condition is equivalent to asking that $id_\mathcal{H} \in \mathcal{K}_\mathcal{A}(\mathcal{H})$.
\begin{proposition}\label{finito}
Let $A = \overline{\sum_{w\in W}L_{i,j}(w)}$ be an expected $W$-decomposition with conditional expectations $E_i : A \rightarrow B_i$ and index set $\mathscr{I}$. The following three conditions are equivalent \begin{enumerate}[label=(\roman*)]
    \item For all $i \in \mathscr{I}$ and $s\in S$, the right $\mathcal{P}$-module $\mathcal{X}_i(s)$ is finitely generated.
    \item For all $i \in \mathscr{I}$ and $w\in W$, the right $\mathcal{P}$-module $\mathcal{X}_i(w)$ is finitely generated.
    \item For all $i \in \mathscr{I}$ and $s\in S$, the right $B_j$-modules $X_{i,j}(s)$ for $j\in \mathscr{I}$ are finitely generated and only a finite number of them are nonzero. 
\end{enumerate}
\end{proposition}
\begin{proof}
The equivalence of (i) and (iii) follows from the formula given by Lemma \ref{box}: \[\mathcal{K}_{\mathcal{P}}(\mathcal{X}_i(s)) \cong \bigoplus^{c_0}_{j\in \mathscr{I}}\mathcal{K}_{B_j}(X_{i,j}(s)),\] as a $c_0$-direct sum of C$^*$-algebras can only be unital if finitely many of the C$^*$-algebras are nonzero. Assume (iii) holds. Let $w =s_1 \ldots s_n$ be a reduced expression of an arbitrary element of $W$. Write \[\mathcal{X}_{i}(w) = \bigoplus_{j\in \mathscr{I}}X_{i,j}(w)\otimes_{B_j}\mathcal{P} \cong \bigoplus_{(i_1 \ldots i_{n}) \in \mathscr{I}^{n}}X_{i,i_1}(s_1)\otimes_{B_{i_1}}\ldots \otimes_{B_{i_{n-1}}}X_{i_{n-1},i_n}(s_n)\otimes_{B_{i_n}}\mathcal{P}.\] As for all $i \in \mathscr{I}$ and $s\in S$, only a finite number of the $X_{i,j}(s)$ are nonzero, this sum is actually finite. As each of the $X_{i_{k-1},i_{k}}(s_{k})$ is finitely generated as a right $B_{i_k}$-module, this proves that $\mathcal{X}_i(w)$ is a finitely generated right $\mathcal{P}$-module.
\end{proof}

Such an expected $W$-decomposition is said to be of finite type if any of the above conditions is verified. This notion should be understood as the noncommutative analog of the local finiteness of the underlying building, as the following proposition shows.
\begin{proposition}\label{finitogroup}
    Let $G = \biguplus_{w\in W}C_{i,j}(w)$ be a discrete group $G$ with a $W$-decomposition of index set $\mathscr{I}$ and Borel subgroups $H_i = C_{i,i}(e)$. Let $A = C^*_r(G) = \overline{\bigoplus_{w\in W}F_{i,j}(w)}$ be the associated expected $W$-decomposition reduced C$^*$-algebra of $G$ as in Example \ref{GOP}. This decomposition is of finite type if and only if the underlying building is locally finite.
\end{proposition}
\begin{proof}
As the underlying building identifies with $\coprod_{i\in \mathscr{I}}G/H_i$, it is locally finite if and only if for all $i \in \mathscr{I}$ and $s\in S$, the disjoint union $\coprod_{j\in\mathscr{I}} C_{i,j}(s)/H_j$ is finite. The result follows from the last proposition and the identification $X_{i,j}(s) \cong l^2(C_{i,j}(s)/H_j,B_j)$ as right $B_j$-modules. \end{proof}
Let $\mathcal{A}$ and $\mathcal{B}$ be two C$^*$-algebras. We will denote by $\mathcal{A}\otimes\mathcal{B}$ their minimal tensor product. Given a right Hilbert $\mathcal{A}$-module $H$ and a right Hilbert $\mathcal{B}$-module $K$, their exterior tensor product is the right Hilbert $\mathcal{A}\otimes \mathcal{B}$-module $H\otimes_{\mathrm{ext}}K$ given by separation and completion of their algebraic tensor product over $\mathbb{C}$ with respect to the inner product $\langle h_1\otimes k_1,h_2\otimes k_2\rangle = \langle h_1, h_2 \rangle \otimes \langle k_1, k_2 \rangle$.
\begin{proposition}\label{tensor}
    Let $A = \overline{\sum_{w\in W}L_{i,j}(w)}$, together with conditional expectations $E_i : A \rightarrow B_i$ be an expected $W$-decomposition of a unital C$^*$-algebra $A$ with index set $\mathscr{I}$. Let $Q$ be a unital C$^*$-algebra. For all $w\in W$, define $G_{i,j}(w)$ to be the linear span of elements of the form $x\otimes q$ for $x\in L_{i,j}(w)$ and $q\in Q$ inside the minimal tensor product $A\otimes Q$. Then
    \[
        A \otimes Q = \overline{\sum_{w\in W}G_{i,j}(w)},
    \]
    together with conditional expectations $E_i\otimes \mathrm{id}_Q :
        A\otimes Q \rightarrow B_i\otimes Q$
    is an expected $W$-decomposition of the unital C$^*$-algebra $A\otimes Q$, with Borel subalgebras $B_i\otimes Q$ and index set $\mathscr{I}$. Moreover, the corresponding Fock modules
    \[
        Y_{j} = \bigoplus_{w\in W}Y_{i,j}(w)
    \]
    are given by the exterior tensor products $Y_j \cong X_j \otimes_{\mathrm{ext}}Q$ and $Y_{i,j}(w) \cong X_{i,j}(w)\otimes_{\mathrm{ext}}Q.$
\end{proposition}
\begin{proof}
    One easily checks the axioms of an expected $W$-decomposition. Let $Y_{j} = L^2(A\otimes Q,E_j\otimes \mathrm{id}_Q).$ Denote by $\eta_j$ the cyclic vector for $X_j = L^2(A,E_j)$ and by $\xi_j$ the cyclic vector for $Y_j = L^2(A\otimes Q,E_j\otimes\mathrm{id}_Q)$ . Define the operator $V : X_j\otimes_{\mathrm{ext}}Q \rightarrow Y_j$ by $V(a\cdot\eta_j \otimes q) = (a\otimes q)\cdot \xi_j$ for all $a\in A$ and $q\in Q$. The operator $V$ is well-defined and isometric as for all $a_1,\ldots a_n\in A$ and $q_1,\ldots q_n \in Q$ we have  \begin{align*}
&(E_j \otimes \mathrm{id}_Q)
\left(
\left(\sum_{k=1}^n a_k \otimes q_k\right)^*
\left(\sum_{k=1}^n a_k \otimes q_k\right)
\right) \\
&\qquad =
(E_j \otimes \mathrm{id}_Q)
\left(
\sum_{k,l=1}^n a_k^*a_l \otimes q_k^*q_l
\right) \\
&\qquad =
\sum_{k,l=1}^n E_j(a_k^*a_l) \otimes q_k^*q_l.
\end{align*}
Moreover, $V$ is clearly left $A$-linear, right $B_j$-linear, adjointable and surjective.
\end{proof}
Notice that if the decomposition of $A$ is of finite type then that of $A\otimes Q$ is also of finite type.
\subsection{Gluing conditional expectations}
Let $A = \overline{\sum_{w\in W}L(w)}$ be a Bruhat decomposition of type $W$ of a unital C$^*$-algebra. We write the Borel C$^*$-subalgebra $B = \overline{L(e)}$. In this section, we will see that one can define a maximal C$^*$-algebra $A^{\max}$ associated to this decomposition. Moreover this maximal C$^*$-algebra admits a natural Bruhat decomposition, and a conditional expectation for this decomposition can be built from smaller conditional expectations of the form $E_{I} : A^{\max}_{I} \rightarrow B$ (see Theorem \ref{VV}).
\begin{definition} For all $I\subset S$, the parabolic C$^*$-subalgebra of type $I$ of $A$ is defined by $A_I = \overline{\sum_{w\in W_I}L(w)}$. \end{definition}
It is easy to check that $A_I = \overline{\sum_{w\in W_I}L(w)}$ is a Bruhat decomposition of $A_I$ of type $W_I$. Inspired by Proposition \ref{univgroup}, we let the universal C$^*$-algebra of the Bruhat decomposition $A^{\max}$ be the amalgamated sum (the colimit in the category of unital C$^*$-algebras) of the C$^*$-algebras $A_I$ for $I\in \mathcal{R}_2$ together with the inclusions $A_I \subset A_J$ whenever $I\subset J \in \mathcal{R}_2$.

In other words, $A^{\max}$ is the universal C$^*$-algebra generated by families of $*$-homomorphisms $\alpha_I : A_I \rightarrow B(H)$ for all $I \in \mathcal{R}_2$, where $H$ is any Hilbert space, satisfying $\alpha_J\circ i_{I\subset J} = \alpha_I$ whenever $I\subset J \in \mathcal{R}_2$ (here $i_{I\subset J}$ just denotes the inclusion map $A_I \hookrightarrow A_J$). By definition of $A^{\max}$ there are canonical $*$-homomorphisms $j_I : A_I \rightarrow A^{\max}$ and $\alpha : A^{\max} \rightarrow A$ such that $\alpha \circ j_I = \mathrm{id}_{A_I}$ for all $I\in \mathcal{R}_2$. Write $j_{s,t} = j_{\{s,t\}}$ for all $s,t \in S$ with $m_{s,t} < \infty$ and $j = j_\emptyset : B \rightarrow A^{\max}$. \begin{proposition} For all $s\in S$ define $U(s) = j_s(L(s))$. For all $w\in W$ and reduced expression $w = s_1\ldots s_n$ define \[ U(w) = U(s_1)\ldots U(s_n). \] For all $w\in W$, $U(w)$ is a well-defined subspace of $A^{\max}$. Moreover, \[ A^{\max} = \overline{\sum_{w\in W}U(w)} \] is a Bruhat decomposition of $A^{\max}$ of type $W$ such that for all $I\in \mathcal{R}_2$ we have $A_I^{\max} \cong A_I$. \end{proposition} \begin{proof} We first prove that the subspaces $U(w)$ are well-defined. By Matsumoto's theorem \cite[Theorem 3.3.1]{BB05}, it suffices to prove that \[ \underbrace{U(s)U(t)\ldots}_{m_{s,t} \text{ terms}} = \underbrace{U(t)U(s)\ldots}_{m_{s,t} \text{ terms}} \] for every $s,t\in S$ with $m_{s,t} < \infty$. Compute \[ \begin{aligned} U(s)U(t)\ldots &= j_{s,t}(L(s))j_{s,t}(L(t))\ldots \\ &= j_{s,t}(L(s)L(t)\ldots) \\ &= j_{s,t}(L(st\ldots)) \\ &= j_{s,t}(L(ts\ldots)) \\ &= j_{s,t}(L(t)L(s)\ldots) \\ &= U(t)U(s)\ldots . \end{aligned} \] The C$^*$-algebra $A^{\max}$ is the closure of the span of products of elements of the subalgebras $j_I(A_I)$ for $I\in \mathcal{R}_2$. This gives the decomposition $A^{\max} = \overline{\sum_{w\in W}U(w)}.$ Moreover, $\overline{U(e)} = \overline{j(L(e))} \cong B$ is a C$^*$-subalgebra of $A^{\max}$ with the same unit and, for all $w\in W$, \[ j(L(e))\cdot U(w)\cdot j(L(e)) = j_{s_1}(L(e)\cdot L(s_1))\ldots j_{s_n}(L(s_n)\cdot L(e)) = U(w). \] For all $s\in S$ we have \[ U(s)U(w) = U(s)U(s_1)\ldots U(s_n). \] If $l(sw) > l(w)$, then $U(s)U(w) = U(sw)$ by definition. Otherwise, we can assume that $s_1 = s$ and \[ U(s)U(w) = U(s)U(s)\ldots U(s_n). \] We have \[ U(s)U(s) = j_s(L(s))j_s(L(s)) = j_s(L(s)L(s)) \subset j_s(L(s)+L(e)) = U(s)+U(e).\] This proves (OWD3). For all $I\in \mathcal{R}_2$, $A_I$ is isomorphic to $j_I(A_I)$ as a C$^*$-algebra, with inverse given by the restriction of $\alpha$ to $j_I(A_I)$. Compute \[ j_I(A_I) = j_I\left(\overline{\sum_{w\in W_I}L(w)}\right) = \overline{\sum_{w\in W_I}U(w)} = A_I^{\max}. \] \end{proof}
As $L(w) = L(s_1)\ldots L(s_n)$ for every $w\neq e$ in $W$ and every reduced expression $w = s_1\ldots s_n$, we also know that $\alpha : A^{\max} \rightarrow A$ is surjective. For all $n\geq 0,$ define  \[\mathcal{R}_n = \left\{ I\subseteq S \;\middle|\; |I|\leq n \text{ and } W_I \text{ is finite} \right\}.\] Recall that a subset $I \subset S$ such that $W_I$ is finite is called spherical. Our goal is now to prove the following theorem.
\begin{theorem}\label{VV}
Assume that there are conditional expectations $E_{I} : A_{I} \rightarrow B$ for all $I \in \mathcal{R}_3$, satisfying
\begin{itemize}
    \item For all $I,J \in \mathcal{R}_3$, $E_{I}$ and $E_{J}$ coincide on $A_I\cap A_J$.
    \item For all $I\in \mathcal{R}_3$, $w\in W_{I}$, if $w\neq e$ then $E_{I}(L(w)) = 0$.
\end{itemize}
then there exists a (necessarily unique) conditional expectation $E : A^{\max}\rightarrow B$ satisfying $E(U(w)) = 0$ for all $w \in W\setminus \{e\}.$ Moreover for all $I\in \mathcal{R}_2, a\in A_I^{\max},$ we have $E(a) = E_I(\alpha(a))$.
\end{theorem}
The strategy consists in building a Fock module $X$ on which $A^{\max}$ will act on the left. Let $I\in \mathcal{R}_3,$ we define $X_I = L^2(A_I,E_I)$ to be the right Hilbert $B$-module obtained by the GNS construction. We write $\eta_I \in X_I$ for the associated cyclic vector. By assumption we know that $A_I = \overline{\sum_{w\in W_I}L(w)}$ is an expected Bruhat decomposition. Hence by Proposition \ref{fock} we have
\[X_I = \bigoplus_{w\in W_I}X_I(w),\]
where for all $w\in W_I$, $X_I(w) = \overline{L(w)\eta_I}$.
\begin{lemma}
    Let $I,J \in \mathcal{R}_3$. Let $w\in W_{I\cap J} = W_I\cap W_J$. There is a right $B$-linear unitary isomorphism $X_I(w)\xrightarrow{\cong} X_J(w)$ sending $a\eta_I\in X_I(w)$ to $a\eta_J\in X_J(w)$ for all $a\in L(w)$.
\end{lemma}
\begin{proof}
    Use Proposition \ref{calcul} and the fact that $E_I$ and $E_J$ coincide on each $L(w)$ for $w\in W_{I\cap J}$.
\end{proof}
For all $s\in S$ we let $X(s) = X_{\{s\}}(s)$ so that $X_{\{s\}} = L^2(A_{\{s\}},E_{\{s\}}) = B\eta_{\{s\}}\oplus X(s)$. Let $\operatorname{Red}(W) \subset \coprod_{n\geq 0} S^n$ be the set of reduced words of $W$. Let \[
\operatorname{ev}\colon \operatorname{Red}(W)\longrightarrow W,
\qquad
(s_1,\ldots,s_n)\longmapsto s_1\cdots s_n.
\] be the evaluation map. We take the convention that the empty word $\emptyset$ is reduced and satisfies $\operatorname{ev}(\emptyset) = e$.
For all $f = (s_1,\ldots s_n) \in \operatorname{Red}(W)$ we define \[X(f) = X(s_1)\otimes_B\ldots \otimes_B X(s_n)\] as a $B$-correspondence. We let $X(\emptyset) = B.$ For all $w\in W,$ we choose an arbitrary reduced expression $f_w$ such that $\operatorname{ev}(f_w) = w$ and define the following right Hilbert $B$-module \[X = \bigoplus_{w\in W}X(f_w).\]
In particular, $X(f_e) = B$. The main difficulty is to prove that there is a natural $*$-homomorphism $A^{\max}\rightarrow \mathcal{L}_B(X).$ We first recall several results describing more precisely the combinatorics of reduced words, following primarily \cite[Chapter~2]{Ron09}. For any two reduced words $f_1,f_2,$ we denote by $f_1f_2 \in \operatorname{Red}(W)$ their concatenation whenever it is reduced. For all $f\in \operatorname{Red}(W),$ we let $f^{-1}$ denote the word with same letters as $f$ but in reverse order.

\begin{definition}
    For all $s,t\in S$ such that $m_{s,t}$ is finite we define $p(s,t) = (s,t,\ldots) \in \operatorname{Red}(W)$ to be the longest reduced word in $s$ and $t$ starting by $s$. An elementary homotopy is a transformation from a reduced word of the form $f_1p(s,t)f_2$ to the reduced word $f_1p(t,s)f_2$. A homotopy from a reduced word $f_1$ to a reduced word $f_2$ is a sequence of elementary homotopies from $f_1$ to $f_2$. For all $f_1,f_2\in \operatorname{Red}(W),$ we let $H(f_1,f_2)$ be the set of homotopies from $f_1$ to $f_2$ and we say that $f_1$ and $f_2$ are homotopic if there exists a homotopy between them. We write $\gamma : f_1\rightarrow f_2$ to signify $\gamma\in H(f_1,f_2)$. A homotopy from $f$ to $f$ is called a self-homotopy, we write $H(f) = H(f,f).$ The composition of two homotopies with corresponding source and target reduced words is just the concatenation of their underlying sequences of elementary homotopies.
\end{definition}
Matsumoto's theorem \cite[Theorem 3.3.1]{BB05} says that for all $f_1,f_2\in \operatorname{Red}(W),$ we have $\operatorname{ev}(f_1) = \operatorname{ev}(f_2)$ if and only if $f_1$ and $f_2$ are homotopic. For any two composable homotopies $\gamma_1,\gamma_2,$ we denote by $\gamma_1\gamma_2$ their concatenation. For any homotopy $\gamma\in H(f_1,f_2),$ we let $\gamma^{-1}\in H(f_2,f_1)$ denote the homotopy obtained by reversing the order of the elementary homotopies of $\gamma$. For all $s,t \in S$ such that $m_{s,t} < \infty$, write $s\vee t = \operatorname{ev}(p(s,t)) = \operatorname{ev}(p(t,s)).$

\begin{lemma}\label{unitaryri}
    Let $f_1,f_2 \in \operatorname{Red}(W)$. Every $\gamma\in H(f_1,f_2)$ induces a unitary right $B$-linear isomorphism $u_\gamma : X(f_1) \xrightarrow{\cong}X(f_2)$ such that $u_\gamma^* = u_{\gamma^{-1}}$. Moreover for all $f_1,f_2,f_3 \in \operatorname{Red}(W),$ $\gamma_1 \in H(f_2,f_3)$ and $\gamma_2 \in H(f_1,f_2)$ we have $u_{\gamma_1\gamma_2} = u_{\gamma_1}u_{\gamma_2}$.
\end{lemma}
\begin{proof}
    By Proposition \ref{fock}, we have a unitary $u_{s,t} : X(p(s,t)) \xrightarrow{\cong}X(p(t,s))$ for every $s,t\in S$ such that $m_{s,t} < \infty$. Indeed, we let $u_{s,t}$ be the composition of the unitaries \[X(p(s,t)) \xrightarrow{\cong} X_{\{s,t\}}(s\vee t) \xrightarrow{\cong} X(p(t,s)).\] Let $\gamma$ be an elementary homotopy from the reduced word $f_1 =g_1p(s,t)g_2$ to the reduced word $f_2 =g_1p(t,s)g_2$. We know that $X(f_1) = X(g_1)\otimes_B X(p(s,t))\otimes_BX(g_2)$ and $X(f_2) = X(g_1)\otimes_B X(p(t,s))\otimes_BX(g_2).$  Define $u_\gamma = \mathrm{id}_{X(g_1)}\otimes u_{s,t} \otimes \mathrm{id}_{X(g_2)}$. For every homotopy $\gamma,$ define $u_\gamma$ to be the composition of such operators corresponding to the elementary homotopies of $\gamma$.
\end{proof}
Let $s,t \in S$ such that $m_{s,t} < \infty.$ Let $m = m_{s,t}$. For any $a_1\in L(s), a_2 \in L(t)\ldots$ so that $a_1\ldots a_m \in L(s\vee t)$, there exists $(b_{1,i})_{1\leq i \leq n} \in L(t), (b_{2,i})_{1\leq i \leq n} \in L(s)\ldots $ such that $a_1a_2\ldots = \sum_{1\leq i \leq n} b_{1,i}b_{2,i} \ldots$ in $L(s\vee t) \subset A_{s,t}$. The unitary operator $u_{s,t} : X(p(s,t)) \xrightarrow{\cong} X(p(t,s))$ of the last proposition is then characterized by the formula \[u_{s,t}(a_1\eta_s \otimes a_2 \eta_t \ldots) = \sum_{1\leq i \leq n} b_{1,i}\eta_t \otimes b_{2,i} \eta_s \ldots.\]
\begin{definition}
    Let $f\in \operatorname{Red}(W)$. We call a self-homotopy inessential if it is of the form
\[
f=f_0
\rightarrow f_1
\rightarrow \cdots
\rightarrow f_{r-1}
\rightarrow f_r
\rightarrow f_{r-1}
\rightarrow \cdots
\rightarrow f_1
\rightarrow f_0=f,
\]
i.e. of the form $g^{-1}g$ for some homotopy $g :f_0 \rightarrow f_r$,  or if it is of the form
\[
\begin{tikzcd}[column sep=large,row sep=large]
f = f_1p(s,t)h p(s',t')f_2
  \arrow[r,rightarrow]
  \arrow[d,leftarrow]
&
f_1p(t,s)h p(s',t')f_2
  \arrow[d,rightarrow]
\\
f_1p(s,t)h p(t',s')f_2
  \arrow[r,leftarrow]
&
f_1p(t,s)h p(t',s')f_2 .
\end{tikzcd}
\]
\end{definition}
 \begin{definition}
    Let $I\subset S$. We say that a homotopy $\gamma$ lies in a residue of type $I$ if there exist $g_1,g_2 \in \operatorname{Red}(W)$ such that $\gamma$ is the composition of elementary homotopies of the form $g_1h_kg_2 \rightarrow g_1h_{k+1}g_2,$ involving only braid relations between elements of $I$, with $h_0, \ldots h_n \in \operatorname{Red}(W_I)$ representing the same element of $W_I$, and such that for all $k,$ $g_1h_kg_2$ is reduced.
\end{definition}
The following theorem is the reason why we ask for conditional expectations to be defined on the rank three spherical parabolic subalgebras. We say that a homotopy $\gamma$ decomposes into two homotopies $\gamma_1,\gamma_2$ if there exists a homotopy $\gamma_3$ such that $\gamma = \gamma_1\gamma_3^{-1}\gamma_3\gamma_2$ whenever this composition is well-defined.
\begin{theorem}\cite[Chapter~2, \S5, Theorem~2.17]{Ron09}\label{coherence}
Every self-homotopy decomposes into a sequence of self-homotopies each of which is inessential or lies inside some residue of type $I \in \mathcal{R}_3.$
\end{theorem}
This allows us to prove the following key lemma.
\begin{lemma}
    Let $f_1,f_2 \in \operatorname{Red}(W).$ Let $\gamma_1,\gamma_2 \in H(f_1,f_2).$ We have $u_{\gamma_1} = u_{\gamma_2}$.
\end{lemma}
\begin{proof}
Thanks to Lemma \ref{unitaryri} and the last theorem, it suffices to prove that $u_\gamma = 1$ whenever $\gamma\in H(f)$ is a self-homotopy which is inessential or lies in a residue of type $I$ for some $I\in \mathcal{R}_3.$  This is obvious when $\gamma$ is inessential. Assume that there exist $g_1,g_2 \in \operatorname{Red}(W)$ such that $\gamma$ is the composition of elementary homotopies of the form $g_1h_kg_2 \rightarrow g_1h_{k+1}g_2,$ involving only braid relations between elements of $I$, with $h_0, \ldots h_n \in \operatorname{Red}(W_I)$ representing the same element $w$ of $W_I$, such that for all $k,$ $g_1h_kg_2$ is reduced. As $\gamma$ is a self-homotopy we have $h_0 = h_n=h$. We know that $X(f) = X(g_1)\otimes_BX(h)\otimes_BX(g_2)$ and that there exists a unitary operator $U \in \mathcal{L}_B(X(h))$ such that $u_\gamma = \mathrm{id}_{X(g_1)}\otimes U \otimes \mathrm{id}_{X(g_2)}$. Moreover $U$ is the composition of unitaries $u_{\gamma_k} : X(h_k) \xrightarrow{\cong} X(h_{k+1})$ coming from elementary homotopies $\gamma_k : h_k \rightarrow h_{k+1}$ happening inside $\operatorname{Red}(W_I)$. As $w\in W_I$ and $A_I = \overline{\sum_{w\in W_{I}}L(w)}$ is an expected Bruhat decomposition, there is a unitary $V(h') : X(h')\xrightarrow{\cong} X_I(w)$ for all $h'\in \operatorname{Red}(W_I)$ such that $\operatorname{ev}(h') = w.$ By definition of the unitaries $u_{s,t}$, we know that for all $1 \leq k \leq n,$ we have $V(h_{k+1})u_{\gamma_k}V(h_k)^{-1} = \mathrm{id}_{X_I(w)}.$ This gives $V(h)UV(h)^{-1} = \mathrm{id}_{X_I(w)}$, hence $U = \mathrm{id}_{X(h)}$.
\end{proof}
If $f_1$ and $f_2$ are two reduced expressions of the same element $w\in W$, we write \[U(f_1\rightarrow f_2) =u_\gamma :X(f_1) \xrightarrow{\cong} X(f_2)\] for any homotopy $\gamma : f_1 \rightarrow f_2.$ Recall that for each $I \subset S$, each coset of $W_I\backslash W$ admits a unique element of minimal length \cite{BB05}. Denote by $^IW$ the set of these minimal length representatives so that for each $u\in W_I$ and $v\in {}^IW$ we have $l(uv) = l(u) + l(v)$. For all $I\subset S$ and $w\in W$ we will write $w = w_I{}^Iw$ with $w_I \in W_I$ and ${}^Iw\in {}^IW$.
\begin{proof}[Proof of Theorem~\ref{VV}]
Let $I \in \mathcal{R}_2$. Define the unitary 
\[
U_{I,1}:=\bigoplus_{w\in W}U(f_w\rightarrow f_{w_I}f_{{}^Iw}):
X=\bigoplus_{w\in W}X(f_w)
\xrightarrow{\cong}
\bigoplus_{w\in W}X(f_{w_I})\otimes_B X(f_{{}^Iw}).
\]
Recall that $X_I = \bigoplus_{w\in W_I}X_I(w)$. Define the right Hilbert $B$-module \[{}^IX = \bigoplus_{w\in {}^IW}X(f_{w}).\] Define the unitary 
\[U_{I,2}: \bigoplus_{w\in W_I}X(f_{w_I}) \xrightarrow{\cong}
\bigoplus_{w\in W_I}X_I(w) = X_I 
\] as the one given by Proposition \ref{fock}. Write \[\bigoplus_{w\in W}X(f_{w_I})\otimes_B X(f_{{}^Iw}) = (\bigoplus_{w\in W_I}X(f_{w}))\otimes_B (\bigoplus_{w\in {}^IW}X(f_{w}))\] and use this decomposition to define the unitary \[U_I = (U_{I,2}\otimes \mathrm{id}_{{}^IX})U_{I,1} : X = \bigoplus_{w\in W} X(f_w) \xrightarrow{\cong} X_I\otimes_B{}^IX.\]
Let $\rho_I : A_I \rightarrow \mathcal{L}_B(X_I)$ be the left action coming from the GNS construction. Let $\pi_I : A_I\rightarrow \mathcal{L}_B(X)$ be the $*$-homomorphism defined by $\pi_I(a) = U_I^*(\rho_I(a)\otimes \mathrm{id}_{{}^IX})U_I$ for all $a\in A_I.$

Let us check that the family $(\pi_I:A_I \rightarrow\mathcal{L}_B(X))_{I\in \mathcal{R}_2}$ defines a $*$-homomorphism $\pi: A^{\max}\rightarrow \mathcal{L}_B(X)$. Let $I\subset J \in \mathcal{R}_2.$ Define the right Hilbert $B$-module \[{}^IX_J = \bigoplus_{w\in W_J\cap{}^IW}X_J(w).\] Notice that $X_J = X_I\otimes_B {}^IX_{J}$ and that for all $a\in A_I \subset A_J$ we have $\rho_J(a) = \rho_I(a)\otimes\mathrm{id}_{{}^IX_J}$. The restriction of the unitary $U_{J,1}$ is a unitary between the complemented Hilbert submodules \[{}^IV_{J,1}=U_{J,1}|_{{}^IX}:\bigoplus_{w\in {}^IW}U(f_w\rightarrow f_{w_I}f_{{}^Iw}):
{}^IX=\bigoplus_{w\in {}^IW}X(f_w)
\xrightarrow{\cong}
\bigoplus_{w\in {}^IW}X(f_{w_J})\otimes_B X(f_{{}^Jw}).\] Moreover, the restriction of the unitary $U_{J,2}$ induces a unitary \[{}^IV_{J,2}: \bigoplus_{w\in {}^IW_J}X(f_{w_J}) \xrightarrow{\cong}
\bigoplus_{w\in {}^IW_J}X_J(w) = {}^IX_J. 
\] Hence we have a unitary ${}^IV_J = ({}^IV_{J,2}\otimes \mathrm{id}_{{}^JX}){}^IV_{J,1}:{}^IX \xrightarrow{\cong} {}^IX_J\otimes_B {}^JX$. The composition \[X \xrightarrow{U_J} X_J\otimes_B {}^JX = X_I \otimes_B {}^IX_J\otimes_B{}^JX \xrightarrow{\mathrm{id}_{X_I}\otimes {}^IV_J^{-1}}X_I \otimes_B{}^IX \xrightarrow{U_I^{-1}}X\] stabilizes each submodule $X(f_w)$ and, when restricted to $X(f_w)$ is the composition of unitary operators given by Lemma \ref{unitaryri} associated to a self-homotopy of $w$. Hence by the last lemma this composition is the identity and we have $U_J = (\mathrm{id}_{X_I}\otimes {}^IV_J)U_I.$ Let $a\in A_I \subset A_J.$ Compute \begin{align*}
\pi_J(a)
&= U_I^*
   \bigl(\mathrm{id}_{X_I}\otimes {}^I V_J^*\bigr)
   \bigl(
      \rho_I(a)\otimes \mathrm{id}_{{}^I X_J}
      \otimes \mathrm{id}_{{}^J X}
   \bigr)
   \bigl(\mathrm{id}_{X_I}\otimes {}^I V_J\bigr)
   U_I \\
&= U_I^*
   \bigl(
      \rho_I(a)\otimes
      ({}^I V_J^*\,{}^I V_J)
   \bigr)
   U_I \\
&= U_I^*
   \bigl(
      \rho_I(a)\otimes \mathrm{id}_{{}^I X}
   \bigr)
   U_I \\
&= \pi_I(a).
\end{align*}
Hence we do have a unital $*$-homomorphism $\pi : A^{\max} \rightarrow \mathcal{L}_{B}(X)$ extending the family of $*$-homomorphisms $(\pi_I)_{I\in \mathcal{R}_2}$. We define the conditional expectation $E : A^{\max}\rightarrow B$ by $E(a) = \langle\eta,\pi(a)\cdot\eta\rangle_B$ where $\eta = 1_B \in B = X(f_e)$. We now show that $E(L(w)) = 0$ for all $w\in W\setminus\{e\}.$ We do so by proving that for all $w\in W,$ we have $\pi(U(w))\cdot \eta \subset X(f_w)$ by induction on $l(w) = n.$ The result is true for $w=e.$ Assume that the result holds for some $n\geq 0$. Take $w\in W$ and $s\in S$ with $l(w) = n$ and $l(sw) = n+1.$ We have \begin{align*}
\pi(U(sw))\cdot \eta
&= \pi(U(s))\cdot\bigl(\pi(U(w))\cdot \eta\bigr) \\
&= \pi(U(s))\cdot X(f_w) \\
&= \left(
    U_{\{s\}}^*
    \bigl(\rho_{\{s\}}(L(s))\otimes \mathrm{id}_{{}^{\{s\}}X}\bigr)
    U_{\{s\}}
   \right)\cdot X(f_w).
\end{align*}
As $l(sw)> l(w),$ we have $U_{\{s\}}\cdot X(f_w) \subset B\cdot\eta_s \otimes_B{}^{\{s\}}X$ through the identification $ X_{\{s\}}\otimes {}^{\{s\}}X = (B\cdot\eta_{\{s\}}\otimes_B{}^{\{s\}}X) \oplus (X(s)\otimes_B {}^{\{s\}}X).$ Moreover we have $\rho_{\{s\}}(L(s))\cdot \eta_{\{s\}} = X(s)$. This gives $\pi(U(sw))\cdot \eta = U_{\{s\}}^*(X(s)\otimes_B X(f_w)) = X(f_{sw}).$ Let $I \in \mathcal{R}_2$, the equality $E(a) = E_I(\alpha(a))$ for all $a\in A_I^{\max},$ follows as $E$ and $E_I \circ \alpha$ are two conditional expectations of the same Bruhat decomposition when restricted to $A_I^{\max}$.
\end{proof}
\begin{example}
    Let $A_1$ and $A_2$ be two unital C$^*$-algebras with a common C$^*$-subalgebra $B$ with the same unit and conditional expectations $E_k : A_k \rightarrow B$. Rewrite $A_1 = A_s$, $A_2 = A_t$ $E_1 = E_s$ and $E_2 = E_t$. Let $A$ be the maximal amalgamated free product $A_1 \star_B^m A_2$. Define $L(e) = B$, $L(s) = A_s^0 = \{a\in A_s| E_s(a) = 0\}$ and $L(t) = A_t^0$.  For all $w\in D_{\infty} = \langle s,t | s^2 = t^2 = e \rangle$ define $L(w)$ as the space of linear combinations of elements of the form $a_{s_1}\ldots a_{s_n}$ with $a_{s_{k}} \in A_{s_{k}}^0$ such that $w = s_1\ldots s_n$ is a reduced expression of $w$. Here we have $\mathcal{R}_3 = \{\emptyset, \{s\}, \{t\}\}$. All the conditions of Theorem \ref{VV} are satisfied, hence $A = \overline{\sum_{w\in W}L(w)}$ together with the conditional expectation $E : A \rightarrow B$ coming from gluing $E_1$ and $E_2$ using the last theorem is an expected Bruhat decomposition of $A$. Notice that $A^{\max} = A$ here. The Fock module $X$ coincides with the one of Voiculescu \cite{Voi85} and the reduced C$^*$-algebra is exactly the reduced amalgamated free product $A^r \cong (A_1,E_1)\star_B(A_2,E_2)$.
\end{example}
\begin{example}\label{graphproductexample}
    Let $\Gamma=(V\Gamma,E\Gamma)$ be a simplicial graph. Let 
\[
W_\Gamma
=
\left\langle
V\Gamma
\;\middle|\;
v^2=1 \text{ for all } v\in V\Gamma,\quad
uv=vu \text{ whenever } (u,v)\in E\Gamma
\right\rangle .
\]
be the associated right-angled Coxeter group. Let $(A_v)_{v\in V\Gamma}$ be a family of unital C$^*$-algebras and $\phi_v : A_v \rightarrow \mathbb{C}$ a family of GNS-faithful states. It is not hard to show that the maximal graph product C$^*$-algebra introduced in \cite{CF17} admits a natural Bruhat decomposition of type $W_\Gamma$ with a trivial Borel C$^*$-subalgebra $B = \mathbb{C}$. Gluing the states of the form $\bigotimes_{v\in I} \phi_v: \bigotimes_{v\in I} A_v \rightarrow \mathbb{C}$ for all clique $I \subset V\Gamma$ of cardinal at most $3$, the last theorem tells us that this decomposition is expected. It is not hard to show that the graph product Hilbert space of \cite{CF17} coincides with the Fock space $X$ built above and that their reduced C$^*$-algebra coincides with ours.
\end{example}

\section{A noncommutative compactification of buildings}\label{section3}

\subsection{Klisse's topological boundary}\label{31}
In this section we recall some notions and results of \cite{Kli23,Lec09,DJ99} that we will need in order to study the covariance C$^*$-algebras associated to an expected $W$-decomposition. We first recall the construction of the bordification of a rooted graph as in \cite[Section 2]{Kli23}. A graph is the data of a set $V$ of vertices and a subset $E \subset V\times V$ of edges. We assume here that $(x,x) \notin E$ for all $x\in V$ and $(y,x) \in E$ for all $(x,y) \in E$ (i.e. the graph is undirected and simplicial). We also assume the graph $K$ to be connected and the set $V$ of vertices to be countable. The set $V$ will be seen as a metric space with the graph distance (i.e. the shortest-path metric) induced by $E$. We write this distance $d_K$. A rooted graph is the data $(K,o)$ of a graph $K = (V,E)$ together with a distinguished vertex $o\in V$. For such a rooted graph, define the graph order on $V$ based at $o$ by $y \geq_o x$ if there exists a geodesic path of $K$ starting with $o$, passing by $x$ and ending at $y$. We will often write $\geq$ for $\geq_o$ in this section, but it is important to remember that this partial order depends on the choice of base vertex.
\begin{definition}\cite{Kli23}
    Let $K = (V,E)$ be a rooted graph with root $o\in V$. Let $\mathbf{x} = (x_n)\in V^\mathbb{N}$ be any sequence of vertices. We say that $\mathbf{x}$ $o$-converges if for all $y\in V$ we either have $x_n \geq y$ for $n$ large enough or $x_n \ngeq y$ for $n$ large enough. We say that $\mathbf{x}$ $o$-converges to infinity if $\mathbf{x}$ $o$-converges and $\sup_{y\leq \mathbf{x}}d_K(y,o) = \infty$. The bordification $\overline{(K,o)}$ of the rooted graph $(K,o)$ is the quotient of the set of all $o$-converging sequences of $V$ by the equivalence relation given by $\mathbf{x} \sim_o \mathbf{y}$ if \[\forall z \in V, z \leq \mathbf{x} \iff z \leq \mathbf{y}.\]
The boundary of $(K,o)$ is the subset $\partial(K,o)$ of $\overline{(K,o)}$ consisting of the equivalence classes of sequences $o$-converging to infinity.
\end{definition}
Constant sequences of $V$ clearly $o$-converge. This allows us to see vertices as elements of the bordification $\overline{(K,o)}$. One easily extends the graph order to $\overline{(K,o)}$ by stating that $\mathbf{x} \leq \mathbf{y}$ if for all $z \in V$, $z \leq \mathbf{x}$ implies $z \leq \mathbf{y}$. 

\begin{proposition}\cite{Kli23}
    Let $(K,o)$ be a connected rooted graph. Equip the bordification $\overline{(K,o)}$ with the topology generated by the subbasis of open subsets of the form \[\mathcal{U}_x = \{z\in \overline{(K,o)} | z \geq x\} \qquad \mathcal{U}_x^c = \{z\in \overline{(K,o)} | z \ngeq x\}\]
    Then $\overline{(K,o)}$ is a compact Hausdorff space containing $V$ as a dense subset. Moreover, when $K$ is locally finite, $V \subset \overline{(K,o)}$ inherits the discrete topology.
\end{proposition}
Let $x \in V$. Let $P_x \in B(l^2V)$ be the orthogonal projection on the Hilbert subspace generated by the vectors $\delta_y$ with $y \geq x$. Let $\mathcal{D}(K,o)$ be the commutative C$^*$-subalgebra of $B(l^2V)$ generated by the projections $P_x$ for $x\in V$.
\begin{proposition}\cite{Kli23}\label{CT}
 There is a $*$-isomorphism $C(\overline{(K,o)}) \cong \mathcal{D}(K,o)$.
\end{proposition}

Let $(W,S)$ be a Coxeter system. Its Cayley graph (also denoted $W$) is a connected rooted graph, the root being the unit $e$. The graph order on this graph coincides with the right weak order on $W$ \cite{BB05}. We denote simply by $\overline{(W,S)}$ and $\partial(W,S)$ the bordification and boundary of the Cayley graph. When $S$ is finite, Klisse showed that the natural left action of $W$ on its Cayley graph extends to a continuous action on the compact Hausdorff space $\overline{(W,S)}$ \cite[Theorem 3.3]{Kli23}. The following result will be central in our approach. 
\begin{theorem}\label{coxexact}
Assume that $S$ is finite. The induced action $W \curvearrowright \overline{(W,S)}$ is amenable.
\end{theorem}
Dranishnikov and Januszkiewicz \cite{DJ99} only proved the exactness of $W$. It was Lécureux \cite{Lec09}, and then Klisse \cite{Kli23}, who refined their result by giving a more pleasant description of a space on which $W$ acts amenably.

\subsection{The covariance \texorpdfstring{C$^*$-algebras}{C*-algebras}}
We fix a Coxeter system $(W,S)$. We consider the right weak order on $W$ (see \cite[Section 3.1.]{BB05}) and denote it simply by $\leq$. Recall that for all $u,v\in W$, we have $u\leq v$ if and only if there is a $w\in W$ such that $v = uw$ with $l(v) = l(u) + l(w)$ (i.e. $u$ is a prefix of $v$).

We fix a unital C$^*$-algebra $A$ together with an expected $W$-decomposition $A = \overline{\sum_{w\in W}L_{i,j}(w)}$ with conditional expectations $E_i : A \rightarrow B_i$ and index set $\mathscr{I}$. Recall from last section that the reduced C$^*$-algebra $A^r = \overline{\sum_{w\in W}F_{i,j}(w)}$ is a C$^*$-subalgebra of $\mathcal{L}_\mathcal{P}(\mathcal{X})$ where $\mathcal{X}$ is the right Hilbert $\mathcal{P}$-module given by \[\mathcal{X} = \bigoplus_{w\in W}\mathcal{X}_i(w).\]
for an arbitrary choice of $i\in \mathscr{I}$. Here $\mathcal{X}_i(w) = \bigoplus_{j\in \mathscr{I}}X_{i,j}(w)\otimes_{B_j}\mathcal{P}$ for all $i\in \mathscr{I}$.

Let $i\in \mathscr{I}$. Let $p_{i,w} \in \mathcal{L}_{\mathcal{P}}(\mathcal{X})$ be the orthogonal projection on the Hilbert submodule \[\mathcal{X}_{i,\geq w} = \bigoplus_{u\geq w} \mathcal{X}_i(u)\]
and let $q_{i,w} \in \mathcal{L}_{\mathcal{P}}(\mathcal{X})$ be the orthogonal projection on the Hilbert submodule \[\mathcal{X}_{i,w+} = \bigoplus_{w \cdot u \text{ reduced}} \mathcal{X}_i(u).\]
Notice that for all $i\in \mathscr{I}$ we have $p_{i,e} = q_{i,e} = \mathrm{id}_\mathcal{X}$ and for all $s\in S$ we have $ q_{i,s} = p_{i,s}^\perp$.
Let $\Delta_0(i)$ be the C$^*$-subalgebra of $\mathcal{L}_\mathcal{P}(\mathcal{X})$ generated by the projections $p_{i,w}$ for $w\in W$. Let $\Delta(i)$ be the C$^*$-subalgebra of $\mathcal{L}_\mathcal{P}(\mathcal{X})$ generated by the subset $\Delta_0(i) \cup B_i$. Notice that for all $u,v\in W$ we have \[
p_{i,u} p_{i,v} =
\begin{cases}
p_{i,u \vee v} & \text{if } u \vee v < \infty, \\
0 & \text{otherwise.}
\end{cases}
\]
Here $u \vee v$ denotes (when it exists) the join of $u$ and $v$ for the right weak order, i.e. the smallest word of $W$ beginning both with $u$ and $v$.
\begin{proposition}\label{cw}
Fix an $i\in \mathscr{I}$. Assume that for all $s\in S$ and $j \in \mathscr{I}$ the left action of $B_j$ on $\mathcal{X}_j(s)$ is faithful. Then there are $*$-isomorphisms \[ \Delta_0(i) \cong \mathcal{D}(W,e), \qquad \Delta(i)\cong \mathcal{D}(W,e)\otimes B_i.\]
\end{proposition}
\begin{proof}
 By \cite[Theorem 2.13]{Kli23}, $\mathcal{D}(W,e)$ is the universal C$^*$-algebra generated by projections $P_w$ such that for all $u,v \in W$ we have $P_{u} P_{v} =
\begin{cases}
P_{u \vee v} & \text{if } u \vee v < \infty, \\
0 & \text{otherwise}.
\end{cases}$
Hence there is a surjective $*$-homomorphism $ \psi_i: \mathcal{D}(W,e) \twoheadrightarrow \Delta_0(i)$.
Moreover, every element of $B_i$ commutes with the projections $p_{i,w}$ for $w\in W$. 
Hence there is a surjective $*$-homomorphism $\chi_i :\mathcal{D}(W,e)\otimes_{\max}B_i \twoheadrightarrow \Delta(i)$ being the identity on $B_i$ and sending each $P_w$ to $p_{i,w}$. Moreover $\mathcal{D}(W,e)$ is commutative, thus nuclear. Hence we actually have a $*$-homomorphism defined on the minimal tensor product $\chi_i : \mathcal{D}(W,e)\otimes B_i \rightarrow \Delta(i)$. 

We prove the injectivity of $\chi_i$, the injectivity of $\psi_i$ follows automatically. Let $f\in C(\overline{(W,S)},B_i)$ be a positive nonzero element. There exists a $w\in W$ such that $f(w) \neq 0.$ We assume that the left action of $B_j$ on each $\mathcal{X}_j(s)$ is faithful. An easy induction on the length of $w$ using Lemma \ref{xxxxxxx} and Lemma \ref{ZERO} shows that the left action of $B_i$ on $\mathcal{X}_i(w)$ is faithful for all $w\in W$. This implies that the restriction of $\chi_i(f)$ to the Hilbert subspace $\mathcal{X}_i(w)$ is nonzero. Hence $\chi_i(f) \neq 0$.
\end{proof}
As the last proposition will be fundamental in our approach, we shall check that the nondegeneracy condition holds in the case of groups.
\begin{proposition}\label{nondeggroup}
    Let $A = \overline{\bigoplus_{w\in W}F_{i,j}(w)}$ be the expected $W$-decomposition coming from the $W$-decomposition of a discrete group $G = \biguplus_{w\in W}C_{i,j}(w)$ with index set $\mathscr{I}$. Let $B_i = C^*_r(H_i)$ be the Borel C$^*$-subalgebras. For all $i\in \mathscr{I}$ and $s\in S,$ the left action of $B_i$ on $\mathcal{X}_i(s)$ is faithful.
\end{proposition}
\begin{proof}
We have $\mathcal{L}_{\mathcal{P}}(\mathcal{X}_i(s)) \cong \prod_{j\in \mathscr{I}}\mathcal{L}_{B_j}(X_{i,j}(s)).$ For all $j \in \mathscr{I},$ let $\lambda_j : B_j \hookrightarrow B(l^2H_j)$ be the left regular representation. Use $\lambda_j$ and Lemma \ref{ZERO} to embed $\mathcal{L}_{B_j}(X_{i,j}(s))$ into $B(X_{i,j}(s)\otimes_{B_j}l^2H_j)$. By identifying $X_{i,j}(s)\otimes_{B_j}l^2H_j$ with $l^2C_{i,j}(s)$, we get a faithful $*$-homomorphism $\mathcal{L}_{\mathcal{P}}(\mathcal{X}_i(s)) \hookrightarrow B(\bigoplus_{j \in \mathscr{I}}l^2C_{i,j}(s)).$ Under this identification, the left action of $B_i=C^*_r(H_i)$ on $\mathcal{X}_i(s)$ is sent by $\iota$ to the representation induced by left multiplication of $H_i$ on $\coprod_{j\in\mathscr{I}}C_{i,j}(s)$. By (Wdec3), there exists a $j\in\mathscr{I}$ such that $C_{i,j}(s)$ is nonempty. Let $g\in C_{i,j}(s)$. We know that $l^2(H_ig)\subseteq l^2C_{i,j}(s)$ is invariant under the left action of $H_i$. The unitary given by right multiplication by $g$ intertwines the left regular representation of $H_i$ on $l^2H_i$ with its left action on $l^2(H_ig)$. Hence the left action of $B_i$ on $l^2(H_ig)$ is faithful.
\end{proof}
Let $i,j\in \mathscr{I}$ and $s\in S$. Let $a\in F_{i,j}(s)$. Write $a^\dagger = p_{i,s}ap_{j,s}^\perp$ and $\partial(a) = p_{i,s}ap_{j,s}$ (we should always remember to specify the indices $i$ and $j$ as well as the generator $s\in S$ when considering these operators). It is easy to see that $p_{i,s}^\perp ap_{j,s}^\perp = 0.$ This gives the following decomposition \[a = \partial(a) + a^\dagger + ((a^*)^\dagger)^*.\] Elements of the form $\partial(a)$ are called elementary diagonal operators of type $(i,j,s)$. Elements of the form $a^\dagger$ and $(a^\dagger)^*$ are called respectively creation and annihilation operators of type $(i,j,s)$.
\begin{proposition}
    Let $w\in W, w\neq e$. Let $w = s_1 \ldots s_n$, $s_i \in S$, be a reduced expression of $w$. Let $i_0, \ldots i_n \in \mathscr{I}$. For all $1\leq k \leq n$ let $a_k \in F_{i_{k-1},i_{k}}(s_k)$. Let $a = a_1\ldots a_n \in F_{i_0,i_n}(w)$. We have \[a_1^\dagger\ldots a_n^\dagger = aq_{i_{n},w}\]
\end{proposition}
\begin{proof}
Let $u \in W$. Let $j \in \mathscr{I}$ and $b \in F_{i_n,j}(u)$. Then
\[
a_1^\dagger \cdots a_n^\dagger \cdot (b\eta_j)
=
\begin{cases}
0 & \text{if } s_n \cdot u \text{ is not reduced}, \\
a_1^\dagger \cdots a_{n-1}^\dagger (a_n b \eta_j)
  & \text{if } s_n \cdot u \text{ is reduced}.
\end{cases}
\]
The result follows by an induction on $n$.
\end{proof}
We write $a^\dagger = aq_{j,w}$ for all $a\in F_{i,j}(w)$ and $w\in W$. In particular, for all  $i \in \mathscr{I}$ and $b\in B_i$, we have $b^\dagger = b$. Notice that when $u,v \in W$ are such that $u\cdot v$ is a reduced expression we have $a_1^\dagger a_2^\dagger = (a_1a_2)^\dagger$ for all $a_1 \in F_{i,j}(u), a_2 \in F_{j,k}(v)$ with $i,j,k \in \mathscr{I}.$
\begin{proposition}\label{QW}
    Let $i,j,k \in \mathscr{I}$. Let $w\in W$. For all $a\in F_{i,j}(w),b\in F_{i,k}(w)$ we have \[(a^\dagger)^*b^\dagger = \delta_{j,k}E_j(a^*b)q_{j,w}.\]
\end{proposition}
\begin{proof}
    We proceed by induction on $l(w)$. The proposition is obvious for $a,b \in B$. Let $s\in S$. Let $a\in F_{i,j}(s),b\in F_{i,k}(s)$. Assume first that $j = k$. Compute \[(a^\dagger)^*b^\dagger = p_{j,s}^\perp a^* b p_{j,s}^\perp = p_{j,s}^\perp (a^*b-E_j(a^*b))p_{j,s}^\perp + E_j(a^*b)p_{j,s}^\perp= E_j(a^*b)p_{j,s}^\perp.\] Assume now that $j\neq k.$ In that case we must have $a^*b \in F_{j,k}(s)$ by (OWD2). This implies $(a^\dagger)^*b^\dagger = p_{j,s}^\perp a^* b p_{k,s}^\perp = 0$.
    
    Let $n\geq 1$ be an integer. Assume that the proposition is true for all elements of $W$ of length $n$. Let $w\in W$ be of length $n+1$. Write $w = su$ for $s\in S$ and $u\in W$ of length $n$. Let $a = a_1a_2 \in F_{i,k}(w), b= b_1b_2 \in F_{i,k'}(w)$ with $a_1 \in F_{i,j}(s)$, $b_1 \in F_{i,j'}(s)$, $a_2 \in F_{j,k}(u)$ and $b_2 \in F_{j',k'}(u)$. We first assume $j = j'$. Compute \[(a^\dagger)^*b^\dagger = (a_2^\dagger)^* (a_1^\dagger)^*b_1^\dagger b_2^\dagger =  (a_2^\dagger)^* E_j(a_1^*b_1)p_{j',s}^\perp b_2q_{k',u} = (a_2^\dagger)^* E_j(a_1^*b_1) b_2q_{k',w}.\] The last equality comes from the fact that for all $v\in W$ we have $l(uv) = l(u)+l(v)$ and $l(suv) = 1+ l(uv)$ if and only if $l(wv) = l(w) + l(v)$. Let $b_2' = E_j(a_1^*b_1)b_2\in F_{j,k'}(u)$. By the induction hypothesis we have \[(a^\dagger)^*b^\dagger = (a_2^\dagger)^* (b_2')^\dagger q_{k',w} = \delta_{k,k'}E_k(a_2^*b_2')q_{k',u}q_{k,w} = \delta_{k,k'}E_k(a_2^*E_j(a_1^*b_1)b_2)q_{k,w} = \delta_{k,k'}E_k(a^*b)q_{k,w},\] by using Proposition \ref{calcul}. We now assume $j\neq j'$. As $L_{j,j'}(e) = 0$, we know that $a_1^*b_1 \in L_{j,j'}(s)$ and $(a_1^\dagger)^*b_1^\dagger = p_{j,s}^\perp a_1^* b_1 p_{j',s}^\perp = 0.$ This implies $(a^\dagger)^*b^\dagger = 0.$ This concludes the proof as Proposition \ref{calcul} gives $E_k(a^*b) = 0$ whenever $k=k'$.
\end{proof}
\begin{lemma}\label{qwqw}
Let $w\in W\setminus\{e\}$. There are $w_1, \ldots w_k \in W$ such that for all $u\in W$, \[ w\cdot u \text{ is reduced} \iff \forall  1\leq i \leq k, u \text{ does not begin with } w_i.\]
\end{lemma}
\begin{proof}
We proceed by induction on $l(w)$. For all $s\in S$, $u\in W$, the expression $s\cdot u$ is reduced if and only if $u$ does not begin with  $s$. Let $n\geq 1$ be an integer. Assume the lemma holds for all elements $w\in W$ of length $n$. Let $w\in W$ be of length $n+1$. Write $w = vs$ with $s\in S, v\in W$ and $l(v) = n$. Take $w_1, \ldots w_k \in W$ so that the proposition holds for $v$. For all $u \in W$, \[
\begin{aligned}
(vs)\cdot u \text{ is reduced}
&\iff
\begin{cases}
s \cdot u \text{ is reduced, and} \\
v \cdot (su) \text{ is reduced}
\end{cases} \\
&\iff
\begin{cases}
s \cdot u \text{ is reduced, and} \\
\text{for all } 1 \leq i \leq k,\ (su) \text{ does not begin with } w_i
\end{cases} \\
&\iff
\begin{cases}
u \text{ does not begin with } s,\ \text{and} \\
\text{for all } 1 \leq i \leq k,\ su \text{ does not begin with } s \vee w_i
\end{cases} \\
&\iff
\begin{cases}
u \text{ does not begin with } s,\ \text{and} \\
\text{for all } 1 \leq i \leq k,\ u \text{ does not begin with } s(s \vee w_i).
\end{cases}
\end{aligned}
\]Here we take the convention that the conditions $su \geq s\vee w_i$ and $u\geq s(s\vee w_i)$ are false when $s\vee w_i = \infty$.
\end{proof}
In particular, for all $i\in \mathscr{I}$ and $w\in W$, we have $q_{i,w} = p_{i,w_1}^\perp \ldots p_{i,w_k}^\perp$, thus $q_{i,w} \in \Delta_0(i)$.
\begin{definition}
Let $i,j \in \mathscr{I}$. A path operator from $i$ to $j$ is either an element of some $B_i$ in the case when $i = j$ or an
operator of the form
\[
f_1\cdots f_n,
\]
where $n\geq 1$, $s_1,\ldots,s_n\in S$, and
$i_0,\ldots,i_n\in\mathscr{I}$ are such that $i = i_0, j = i_n$ and where, for
each $1\leq k\leq n$, the operator $f_k$ is the creation,
annihilation, or elementary diagonal operator associated with an
element $a_k\in F_{i_{k-1},i_k}(s_k)$. Such an operator is called based at $i$ if $i = j$ (i.e. it is a path operator from $i$ to $i$). The covariance C$^*$-algebra based at $i$ is the C$^*$-subalgebra $\mathscr{C}(i) \subset \mathcal{L}_\mathcal{P}(\mathcal{X})$ generated by all path operators based at $i$. The global covariance C$^*$-algebra is the C$^*$-subalgebra $\mathscr{C} \subset \mathcal{L}_\mathcal{P}(\mathcal{X})$ generated by all path operators with arbitrary source and target.
\end{definition}
The first thing to notice is that $A^r \subset \mathscr{C}(i)$. As $A^r = \overline{\sum_{w\in W}F_{i,i}(w)},$ we know that $A^r$ is spanned by elements of $B_i$ and elements of the form $a_1\ldots a_n$ with $n\geq 1$, $s_1,\ldots,s_n\in S$, and $i_0,\ldots,i_{n}\in\mathscr{I}$ such that $i_0 = i_n = i$ and $a_k\in F_{i_{k-1},i_k}(s_k).$ Every such $a_k\in F_{i_{k-1},i_k}(s_k)$ is a sum of a creation, an annihilation and an elementary diagonal operator of type $(i_{k-1},i_k,s_k)$. We let $\mathcal{E}(i)$ be the set of path operators based at $i$ so that \[
\mathscr{C}(i)=\overline{\operatorname{span}}\,\mathcal{E}(i).
\]

Let $i_0, \ldots i_n\in \mathscr{I}$. Take a path operator of the form $T =f_1 \cdots f_n$, where each $f_k$ is a creation, annihilation, or elementary diagonal operator of type $(i_{k-1},i_k,s_k)$. Define $\epsilon_1 \cdots \epsilon_n = w \in W$ to be the degree of the operator $T$, where for all $1 \leq k \leq n$, we set
\[
\epsilon_k =
\begin{cases}
s_k & \text{if } f_k \text{ is a creation or annihilation operator}, \\
e   & \text{if } f_k \text{ is an elementary diagonal operator}.
\end{cases}
\]
An induction on $n$ proves the following proposition.
\begin{proposition}
    If $T =f_1 \cdots f_n \in\mathcal{L}_\mathcal{P}(\mathcal{X})$ is a path operator of degree $w$ then for all $j \in \mathscr{I}$ and $u \in W$, we have $T(X_{i_n,j}(u)) \subset X_{i_0,j}(wu)$.
\end{proposition}
Let $w \in W$. Define $\mathcal{E}_w(i)$ to be the set of path operators based at $i$ of degree $w$. Let $\mathscr{C}_w(i) = \overline{\operatorname{span}}\,\mathcal{E}_w(i)$ be the set of operators of $\mathscr{C}(i)$ of degree $w$. We call $\mathcal{D}(i) = \mathscr{C}_e(i)$ the diagonal C$^*$-subalgebra based at $i$.

\begin{definition}\cite{Exe97}
Let $\Gamma$ be a discrete group. Let $\mathcal{A}$ be a C$^*$-algebra. A $\Gamma$-grading of $\mathcal{A}$ is the data of a family of closed linear subspaces $(\mathcal{A}_g)_{g\in \Gamma}$ that are in direct sum and such that \begin{enumerate}[label=(\roman*)] 
\item $\mathcal{A} = \overline{\bigoplus_{g\in \Gamma}\mathcal{A}_g}$
\item $\forall g,h\in \Gamma, \mathcal{A}_g\mathcal{A}_h \subset \mathcal{A}_{gh}$
\item $\forall g\in \Gamma, \mathcal{A}_g^* = \mathcal{A}_{g^{-1}}$
\end{enumerate}
\end{definition}
Such a grading is called topological if there exists a conditional expectation $\mathcal{E} : \mathcal{A} \rightarrow \mathcal{A}_e$ such that $\mathcal{E}(\mathcal{A}_g) = 0$ for all $g\in \Gamma \setminus \{e\}$.
\begin{proposition}\label{cgraded}
There is a faithful conditional expectation
    $F : \mathscr{C}(i) \rightarrow \mathcal{D}(i)$ such that $F(\mathscr{C}_w(i)) = 0$ for all $w\in W, w\neq e$. Moreover, for all $u,v\in W,$ we have $\mathscr{C}_u(i)\mathscr{C}_v(i) \subset \mathscr{C}_{uv}(i)$ and $\mathscr{C}_u(i)^* = \mathscr{C}_{u^{-1}}(i)$. Hence
    $\mathscr{C}(i) = \overline{\bigoplus_{w\in W} \mathscr{C}_w(i)}$ is a topological $W$-grading for $\mathscr{C}(i)$. 
\end{proposition}
\begin{proof}
Let $\mathcal{P} \hookrightarrow B(H)$ be a faithful representation of $\mathcal{P}$ on a Hilbert space. Consider the induced embedding $\iota : \mathcal{L}_\mathcal{P}(\mathcal{X}) \hookrightarrow B(\mathcal{X}\otimes_\mathcal{P}H)$. We have \[ \mathcal{X}\otimes_\mathcal{P} H = \bigoplus_{w\in W}\mathcal{X}_i(w)\otimes_\mathcal{P} H.\] For all $w\in W$ let $p_{=w}$ be the orthogonal projection on the Hilbert subspace $\mathcal{X}_i(w)\otimes_{\mathcal{P}} H$ so that $(p_{=w})$ forms a family of orthogonal projections satisfying $\sum_{w\in W}p_{=w} = 1$ for the strong topology on $\mathcal{X}\otimes_\mathcal{P} H$. Define the conditional expectation $F : \iota(\mathscr{C}(i)) \rightarrow \iota(\mathcal{D}(i))$ by the formula \[F(T) = \sum_{w\in W}p_{=w}Tp_{=w}, \text{ for all } T\in \iota(\mathscr{C}(i)),\] where the sum converges in the strong topology. $F$ is idempotent and contractive. Hence by Tomiyama's theorem \cite[Theorem 1.5.10]{BO08}, $F$ is a conditional expectation onto its image. For every elementary operator $f \in \mathcal{E}_w(i)$ for $w\in W$, we have $F(\iota(f)) = 0$ if $w\neq e$ and $F(\iota(f)) = \iota(f)$ if $w = e$.  
This implies $F(\iota(\mathscr{C}_w(i))) = 0$ for all $w\neq e$ and $F(\iota(\mathscr{C}(i))) = \iota(\mathcal{D}(i))$. The two other formulas are obvious. Conclude by \cite[Theorem 3.3.]{Exe97}.
\end{proof}
\begin{remark}
    As the product of two path operators with arbitrary sources and targets is not necessarily a path operator, in general we have no way to understand the global covariance C$^*$-algebra $\mathscr{C}$. This is why we focus on the smaller C$^*$-subalgebras $\mathscr{C}(i) \subset \mathscr{C}.$
\end{remark}
We now assume that $S$ is finite. We also assume that for all $s\in S$ and $j \in \mathscr{I}$ the left action of $B_j$ on $\mathcal{X}_j(s)$ is faithful. By Proposition \ref{cw}, we know that $\Delta_0(i) \cong C(\overline{(W,S)})$. 
For all $i\in \mathscr{I},$ recall that $\psi_i : C(\overline{(W,S)}) \xrightarrow{\cong} \Delta_0(i)$ is the $*$-isomorphism sending each $P_w$ to $p_{i,w}$. Consider the left action of $W$ on $\overline{(W,S)}$ given by \cite[Theorem 3.3]{Kli23} and the induced left action $ \alpha :W \rightarrow \mathrm{Aut}(C(\overline{(W,S)}))$ defined for all $\phi \in C(\overline{(W,S)})$ and $w,u\in W$ by $\alpha_w(\phi)(u) = \phi(w^{-1}u)$ (by a slight abuse of notation, we will also write $\alpha$ for the induced action on some $\Delta_0(i)$). 
\begin{proposition}\label{CWDELTA}
Let $i,j \in \mathscr{I}$. Let $f \in C(\overline{(W,S)})$. Let $w\in W$. Let $T$ be a path operator from $i$ to $j$ of degree $w$. We have $T\cdot \psi_j(f) = \psi_i(\alpha_w(f)) \cdot T$. In particular, the elements of $\mathcal{D}(i)$ commute with the elements of $\Delta_0(i)$.
\end{proposition}
\begin{proof}
    It suffices to prove the formula for all $f = P_{u}$, $u\in W$. The element $\psi_i(\alpha_w(P_{u})) \in \Delta_0(i)$ is the projection on the submodules $\mathcal{X}_i(v)$ for $v\in W$ such that $w^{-1}v \geq u$. The result follows as $T$ sends each submodule $\mathcal{X}_j(v)$ to the submodule $\mathcal{X}_i(wv)$.
\end{proof}
The following lemma will be essential in our understanding of $\mathscr{C}(i)$. It shows that each elementary operator can be understood by its action on a specific submodule of the form $\mathcal{X}_i(p)$, for a $p\in W$. For every $T\in \mathcal{L}_\mathcal{P}(\mathcal{X})$ and submodule $Y\subset X_j$ we denote by $T|_Y :  Y \rightarrow X_j$ the restriction of $T$ to $Y$ (see Proposition \ref{box}). For all $u\in W$ we write \[D(u) = \{w\in W| w\geq u\}\] for the set of words of $W$ beginning with $u$. For every subset $D \subset W$, we let $P_D \in B(l^2W)$ denote the orthogonal projection onto the Hilbert subspace $l^2D$.
\begin{lemma}\label{TWW}
Let $i,j \in \mathscr{I}$. Let $T\in \mathscr{C}$ be a path operator from $i$ to $j$. There exist $u\in W$ and a subset $D \subset W$ (which should be thought of as the domain of the operator $T$) such that \begin{enumerate}[label=(\roman*)] 
    \item $D \subset D(u),$
    \item $D$ is a clopen subset of $\overline{(W,S)}$, in other words $P_{D} \in C(\overline{(W,S)})$,
    \item for all $w \in D$ and $v\in W,$ if $u \leq v \leq w$ then $v \in D$,
\end{enumerate} such that for all $w \in W$, if $w\notin D$ then $T|_{\mathcal{X}_j(w)} = 0$ and if $w\in D$ is written as $w = u\cdot v$ with $v\in W$ such that $l(w) = l(u)+l(v)$ then $T$ acts on each summand of the Hilbert module \[\mathcal{X}_j(uv) \cong \bigoplus_{k\in \mathscr{I}}X_{j,k}(u)\otimes_{B_k}\mathcal{X}_k(v)\] as $T|_{X_{j,k}(u)}\otimes \mathrm{id}_{\mathcal{X}_k(v)}$ (recall that $T(X_{j,k}(u)) \subset X_{i,k}(zu)$, $z$ being the degree of $T$).
\end{lemma}
\begin{proof}
We proceed by induction on the number $n$ of elementary operators appearing in the decomposition of $T$. Assume $n=1$. If $T$ is an elementary diagonal or annihilation operator of type $s$ then let $u=s$ and $D = D(s)$ so that $P_{D} = P_{s} \in C(\overline{(W,S)})$. If $T$ is a creation operator then let $u=e$ and $D = \{w\in W | l(sw) > l(w)\}$ so that $P_{D} = P_{s}^\perp \in C(\overline{(W,S)})$. 

Take $n\geq 1$ and assume the result holds for every path operator of length $n$.  Let $f_1,\ldots f_{n+1}$ be elementary diagonal, creation or annihilation operators, each $f_k$ being of type $(i_{k-1},i_k, s_k)$. Write $i_{n+1} = j$. Take $u\in W$ and $D\subset D(u)$ so that the lemma holds for the operator $T = f_1\ldots f_n$. We want to construct suitable $u'\in W$ and $D' \subset D(u')$ so that the lemma holds for the operator $Tf_{n+1} = f_1\ldots f_{n+1}$. Assume that $f_{n+1}$ is elementary diagonal. If $s_{n+1}\vee_Ru$ does not exist then $f_1\ldots f_{n+1} = 0$, otherwise  we claim that $u' = s_{n+1}\vee_Ru$ and $D' = D\cap D(s_{n+1})$ do the job. Indeed, we clearly have $D' \subset D(u')$ and $P_{D'} = P_{D}P_{s_{n+1}} \in C(\overline{(W,S)})$. Condition (iii) is also obvious. Let $w\in W$. If $w\notin D'$ then $Tf_{n+1}|_{\mathcal{X}_j(w)} = 0.$ Assume that $w\in D'$. Write $w = u' \cdot v$ for some $v \in W$ such that $l(w) = l(u') + l(v).$ Write also $u' = u\cdot t$ for some $t \in W$ such that $l(u') = l(u) + l(t)$. By applying condition (iii) of $D$, the inequality $u \leq u' \leq w \in D$ gives $u' \in D$. Hence the induction hypothesis gives \[T|_{X_{j,k}(u)\otimes \mathcal{X}_k(tv)} = T|_{X_{j,k}(u)}\otimes \mathrm{id}_{\mathcal{X}_k(tv)}\]
for all $k \in \mathscr{I}$. This proves the claim as we have \[f_{n+1}|_{X_{j,k}(s_{n+1})\otimes \mathcal{X}_k(s_{n+1}w) } = f_{n+1}|_{X_{j,k}(s_{n+1})}\otimes \mathrm{id}_{\mathcal{X}_k(s_{n+1}w)}.\]
Assume that $f_{n+1}$ is a creation operator. We can assume that $s_{n+1}\vee_Ru$ exists otherwise $f_1\ldots f_{n+1} = 0$. Then by the same reasoning as before, $u' = s_{n+1}(s_{n+1}\vee_Ru)$ and $D' = (s_{n+1}D)\setminus D(s_{n+1})$ do the job. Moreover $P_{D'} = P_{s_{n+1}D}P_{s_{n+1}}^\perp$ and $P_{s_{n+1}D} = \alpha_{s_{n+1}}(P_{D}) \in C(\overline{(W,S)})$. Assume that $f_{n+1}$ is an annihilation operator. We can assume that $s_{n+1}\cdot u$ is reduced otherwise $f_1\ldots f_{n+1} = 0$. For all $w\in W$ we have \[  \begin{cases}
w \geq s_{n+1} \text{ and} \\
s_{n+1} w \in D
\end{cases} 
\iff
\begin{cases}
w\geq s_{n+1}u \text{ and} \\
w\in s_{n+1}D
\end{cases} 
\] Hence $u' = s_{n+1}u$ and $D' = (s_{n+1}D)\cap D(u')$ do the job. 
\end{proof}
It turns out that operators of the form $a_1^\dagger (a_2^\dagger)^*$ for some $a_1,a_2 \in F_{i,j}(w)$ can be identified with compact operators on some specific Hilbert module. First let us recall a classical result. \begin{proposition}[{\cite[Proposition 4.6.3]{BO08}}]\label{BOBO}
Let $\mathcal{A}$ and $\mathcal{B}$ be two unital C$^*$-algebras. Let $\mathcal{H}$ be a Hilbert $\mathcal{A}$-module. Let $\pi\colon \mathcal{A}\to \mathcal{B}$ be a $*$-homomorphism and
$\tau\colon \mathcal{H}\to \mathcal{B}$ be a linear map such that
$
\tau(x)^*\tau(y)=\pi(\langle x,y\rangle) $ for every $x,y\in\mathcal{H}$. Then the $*$-homomorphism
$\sigma_\tau\colon \mathcal{K}_\mathcal{A}(\mathcal{H})\longrightarrow \mathcal{B},
$ defined by
\[
\sum_{i=1}^n \theta_{x_i,y_i}
\longmapsto
\sum_{i=1}^n \tau(x_i)\tau(y_i)^*,
\]
is continuous and satisfies $ \sigma_\tau(f)\tau(x)=\tau(fx)$ for every $f\in\mathcal{K}_{\mathcal{A}}(\mathcal{H})$ and $x\in\mathcal{H}$. Moreover, $\sigma_\tau$ is injective whenever $\pi$ is injective.
\end{proposition}
 Recall that there is a $*$-isomorphism $\mathcal{K}_{\mathcal{P}}(\mathcal{X}_i(w)) \cong \bigoplus^{c_0}_{j\in \mathscr{I}}\mathcal{K}_{B_j}(X_{i,j}(w)).$
\begin{proposition}
 Let $i \in \mathscr{I}$. Let $w\in W$. There is an injective $*$-homomorphism $\sigma_{i,w} : \mathcal{K}_{\mathcal{P}}(\mathcal{X}_i(w)) \rightarrow \mathscr{C}(i)$, given by \[\sigma_{i,w}(\theta_{a_1\eta_j,a_2\eta_k}) = \delta_{j,k} a_1^\dagger (a_2^\dagger)^*\] for all $a_1 \in F_{i,j}(w), a_2 \in F_{i,k}(w)$, and $j,k \in \mathscr{I}.$ 
 
 Let $i \in \mathscr{I}, w\in W$ and $f \in \mathcal{K}_{\mathcal{P}}(\mathcal{X}_i(w))$. Then for all $u\in W$ and $k \in \mathscr{I}$, $\sigma(f)|_{X_{i,k}(u)}$ is zero if $u\ngeq w$. If $u\geq w$ is of the form $u = w\cdot v$ with $l(u) = l(w) + l(v)$ then for all $l\in \mathscr{I}$, we have $\sigma(f)|_{X_{i,k}(w)\otimes X_{k,l}(v)} = f|_{X_{i,k}(w)}\otimes \mathrm{id}_{X_{k,l}(v)}.$
\end{proposition}
\begin{proof}
Let $\pi : \mathcal{P} \rightarrow \mathcal{L}_{\mathcal{P}}(\mathcal{X}) \cong \prod_{j \in \mathscr{I}}\mathcal{L}_{B_j}(X_j)$ be the injective $*$-homomorphism defined by $\pi((b_j)_{j\in \mathscr{I}}) = (b_jq_{j,w})_{j\in \mathscr{I}}$, where each $q_{j,w}$ is seen as an element of $\mathcal{L}_{B_j}(X_j)$. Let $i\in \mathscr{I}$. Let $w\in W.$ Define the linear map $\tau_{i,w} : \mathcal{X}_i(w) \rightarrow \mathscr{C}(i)$ by $\tau_{i,w} (a\eta_j) = a^\dagger = aq_{j,w}$ for all $a\in F_{i,j}(w)$ and $j\in \mathscr{I}.$ Let $j,k\in \mathscr{I}$ and $a_1\in F_{i,j}(w), a_2 \in F_{i,k}(w)$, Proposition \ref{QW} gives us \[\tau_{i,w}(a_1)^*\tau_{i,w}(a_2) = (a_1^\dagger)^*a_2^\dagger = \delta_{j,k}E_j(a_1^*a_2)q_{k,w}.\] The result then follows from Proposition \ref{BOBO}. The second part of the proposition is easily checked on each operator of the form $a_1^\dagger (a_2^\dagger)^*$.
\end{proof}
This allows us to prove the following fundamental result.
\begin{proposition}\label{pw}
    Assume that the expected $W$-decomposition $A = \overline{\sum_{w\in W}L_{i,j}(w)}$ is of finite type. Then for all $i\in \mathscr{I}$ and $w\in W$, we have $p_{i,w}\in \mathscr{C}(i).$ Hence we have $\Delta_0(i),\Delta(i) \subset \mathscr{C}(i).$
\end{proposition}
\begin{proof}
    As the decomposition is of finite type, we know that $\mathrm{id}_{\mathcal{X}_i(w)} \in \mathcal{K}_{\mathcal{P}}(\mathcal{X}_i(w))$. This implies $p_{i,w} = \sigma_{i,w}(\mathrm{id}_{\mathcal{X}_i(w)}) \in \mathscr{C}(i).$
\end{proof}
\begin{remark}
In the case of a finite type expected Bruhat decomposition, $\mathscr{C} \subset \mathcal{L}_B(\mathcal{X})$ is just the C$^*$-algebra generated by the reduced C$^*$-algebra $A^r$ together with the projections $p_s$ for $s\in S$. Assume that $W$ is right-angled and the decomposition comes from a graph product C$^*$-algebra structure on $A$ as in Example \ref{graphproductexample}.  $\mathscr{C}$ thus coincides with the universal C$^*$-algebra introduced in \cite{Kli25} for graph product C$^*$-algebras, in the particular case of finite-dimensional vertex algebras.
\end{remark}
Whenever $A = \overline{\sum_{w\in W}L_{i,j}(w)}$ is a $W$-decomposition of a unital C$^*$-algebra $A$ and there may be an ambiguity on the underlying C$^*$-algebra, we denote by $\mathscr{C}(A)$ and $\mathscr{C}(A)(i)$ respectively the global covariance C$^*$-algebra and the covariance C$^*$-algebra based at $i$ associated to the decomposition of $A$.
\begin{proposition}\label{covtensor}
     Let $Q$ be a unital C$^*$-algebra. Let $A \otimes Q = \overline{\sum_{w\in W}G_{i,j}(w)}$ together with conditional expectations $E_i\otimes \mathrm{id}_Q : A\otimes Q \rightarrow B_i\otimes Q$ be the expected $W$-decomposition given by Proposition \ref{tensor}. There is a natural $*$-isomorphism $\mathscr{C}(A \otimes Q)(k) \cong  \mathscr{C}(A)(k)\otimes Q$ for all $k\in \mathscr{I}.$
\end{proposition}
\begin{proof}
Recall from Proposition \ref{tensor} that the right $B_j\otimes Q$-modules associated to the decomposition $A \otimes Q = \overline{\sum_{w\in W}G_{i,j}(w)}$ are given by $Y_j = \bigoplus_{w\in W}Y_{i,j}(w)$ for all $j\in \mathscr{I}$, with $Y_{i,j}(w) \cong X_{i,j}(w) \otimes_{\mathrm{ext}}Q$. Let \[\mathcal{P}^Q
=
\prod_{i\in\mathscr{I}}(B_i\otimes Q),\qquad \mathcal{Y} = \bigoplus_{j\in \mathscr{I}}Y_j\otimes_{B_j\otimes Q} (\mathcal{P}^Q),\] so that $\mathcal{L}_{\mathcal P^Q}(\mathcal Y) \cong \prod_{j\in \mathscr{I}}\mathcal{L}_{B_j\otimes Q}(Y_j).$ Under this identification, $\mathscr{C}(A \otimes Q)(k)$ is the C$^*$-subalgebra of $\mathcal{L}_{\mathcal P^Q}(\mathcal Y)$ generated by path operators based at $k$ being the composition of creation and elementary diagonal operators respectively of the form $a^\dagger \otimes q$ and $\partial(a)\otimes q$ for some $a\in L_{i,j}(s)$ and $q\in Q$.  Hence it is just the image of $ \mathscr{C}(A)(k)\otimes Q$ by the embedding \[
\mathcal{L}_{\mathcal{P}}(\mathcal{X})\otimes Q
\longrightarrow
\prod_{j\in\mathscr{I}}
\mathcal{L}_{B_j\otimes Q}
\bigl(X_j\otimes_{\mathrm{ext}}Q\bigr)
\cong
\mathcal{L}_{\mathcal{P}^Q}(\mathcal{Y}).
\]
\end{proof}
\subsection{The case of a group}
Let $X$ be a locally finite building. We assume that $S$ is finite (hence $X$ is necessarily countable). We see $X$ as a graph in the sense of Section \ref{31} whose set of vertices is $X$ and such that two vertices $x,y \in X$ are adjacent if $\delta(x,y) \in S$. For all $o\in X$, we denote by $\Omega(o)$ the bordification of the rooted graph $(X,o).$
\begin{lemma}\label{convW} Let $o \in X$. Let $\mathbf{x} =(x_n)$ be a sequence with values in $X$ which $o$-converges. Then the sequence $\delta(o,x_n)$ converges in the rooted graph $(W,e)$. Moreover, $(x_n)$ $o$-converges to infinity if and only if $\delta(o,x_n)$ converges to infinity in $\overline{(W,S)}$. \end{lemma} \begin{proof} Let $w\in W$. For all $z\in X$, $\delta(o,z) \geq w$ if and only if there exists a geodesic in $X$ starting at $o$, passing by some chamber $y \in X$ such that $\delta(o,y) = w$ and ending at $z$. Hence $\delta(o,x_n) \geq w$ for big enough $n$ if and only if for all big enough $n$, there exists a $y_n \in X$ such that $\delta(o,y_n) = w$ and $x_n \geq_o y_n$. However, as $X$ is locally finite, there are only finitely many $y \in X$ such that $\delta(o,y) = w$. Hence, as $\mathbf{x}$ is $o$-convergent, $\delta(o,x_n) \geq w$ for big enough $n$ if and only if there exists a $y\in X$ such that $\delta(o,y) = w$ and such that $\mathbf{x} \geq_o y$. Assume that this is not the case, this means that for all $y\in X$ such that $\delta(o,y) = w$ we have $\mathbf{x} \ngeq_o y$. This implies $\delta(o,x_n) \ngeq w$ for big enough $n$. \end{proof} We endow $X$ with the graph metric $d$. Notice that we have $d(x,y) = l(\delta(x,y))$ for all $x,y\in X$. Moreover, for all $o,x,y\in X,$ we have \[x\geq_oy \iff d(o,x) = d(o,y) + d(y,x).\] In other words, $x\geq_oy$ if and only if the expression $\delta(o,y)\cdot \delta(y,x)$ is reduced in $W$, in which case we have $\delta(o,x) = \delta(o,y)\delta(y,x).$  For all $o,y \in X$ we define \[
U(o,y)
=
\{x\in X\mid x\geq_o y\}.
\] Recall that the Boolean algebra generated by a family of subsets of $X$ is the set of all subsets of $X$ obtained by taking arbitrary finite intersection, union and complements of these subsets. For all $u,v \in W$ we will say that $u \cdot v$ is reduced if $l(uv) = l(u) + l(v)$.
\begin{lemma}\label{convconv}
Let $o,o' \in X$. Let $y\in X$. $U(o',y)$ belongs to the Boolean algebra generated by the subsets of the form $U(o,z)$ for $z\in X$. In other words, the condition $x \geq_{o'}y$, on an arbitrary $x\in X$, can be formulated using a finite combination of conditions of the form $x \geq_o z$ for some $z\in X$.
\end{lemma} 
\begin{proof}
   As any two chambers are connected by a minimal gallery, we may assume that $\delta(o,o') = s \in S$. We write
    \[
        a=\delta(o,y),
        \qquad
        b=\delta(o',y),
        \qquad
        c=\delta(y,x),
        \qquad
        w=\delta(o,x).
    \]
Moreover, there are two equivalences
    \[
        x\geq_o y
        \quad\Longleftrightarrow\quad
        a\cdot c \text{ is reduced}, \qquad
        x\geq_{o'}y
        \quad\Longleftrightarrow\quad
        b\cdot c \text{ is reduced}.
    \]
By (Bui2), we know that $b\in\{a,sa\}$. We distinguish three cases. If $b=a$, then $b\cdot c$ is reduced if and only if $a\cdot c$ is reduced. Hence we are done. 

Suppose that $b=sa$ and $l(b)=l(a)+1$. In that case we know that $b\cdot c$ is reduced if and only if $a\cdot c$ is reduced and $s\nleq w$. Moreover the condition $w \geq s$ is equivalent to asking that there is a $z \in X$ such that $\delta(o,z) = s$ and $x\geq_o z$. As $X$ is locally finite there are only a finite number of such $z$ so we are done.

    It remains to consider the case when $b=sa$ and $l(b)=l(a)-1$. Equivalently, $a=s\cdot b$ and this product is reduced. Suppose first that $s$ and $b$ have no common upper bound for the right weak order. In that case we claim that $b\cdot c$ is reduced if and only if $a\cdot c$ is
    reduced. One implication is immediate. Suppose that $b\cdot c$ is reduced but $a\cdot c$ is not
    reduced. Then $bc$ is a common upper
    bound of $s$ and $b$, a contradiction.
    
    Assume now that $s$ and $b$ have a common upper bound, and let
    $r=s\vee b$. Write $r=bd$, for some $d \in W$ such that $l(r)=l(b)+l(d)$.
    Consider the sets
    \[
        F_0=
        \left\{
            z\in X :
            z\geq_{o'}y,\ 
            \delta(o',z)=r,\ 
            \delta(o,z)=r
        \right\}
    \]
    and
    \[
        F_1=
        \left\{
            z\in X :
            z\geq_{o'}y,\ 
            \delta(o',z)=r,\ 
            \delta(o,z)=sr
        \right\}.
    \]
    These sets are finite because $X$ is locally finite. We claim that
    \[
    \begin{split}
        x\geq_{o'}y
        \quad\Longleftrightarrow\quad&
        x\geq_o y\\
        &\text{or there exists }z\in F_0\text{ such that }x\geq_o z\\
        &\text{or }
        \left(
            w \ngeq s
            \text{ and there exists }z\in F_1
            \text{ such that }x\geq_o z
        \right).
    \end{split}
    \]

    Suppose first that $x\geq_{o'}y$ and that $x\ngeq_o y$. Since $b\cdot c$ is reduced and $s\cdot (bc)$ is not reduced we have $bc \geq r=s\vee b$. Write then $bc=rf$ for some $f \in W$ such that $l(bc) = l(r) + l(f)$. The reduced expressions $r=bd$ and
    $rf = bc$, give
    $c=df$ with $l(c) = l(d) + l(f)$. Choose a chamber $z$ on a geodesic from $y$ to $x$ such that $\delta(y,z)=d$. Then
    $z\geq_{o'}y$, $\delta(o',z)=r$, and $\delta(z,x)=f$. By (Bui2), we know that $\delta(o,z)\in\{r,sr\}$. If
    $\delta(o,z)=r$, then $r\cdot f = \delta(o,z)\cdot \delta(z,x)$ is reduced, so $x\geq_o z$ and
    $z\in F_0$. Assume now $\delta(o,z)=sr$. Then as $(sr)\cdot f$ is reduced, we have $x\geq_o z$ and $z\in F_1$. Moreover, as
    $r\cdot f$ is reduced we know that
    $w=\delta(o,x)=srf \ngeq s$.

    Conversely, if $x\geq_o y$, then $a\cdot c$ is reduced, and
    therefore $b\cdot c$ is reduced. Hence $x\geq_{o'}y$. Suppose that $z\in F_0$ and $x\geq_o z$. Put
    $f=\delta(z,x)$. Since $\delta(o,z)=r$, the product $r\cdot f$ is reduced. As $\delta(o',z)=r$, it follows that $x\geq_{o'}z$. Since $z\geq_{o'}y$, we obtain $x\geq_{o'}y$. Finally, suppose that $z\in F_1$, that $x\geq_o z$, and that $w\ngeq s$. Put again $f=\delta(z,x)$. Then $w=(sr)\cdot f$ is reduced. Since $w\ngeq s$, the products $s\cdot w=r\cdot f$ are reduced. As $\delta(o',z)=r$, this implies $x\geq_{o'}z$, and therefore $x\geq_{o'}y$. This proves the claim.
\end{proof}
A direct consequence of the last lemma is the following proposition.
\begin{proposition}\label{omega}
    Let $o,o' \in X$. A sequence in $X$ $o$-converges if and only if it $o'$-converges. Hence the underlying sets of $\Omega(o)$ and $\Omega(o')$ are equal. Moreover the open sets of $\Omega(o)$ and $\Omega(o')$ are the same. Hence the topological spaces $\Omega(o)$ and $\Omega(o')$ are equal.
\end{proposition}

In the rest of this article we will thus denote $\Omega(o)$ simply by $\Omega$ for any $o \in X$. Moreover, given a sequence $\mathbf{x}$ with values in $X$ we will simply say that it converges if it $o$-converges for some $o \in X$. Fix a discrete group $G$ acting on $X$ by $W$-isometries. We let $\mathscr{I} = G\backslash X$ denote the set of orbits and $G = \biguplus_{w\in W}C_{i,j}(w)$ the associated $W$-decomposition. Denote by $H_i = C_{i,i}(e)$ the Borel subgroups. We now follow the approach of Klisse \cite[Theorem 3.3, Theorem 3.5]{Kli23}.
\begin{proposition}\label{toptop}
The action of $G$ on $X$ extends to a continuous action on $\Omega$ by homeomorphisms. 
\end{proposition}
\begin{proof}
Let $g\in G.$ Let $\mathbf{x},\mathbf{y}$ be two sequences with values in $X$ converging to the same limit in $\Omega$. Let $z\in X$. As $G$ acts by $W$-isometries, for all $n$ we know that $gx_n \geq_o z$ if and only if $x_n \geq_{g^{-1}o} g^{-1}z$ and $gy_n \geq_o z$ if and only if $y_n \geq_{g^{-1}o} g^{-1}z$. Hence the sequences $(gx_n)$ and $(gy_n)$ converge to the same limit in $\Omega(g^{-1}o)$. By the last proposition we conclude that $(gx_n)$ and $(gy_n)$ also converge to the same limit in $\Omega$. We now show that the extension of the map $x\mapsto g\cdot x$ to $\Omega$ is continuous. As $\Omega$ is metrizable it suffices to check continuity for sequences. Let $(z^n)$ be a sequence with values in $\Omega$ converging to some $z \in \Omega$. For all $n$ we take a sequence $\mathbf{x}^n=(x^n_k)_{k\geq 0}$ with values in $X$ representing $z^n$. Let $y \in X$. For all $n$ we know that $gz^n \geq_o y$ if and only if $z^n \geq_{g^{-1}o} g^{-1}y$. By the last proposition, we know that $(z^n)$ converges to $z$ in $\Omega(g^{-1}o)$. As $z \geq_{g^{-1}o}g^{-1}y$ if and only if $gz \geq_o y$, we conclude that $(gz^n)$ converges to $gz.$
\end{proof}
Let $(Y,d)$ be a metric space. Let $y_0 \in Y$ be a base point. Consider the Banach space $C_*(Y) = C(Y,\mathbb{R})/\mathbb{R}\cdot 1_Y$ and equip it with the topology coming from the topology of uniform convergence on bounded sets on $C(Y,\mathbb{R})$. For all $f\in C(Y,\mathbb{R}),$ we let $[f] \in C_*(Y)$ be the class of $f$ modulo the constant functions. Consider the (continuous and injective) map defined by \[
\begin{aligned}
\iota\colon Y &\longrightarrow C_*(Y), \\
y &\longmapsto [f_y],
\end{aligned}
\qquad
\text{where for all } z\in Y, f_y(z)= d(y,z).
\]
\begin{definition}
The horofunction compactification $\hat{Y}$ of $Y$ is the closure of the image of $\iota$ inside $C_*(Y)$.
\end{definition}
One checks that $\hat{Y}$ is compact Hausdorff whenever $Y$ is a proper metric space (see \cite[Chapter II.8]{BH99} for more details). The building $X$ is a metric space with the graph distance (i.e. the distance given by the length of $\delta$ in $W$). Moreover it is proper as it is locally finite.
\begin{theorem}\label{capraceklisse} The map $\iota : X \rightarrow C_*(X)$ extends to a $G$-equivariant homeomorphism $\Omega\cong \hat{X}$ between the bordification $\Omega$ of $X$ seen as a rooted graph and the horofunction compactification of $X$ seen as a metric space.
\end{theorem}
\begin{proof}
As $\Omega$ is a compact space and the image of $\iota$ is dense in $\hat{X}$, it suffices to prove that $\iota$ extends to a continuous injective map $\Omega\rightarrow \hat{X}$ defined by $\iota(\mathbf{x}) = [z\mapsto \lim_{n}(d(x_n,z)-d(x_n,o))] \in C_*(X)$ for all $\mathbf{x} \in \Omega$ and $z\in X$. As $X$ is a locally finite graph the convergence on bounded subsets is equivalent to pointwise convergence.

Let us first check that $\iota$ extends to $\Omega$. Let $\mathbf{x}$ and $\mathbf{y}$ be two sequences converging to the same limit. Let $o \in X$. Let $z\in X$. We claim that the sequences $d(x_n,z)-d(x_n,o)$ and $d(y_n,z)-d(y_n,o)$ converge to the same limit when $n$ goes to infinity. As there is a minimal gallery between $o$ and $z$, it suffices to prove that for all $z_1,z_2\in X$ such that $\delta(z_1,z_2) = s \in  S$ the sequences $d(x_n,z_1)-d(x_n,z_2)$ and $d(y_n,z_1)-d(y_n,z_2)$ converge to the same value. We know that for all $n$, $d(x_n,z_1)-d(x_n,z_2)$ and $d(y_n,z_1)-d(y_n,z_2)$ belong to $\{-1,0,1\}$. Moreover, $d(x_n,z_1)-d(x_n,z_2) = d(z_1,z_2) = 1$ if and only if $x_n \geq_{z_1} z_2$ and $d(x_n,z_1)-d(x_n,z_2) = -1$ if and only if $x_n \geq_{z_2} z_1$. The same equivalences hold for $\mathbf{y}$ so the claim follows by Lemma \ref{convconv}.

We now prove the continuity of $\iota$. As $\Omega$ is metrizable it suffices to check continuity on sequences $\mathbf{x}^n \in \Omega$ converging to $\mathbf{x}\in \Omega$. This means that for all $o,z \in X$, we have $\mathbf{x} \geq_o z$ if and only if for big enough $n$ we have $\mathbf{x}^n \geq_o z$. The reasoning of the last paragraph proves $\iota(\mathbf{x})(z) = \lim_{n}\iota(\mathbf{x}^n)(z)$ for all $z\in X$. 

It remains to prove the injectivity of $\iota$. Let $\mathbf{x}\neq\mathbf{y} \in \Omega$. There exists a $z\in X$ such that $\mathbf{x}\geq_oz$ and $\mathbf{y}\ngeq_oz$. For big enough $n$ we have $d(x_n,o) - d(x_n,z) = d(z,o) >  d(y_n,o) - d(y_n,z).$ Hence $[f_{x_n}] \neq [f_{y_n}]$ in $C_*(X).$
\end{proof}
By \cite[Theorem 3.1]{CL11}, the last theorem tells us that $\Omega$ is $G$-equivariantly homeomorphic to the minimal combinatorial compactification of $X$.
Let us choose a base point $i_0\in \mathscr{I}$. We now use the $G$-equivariant identification $X = \coprod_{i\in \mathscr{I}}G/H_i $ as in Theorem \ref{CTB}. Consider the connected rooted graph $(X,H_{i_0})$ with root $o = H_{i_0}$. Let us describe algebraically the graph order on $X$. For all $i,j \in \mathscr{I}$ and $g\in G$ we denote by $[g]_{i,j}$ the only element $w\in W$ such that $g\in C_{i,j}(w)$. For all $x = gH_j \in X$ with $g \in C_{i_0,j}(w),$ we denote by $|x|_{i_0}$ the length of $[g]_{i_0,j}$, which is the distance in $X$ between $x$ and the origin $H_{i_0}$.

\begin{proposition}\label{order}
    Let $i_0 \in \mathscr{I}$. Let $o = H_{i_0}$ be the associated root. Let $i,j \in \mathscr{I}$. Let $g_1,g_2 \in G$. Let $u = [g_1]_{i_0,i} \in W$ and $v = [g_2]_{i_0,j} \in W$. Then $g_1H_i \geq_{o} g_2H_j$ for the graph order on $(X,o)$ if and only if $u \geq v$ for the right weak order and there exists $g_3 \in C_{j,i}(v^{-1}u)$ such that $g_1 = g_2g_3$.
\end{proposition}
\begin{proof}
    The graph distance between $g_1H_i$ and $g_2H_j$ equals the length of $[g_1^{-1}g_2]_{i,j} \in W$. Assume first that $u \geq v$ for the right weak order and there exists a $g_3 \in C_{j,i}(v^{-1}u)$ such that $g_1 = g_2g_3$. Take reduced expressions $u = s_1\ldots s_n, v = s_1\ldots s_p$ and $v^{-1}u = s_{p+1} \ldots s_n$ and take the corresponding decompositions $g_2 = z_1\ldots z_p, g_3 = z_{p+1} \ldots z_n$ for $z_k \in C_{i_{k-1},i_{k}}(s_k), i_1\ldots i_n \in \mathscr{I}$ and $s_1 \ldots s_n\in S$ (in particular $z_1 \in C_{i_{0},i_{1}}(s_1)$ and $z_p \in C_{i_{p-1},j}(s_p)$). Define $x_0 = H_{i_0}$ and $x_k = z_1\ldots z_kH_{i_k}$ for all $1 \leq k \leq n$. The finite sequence $\mathbf{x} = (x_n)$ is the desired geodesic starting at $H_{i_0}$, passing by $g_2H_j$ and ending at $g_1H_i$. 
    
    Assume that $g_1H_i \geq_o g_2H_j$ for the graph order at $o$. This means that there is a geodesic $x_0, \ldots x_n$ in $X$ starting at $x_0 = H_{i_0}$, passing by $x_p = g_2H_j$ and ending at $x_n = g_1H_i$. For all $0 \leq k \leq n,$ write $x_k = g_k'H_{i_k}$ for some $g_k' \in C_{i_0,i_k}(w_k), i_k \in \mathscr{I}, w_k \in W.$ In particular we choose $g_0' = e, g_p' = g_2$ and $g_n' = g_1$. For all $0\leq k \leq n-1$, $d(x_k,x_{k+1}) = 1$ implies that $g_k'^{-1}g'_{k+1} \in C_{i_k,i_{k+1}}(s_k)$ for some $s_k\in S$.
    Write $z_k = g_k'^{-1}g'_{k+1}$ (in particular $z_0 = g_1'$). We have $x_k = z_0z_1\ldots z_{k-1} H_{i_k}$. Moreover as $x_0, \ldots, x_n$ is a geodesic we know that for all $k$ the expression $s_0\ldots s_k$ is reduced. In particular there are indices $p $ and $ n$ such that $p\leq n$ and $g_1H_i = z_0\ldots z_{n-1} H_{i_n}, u = s_0\ldots s_{n-1}$ and $g_2H_j = z_0\ldots z_{p-1}H_j, v = s_0\ldots s_{p-1}.$
\end{proof}
\begin{lemma}\cite[Proposition 3.1.1]{Has17}\label{FELLABS}
    Let $H \leq G$ be a pair of discrete groups. Consider the usual conditional expectation $E : C^*_r(G) \rightarrow C^*_r(H)$. Write $\pi : G \rightarrow \mathcal{U}(l^2(G/H))$ for the quasi-regular representation. For all $a\in C^*_r(G),$ write  $L(a) : b \rightarrow ab \in \mathcal{L}_{C^*_r(G)}(C^*_r(G))$. There is an isomorphism of $C^*_r(G)$-$C^*_r(G)$ correspondences \[L^2(C^*_r(G),E)\otimes_{C^*_r(H)} C^*_r(G) \cong l^2(G/H)\otimes C^*_r(G).\]
   Here the left action of each $\lambda(g)\in C^*_r(G)$ on $l^2(G/H)\otimes C^*_r(G)$ is given by $\pi(g)\otimes L(\lambda(g))$.
\end{lemma}
\begin{proof}
One checks that the operator $U :l^2(G/H)\otimes C^*_r(G)
\longrightarrow L^2(C^*_r(G),E)\otimes_{C^*_r(H)} C^*_r(G)$ given by 
$U(\delta_{gH}\otimes a)
= \lambda(g)\xi\otimes \lambda(g)^*a$
for $g\in G$ and $a\in C^*_r(G)$ is well-defined and gives the desired unitary equivalence.
\end{proof}
Let $A = C^*_r(G) = \overline{\bigoplus_{w\in W}F_{i,j}(w)}$ be the reduced C$^*$-algebra of $G$ together with its associated expected $W$-decomposition as in Example \ref{GOP}. Denote by $B_i = C^*_r(H_i)$ the Borel C$^*$-subalgebras.
\begin{theorem}\label{convcrossed}
The covariance C$^*$-algebra based at $i$ is independent of the choice of $i$ inside $\mathcal{L}_\mathcal{P}(\mathcal{X})$ and equals the global covariance C$^*$-algebra $\mathscr{C}$. Moreover there is a $*$-isomorphism $\mathscr{C} \cong C(\Omega) \rtimes_rG$.
\end{theorem}
\begin{proof}
    For all $i\in \mathscr{I},$ write $\pi_i : G \rightarrow \mathcal{U}(l^2(G/H_i))$ for the quasi-regular representation. Write $\pi =\bigoplus_{j\in \mathscr{I}}\pi_j: G \rightarrow \mathcal{U}(\bigoplus_{j\in \mathscr{I}}l^2(G/H_j))$. Let $\lambda : C^*_r(G) \rightarrow B(l^2G)$ be the left regular representation of $G$. Let $i\in \mathscr{I}$. Consider the embedding \[\mathcal{L}_\mathcal{P}(\mathcal{X}) \cong \prod_{j\in \mathscr{I}}\mathcal{L}_{B_j}(X_j) \hookrightarrow \prod_{j\in \mathscr{I}}\mathcal{L}_A(X_j\otimes_{B_j}A) \cong \prod_{j\in \mathscr{I}}\mathcal{L}_A(l^2(G/H_j)\otimes A)\] where the first embedding is given by Lemma \ref{ZERO} and the isomorphism is given by the last lemma. Embed each $\mathcal{L}_A(l^2(G/H_j)\otimes A)$ into $B(l^2(G/H_j)\otimes l^2G)$ as in Lemma \ref{ZERO}, using faithfulness of $\lambda$. We denote by $\sigma : \mathcal{L}_\mathcal{P}(\mathcal{X}) \hookrightarrow B(l^2X\otimes l^2G)$ the embedding obtained by composing the one above with the diagonal embedding \[\prod_{j\in \mathscr{I}}B(l^2(G/H_j)\otimes l^2G) \hookrightarrow B(\bigoplus_{j\in \mathscr{I}}l^2(G/H_j)\otimes l^2G) = B(l^2X\otimes l^2G).\] Let $i \in \mathscr{I}$ be a choice of base index and choose $o = H_i\in X$ as root of $X$. We are going to prove the inclusions $C(\Omega)\rtimes_rG \subset \sigma(\mathscr{C}(i)) \subset \sigma(\mathscr{C}) \subset C(\Omega)\rtimes_rG$. The C$^*$-subalgebra $C(\Omega) \subset B(l^2X)$ is generated by the projections $P_{i,x}$ for $x\in X$ as in Proposition \ref{CT}. Here, for all $y\in X,$ $P_{i,y}$ is the projection on the Hilbert space whose basis is the subset \[
U(o,y)
=
\{x\in X\mid x\geq_o y\}.
\] By Lemma \ref{convconv}, this C$^*$-subalgebra is independent of the choice of root. Consider the embedding $C(\Omega) \rtimes_rG \subset B(l^2X\otimes l^2G)$. Inside $B(l^2X\otimes l^2G)$, the C$^*$-algebra $C(\Omega)\rtimes_rG$ is generated by elements of the form $\pi(g) \otimes \lambda(g)$ and the projections $P_{i,gH_j}\otimes \mathrm{id}_G$ for $g\in G$ and $j\in \mathscr{I}$. The C$^*$-algebra $\sigma(\mathscr{C}(i)) \subset B(l^2X\otimes l^2G)$ is generated by the image by $\sigma$ of all path operators based at $i$. For all $j\in \mathscr{I}$ and $w\in W$ we let $p_{j,w}' \in B(l^2X)$ be the projection onto the Hilbert subspace of $l^2X$ spanned by elements of the form $\delta_{gH_k}$ for $g\in C_{j,k}(u), k \in \mathscr{I}$ and $u\geq w$.  Easy computations prove that \[\sigma(p_{j,w}) = p_{j,w}'\otimes \mathrm{id}_{l^2G},\qquad \sigma (\lambda(g)) = \pi(g) \otimes \lambda(g).\] for all $j\in \mathscr{I}, w\in W$ and $g\in G.$ We know that $A^r \subset \mathscr{C}(i)$. Hence $\sigma(\mathscr{C}(i))$ contains elements of the form $\pi(g)\otimes\lambda(g),$ $g\in G$. Let $g \in C_{i,j}(w)$ for some $w\in W, j \in \mathscr{I}$. Take $w = s_1\ldots s_n$ a reduced expression of $w$. Write $g = g_1\ldots g_n$ for $g_k \in C_{i_{k-1},i_k}(s_k), 1\leq k \leq n$ and $i_0 = i$ and $i_n = j$. Proposition \ref{order} gives the formula \[P_{i,gH_j}\otimes \mathrm{id}_{l^2G} = \sigma(\lambda(g_1)^{\dagger})\ldots \sigma(\lambda(g_n)^{\dagger})\sigma(\lambda(g_n)^{\dagger})^*\ldots \sigma(\lambda(g_1)^{\dagger})^*.\] This implies $P_{i,gH_j}\otimes \mathrm{id}_{l^2G} \in \sigma(\mathscr{C}(i))$ hence $C(\Omega)\rtimes_rG \subset \sigma(\mathscr{C}(i)).$

We now prove $\sigma(\mathscr{C}) \subset C(\Omega)\rtimes_rG$. We know that $\sigma(\mathscr{C})$ is in the C$^*$-subalgebra generated by $\sigma(C^*_r(G))$ and the projections $p_{i,w}'\otimes \mathrm{id}_{l^2G}$ for arbitrary $i \in \mathscr{I}$ and $w\in W$. As $C(\Omega)\rtimes_rG$ already contains the copy of $C^*_r(G)$ it remains to prove that $C(\Omega)$ contains all the projections $p_{i,w}'$. Let $i\in \mathscr{I}$ and $w\in W$. As $X$ is locally finite, we know that there are only finitely many chambers at distance $w$ of $H_i$. Using Proposition \ref{order}, one may thus write \[
p_{i,w}'
=
\bigvee_{\substack{
x\in X\\
\delta(H_i,x) = w
}}
P_{i,x}.
\] Hence $p_{i,w}' \in C(\Omega)$. By Proposition \ref{omega}, this reasoning can be done for any $i \in \mathscr{I}$. This finishes the proof.
\end{proof}
\begin{example}
Let $C^*_{r,q}(W) = \overline{\bigoplus_{w\in W}\mathbb{C}\cdot T_w}$ be a multiparameter Iwahori-Hecke C$^*$-algebra as in Example \ref{Bruhathecke}. The canonical state $\phi : C_{r,q}(W) \rightarrow \mathbb{C}$ makes this an expected Bruhat decomposition whose Fock space is $l^2W$. We have a simple description of the covariance C$^*$-algebra in this context as $\mathscr{C} \cong C(\overline{(W,S)})\rtimes_r W$. In particular, $\mathscr{C}$ is independent of the deformation parameter.
\end{example}
\section{Understanding the covariance \texorpdfstring{C$^*$-algebra}{C*-algebra}}\label{section4}

\subsection{Universality of the covariance \texorpdfstring{C$^*$-algebra}{C*-algebra}}
We fix a Coxeter system $(W,S)$. We assume that $S$ is finite. Let $A$ be a unital C$^*$-algebra with a finite type expected $W$-decomposition $A = \overline{\sum_{w\in W}L_{i,j}(w)}$ of index set $\mathscr{I}$ with Borel C$^*$-subalgebras $B_i$ and conditional expectations $E_i : A \rightarrow B_i$. Let $A^r = \overline{\sum_{w\in W}F_{i,j}(w)} \subset \mathcal{L}_\mathcal{P}(\mathcal{X})$ be the corresponding reduced C$^*$-algebra. Throughout this section we assume that the left action of $B_i$ on $\mathcal{X}_i(s)$ is faithful for any $i \in \mathscr{I}$.

Let $i\in \mathscr{I}$. Our goal is to prove that $\mathscr{C}(i)$ has a universal property that can be expressed in terms of the so-called covariant representations of the decomposition. For all $u,v\in W, i,j,k \in \mathscr{I}$, such that $l(uv) = l(u)+l(v),$  write $\phi_{i,j,k}^{u,v} : X_{i,j}(u)\otimes_{B_j} X_{j,k}(v)\hookrightarrow X_{i,k}(uv)$ for the embedding given by Proposition \ref{fock}. This embedding is defined by the formula  $\phi_{i,j,k}^{u,v}(a\eta_j \otimes b\eta_k) = ab\eta_k$ for $a\in F_{i,j}(u),b\in F_{j,k}(v)$. For all $x\in X_{i,j}(u), y \in X_{j,k}(v)$ define the product $xy\in X_{i,k}(uv)$ by the formula \[xy = \begin{cases}
    \phi_{i,j,k}^{u,v}(x\otimes y) \text{ if } l(uv) = l(u)+l(v), \\
    0 \text{ otherwise.}
\end{cases}\] This product is associative as it is associative on the elements of the form $a\eta_j, a\in F_{i,j}(w), i,j\in \mathscr{I},w\in W$. In the situation $u = e, i = j$ (resp. $v = e, j=k$), the map $\phi_{i,i,k}^{e,v}$ (resp. $\phi_{i,j,j}^{u,e}$) implements the left (resp. right) action of $B_i$ (resp. $B_j$) on $X_{i,k}(v)$ (resp. $X_{i,j}(v)$). If $u=v=e$ and $i=j$ then this product is just the product of $B_i \cong X_{i,i}(e)$.

Recall from  \cite[Theorem 2.13]{Kli23} that $C(\overline{(W,S)}) \cong \mathcal{D}(W,e)$ is the universal C$^*$-algebra generated by projections $P_w$ such that for all $u,v \in W,$ we have $P_u P_v =
\begin{cases}
P_{u \vee v} & \text{if } u \vee v < \infty, \\
0 & \text{otherwise}.
\end{cases}$
For all $w\in W$ we denote $Q_w \in C(\overline{(W,S)})$ the projection on the clopen set consisting of those elements $x\in \overline{(W,S)}$ such that $x$ does not begin with any of the elements $w_i \in W$ given by Lemma \ref{qwqw}, i.e. $Q_w = P_{w_1}^\perp \ldots P_{w_n}^\perp$.
\begin{definition}\label{covariantrep}
    A representation of the expected $W$-decomposition on a unital C$^*$-algebra $C$ is the data of linear maps $\rho_{i,j}^w : X_{i,j}(w) \rightarrow C$ for all $i,j \in \mathscr{I}, w\in W$ together with unital $*$-homomorphisms $\psi_i: C(\overline{(W,S)}) \rightarrow C$ for all $i\in \mathscr{I}$ such that \begin{enumerate}[label=(\roman*)]
    \item $\forall i,j,k\in \mathscr{I}, \forall u,v\in W, \forall x\in X_{i,j}(u), \forall y\in X_{j,k}(v), \rho_{i,k}^{uv}(xy) = \rho_{i,j}^u(x)\rho_{j,k}^v(y),$
    \item $\forall i,j,k\in \mathscr{I}, \forall w\in W, \forall x\in X_{i,j}(w),\forall y \in X_{i,k}(w), \rho_{i,j}^w(x)^*\rho_{i,k}^w(y) = \delta_{j,k}\rho_{j,j}^e(\langle x,y\rangle_{B_j})\psi_j(Q_w),$
    \item $\forall i,j \in \mathscr{I}, \forall u,v\in W, \forall x\in X_{i,j}(u), \rho_{i,j}^u(x)\psi_j(P_v) = \psi_i(\alpha_u(P_v))\rho_{i,j}^u(x).$
\end{enumerate}
\end{definition}
These axioms imply that the maps $\rho_{i,i}^e:B_i \rightarrow C$ are $*$-homomorphisms. We will often write $\rho(x)$ for $\rho_{i,j}^w(x)$ when it is already understood that $x\in X_{i,j}(w).$
\begin{lemma}\label{formulabasic}
    Let $(\rho_{i,j}^w : X_{i,j}(w) \rightarrow C)_{i,j \in \mathscr{I}, w\in W}$ be such a representation. For all $i,j \in \mathscr{I},w\in W$ and $x\in X_{i,j}(w)$ we have $\rho(x)\psi_{j}(Q_w) = \psi_i(P_w)\rho(x)=\rho(x)$. For every $i,j \in \mathscr{I},w\in W$ the linear map $\rho_{i,j}^w : X_{i,j}(w) \rightarrow C$ together with the $*$-homomorphism $B_j \rightarrow C$ defined by $b \mapsto \rho_{j,j}^{e}(b)\psi_j(Q_w)$ induce $*$-homomorphisms $\pi_{i,j}^w : \mathcal{K}_{B_j}(X_{i,j}(w)) \rightarrow C$ by the formula $\pi_{i,j}^w(\theta_{x,y}) = \rho(x)\rho(y)^*$. If $b \mapsto \rho_{j,j}^{e}(b)\psi_j(Q_w)$ is injective then $\pi_{i,j}^w$ is injective for all $w\in W$.
\end{lemma}
\begin{proof}
Let $i,j\in\mathscr{I}$. By applying (iii) with $u=e$, we see that the ranges of $\rho_{j,j}^e$ and $\psi_j$ commute, hence $b \mapsto \rho_{j,j}^{e}(b)\psi_j(Q_w)$ is a $*$-homomorphism. Let $w\in W$ and $x\in X_{i,j}(w)$. Applying (ii) we get
\[
\begin{aligned}
\bigl(\rho(x)(1-\psi_j(Q_w))\bigr)^*
\rho(x)(1-\psi_j(Q_w))
&=
(1-\psi_j(Q_w))\rho(x)^*\rho(x)(1-\psi_j(Q_w))\\
&=
(1-\psi_j(Q_w))
\rho_{j,j}^e(\langle x,x\rangle_{B_j})
\psi_j(Q_w)
(1-\psi_j(Q_w))\\
&=0.
\end{aligned}
\]
Therefore $\rho(x)\psi_j(Q_w)=\rho(x)$.
Moreover, as $\alpha_w(Q_w)=P_w$, axiom (iii) gives $\rho(x)\psi_j(Q_w)=\psi_i(P_w)\rho(x).$ The rest of the proof follows from Proposition \ref{BOBO}.
\end{proof}
Through the identification $\mathcal{K}_{\mathcal{P}}(\mathcal{X}_i(w)) \cong \bigoplus^{c_0}_{j\in \mathscr{I}}\mathcal{K}_{B_j}(X_{i,j}(w)),$ the above representation also induces $*$-homomorphisms $\pi_i^w : \mathcal{K}_{\mathcal{P}}(\mathcal{X}_i(w)) \rightarrow C$ which are injective if for all $j\in \mathscr{I},$ $b \mapsto \rho_{j,j}^{e}(b)\psi_j(Q_w)$ is injective. The representation will be called covariant if it satisfies $\pi_i^w(\mathrm{id}_{\mathcal{X}_i(w)}) = \psi_i(P_w)$ for all $w\in W$. Whenever the representation is covariant then $\rho_{j,j}^e$ is unital for all $j\in \mathscr{I}$.
\begin{proposition}
    For all $i,j \in \mathscr{I}, w\in W\setminus\{e\},x\in X_{i,j}(w)$, let $\rho_{i,j}^w(x) = x^\dagger$. Let $\rho_{i,i}^e : B_i \rightarrow \mathscr{C}$ be the $*$-homomorphism given by the left action of $B_i$ on $\mathcal{X}$. Let $\psi_i : C(\overline{(W,S)}) \rightarrow \mathscr{C}$ be the embedding given by Proposition \ref{cw}. The family of linear maps $(\rho_{i,j}^w)_{i,j \in \mathscr{I}, w\in W}$ together with the $*$-homomorphisms $(\psi_i)_{i \in \mathscr{I}}$ form a covariant representation of the $W$-decomposition.
\end{proposition}
\begin{proof}
    Point (i) is obvious. Point (ii) is precisely Proposition \ref{QW}. Point (iii) is Proposition \ref{CWDELTA}. The covariance condition was proven in Proposition \ref{pw}.
\end{proof}
We let $\mathscr{C}^{\max}$ be the universal C$^*$-algebra generated by all covariant representations of the expected $W$-decomposition. 
From now on, we let $\rho_{i,j}^w : X_{i,j}(w) \rightarrow \mathscr{C}^{\max}$ and $\psi_i : C(\overline{(W,S)}) \rightarrow \mathscr{C}^{\max}$ be the maps of the universal covariant representation.
\begin{definition}
Let $i,j \in \mathscr{I}$. A path operator of $\mathscr{C}^{\max}$ from $i\in \mathscr{I}$ to $j$ is either an element of the form \[\rho_{i_{0},i_1}^{s_1}(x_1)^{\epsilon_1}\ldots \rho_{i_{n-1},i_n}^{s_n}(x_n)^{\epsilon_n}\] for some $n\geq 1$, $i_0,\ldots i_n \in \mathscr{I}$ such that $i = i_0$ and $j = i_n$, $s_k \in S$, $x_k \in X_{i_{k-1},i_k}(s_k)$ and $\epsilon_k \in \{\star,1\}$, or an operator of the form $\rho_{i,i}^e(b)$ for some $b\in B_i$ in the case when $i=j$. A path operator of $\mathscr{C}^{\max}$ from $i\in \mathscr{I}$ to $i$ is called a path operator of $\mathscr{C}^{\max}$ based at $i$. We let $\mathscr{C}^{\max}(i)$ be the C$^*$-subalgebra of $\mathscr{C}^{\max}$ generated by all path operators of $\mathscr{C}^{\max}$ based at $i.$
\end{definition}
The last proposition thus gives us a $*$-homomorphism $\Phi : \mathscr{C}^{\max} \rightarrow \mathscr{C} \subset \mathcal{L}_\mathcal{P}(\mathcal{X})$ which induces $*$-homomorphisms $\Phi_i : \mathscr{C}^{\max}(i) \rightarrow \mathscr{C}(i)$ for all $i\in \mathscr{I}$. \begin{proposition}
    The $*$-homomorphism $\Phi : \mathscr{C}^{\max} \rightarrow \mathscr{C}$ is surjective. For all $i \in \mathscr{I}$, the restricted $*$-homomorphism $\Phi_i : \mathscr{C}^{\max}(i) \rightarrow \mathscr{C}(i)$ is surjective.
\end{proposition}
\begin{proof}
To prove the surjectivity of $\Phi$, it suffices to prove that its image contains all elementary diagonal operators. Let $j,k\in \mathscr{I}$ and $s\in S, a\in F_{j,k}(s)$. For all $l \in \mathscr{I}$, we may see the restriction $\partial(a)|_{X_l}$ of $\partial(a)$ to $\mathcal{L}_{B_l}(X_l)$ as an operator of $\mathcal{L}_{B_l}(X_{k,l}(s),X_{j,l}(s))$. As $X_{k,l}(s)$ and $X_{j,l}(s)$ are finitely generated over $B_l$, we may view the restriction of $\partial(a)$ to $\mathcal{L}_{B_l}(X_l)$ as an operator of $\mathcal{K}_{B_l}(X_{k,l}(s),X_{j,l}(s))$.  There exist $x_1,\ldots x_n \in X_{j,l}(s),y_1,\ldots y_n \in X_{k,l}(s)$ such that \[\partial(a)|_{X_l} = \sum_{1\leq q \leq n}x_q^\dagger (y_q^\dagger)^*\] as an element of $\mathcal{L}_{B_l}(X_l)$. Hence we can write $\partial(a) \in \mathcal{L}_\mathcal{P}(\mathcal{X})$ as a finite linear combination of operators of the form $x^\dagger (y^\dagger)^*$ for $x\in X_{j,l}(s), y \in X_{k,l}(s), l \in \mathscr{I}.$
\end{proof}
Our goal is now to show the injectivity of $\Phi_i$. The proof of the following proposition is routine, so we omit it.
\begin{proposition}\label{fellbundlecmax}
For every \(w \in W\), let \(\mathscr{C}_w^{\max}(i)\) denote the closure of the linear span of the path operators of $\mathscr{C}^{\max}$ based at $i$ of degree $w$, meaning the elements of the form \[\rho_{i_{0},i_1}^{s_1}(x_1)^{\epsilon_1}\ldots \rho_{i_{n-1},i_n}^{s_n}(x_n)^{\epsilon_n}\] for some $n\geq 1$, $i_0,\ldots i_n \in \mathscr{I}$ such that $i = i_0 = i_n$, $s_1,\ldots, s_n \in S$, $x_k \in X_{i_{k-1},i_k}(s_k)$ and $\epsilon_k \in \{\star,1\}$, such that \(s_1\cdots s_n = w\) or an operator of the form $\rho_{i,i}^e(b)$ for some $b\in B_i$ in the case when $w = e$. We have \[
\mathscr{C}^{\max}(i)
=
\overline{\sum_{w \in W} \mathscr{C}^{\max}_w(i)},
\qquad
\,\ \mathscr{C}_u^{\max}(i)\mathscr{C}_v^{\max}(i)
\subset
\mathscr{C}_{uv}^{\max}(i),
\qquad
(\mathscr{C}^{\max}_w(i))^*
=
\mathscr{C}_{w^{-1}}^{\max}(i),
\]
for all $u,v,w\in W.$ Moreover, we have $\Phi_i(\mathscr{C}^{\max}_w(i)) \subset \mathscr{C}_w(i)$ for all $w\in W$.
\end{proposition}
We denote by \(\mathcal{D}^{\max}(i) = \mathscr{C}_e^{\max}(i)\) the diagonal subalgebra of the universal covariance \(C^*\)-algebra based at $i$.  First we are going to show the injectivity of $\Phi_i$ on $\mathcal{D}^{\max}(i)$. Define $X^\Delta_{i,j}(w) = \overline{\operatorname{span}}\{x^\dagger p_{j,v} | x\in X_{i,j}(w),v\in W\}$ as a closed subspace of $\mathcal{L}_\mathcal{P}(\mathcal{X})$ for all $w\in W$.
\begin{lemma}
 The C$^*$-algebra $\Delta(j) \cong C(\overline{(W,S)})\otimes B_j$ acts on $X^\Delta_{i,j}(w)$ on the right. The formula $(a^\dagger)^*b^\dagger = E_j(a^*b)q_{j,w}$ for all $a,b\in F_{i,j}(w)$ makes $X^\Delta_{i,j}(w)$ a right Hilbert $\Delta(j)$-module. See $\Delta(j)$ as a $B_j$-$\Delta(j)$ correspondence through the $*$-homomorphism defined by $b \mapsto bq_{j,w}$ for all $b\in B_j$. As a right Hilbert $\Delta(j)$-module, $X^\Delta_{i,j}(w)$ is isomorphic to $X_{i,j}(w)\otimes_{B_j}\Delta(j)$.
\end{lemma}
\begin{proof}
    The right action of $\Delta(j)$ on $X^\Delta_{i,j}(w)$ is just the one induced by the product inside $\mathcal{L}_\mathcal{P}(\mathcal{X})$. The isomorphism $X^\Delta_{i,j}(w)\cong X_{i,j}(w)\otimes_{B_j}\Delta(j)$ is a direct consequence of Proposition \ref{QW}.
\end{proof}
We now define the unital C$^*$-algebra \[\mathcal{P}^\Delta = \prod_{j\in \mathscr{I}}\Delta(j).\] We also define the right Hilbert $\mathcal{P}^\Delta$-module \[\mathcal{X}_i^\Delta(w) = \bigoplus_{j\in \mathscr{I}} X_{i,j}(w)\otimes_{B_j}\mathcal{P}^\Delta,\] where $\mathcal{P}^\Delta$ becomes a left $B_j$-module through the $*$-homomorphism $b \mapsto (\delta_{j,k}bq_{k,w})_{k \in \mathscr{I}}$. Notice that the right $\mathcal{P}^\Delta$-modules of the form $\mathcal{X}_i^\Delta(w)$ and $X_{i,j}^\Delta(w)$ are finitely generated.
\begin{lemma}\label{compactDelta}
Let $i\in \mathscr{I}$ and $w\in W$. Proposition \ref{BOBO} gives a $*$-isomorphism  \[
\mathcal{K}_{\mathcal{P}^\Delta}(\mathcal{X}_i^\Delta(w)) \cong 
\overline{\operatorname{span}}\left\{ x^\dagger (y^\dagger)^*p_{i,u} \;\middle|\; x,y \in X_{i,j}(w), j\in \mathscr{I},u\in W \right\}.
\] It is induced on each factor $\mathcal{K}_{\Delta(j)}(X_{i,j}^\Delta(w))$ by the inclusions $\Delta(j) \subset \mathcal{L}_\mathcal{P}(\mathcal{X})$ and $X_{i,j}^\Delta(w) \subset \mathcal{L}_\mathcal{P}(\mathcal{X}).$ In the same way, we get an identification \[
\mathcal{K}_{\mathcal{P}^\Delta}(\mathcal{X}_i^\Delta(w)) \cong 
\overline{\operatorname{span}}\{ \rho_{i,j}^w(x)\rho_{i,j}^w(y)^*\psi_i(P_u) | x,y \in X_{i,j}(w), j\in \mathscr{I}, u \in W \}
\] induced on each factor $\mathcal{K}_{\Delta(j)}(X_{i,j}^\Delta(w))$ by the embedding $\Delta(j) \hookrightarrow \mathscr{C}^{\max}$ and the linear map $X_{i,j}^\Delta(w) \cong X_{i,j}(w)\otimes_{B_j} \Delta(j) \rightarrow \mathscr{C}^{\max}$ induced by $\rho^w_{i,j}$.
\end{lemma}
\begin{proof}
    Recall that there is a $*$-isomorphism $\mathcal{K}_{\mathcal{P}^\Delta}(\mathcal{X}^\Delta_i(w)) \cong \bigoplus^{c_0}_{j \in \mathscr{I}}\mathcal{K}_{\Delta(j)}(X_{i,j}^\Delta(w))$. The condition $\tau(x)^*\tau(y) = \pi(\langle x,y\rangle)$ of Proposition \ref{BOBO} is verified in the first case because of Proposition \ref{QW} and in the second case because of condition (ii) in Definition \ref{covariantrep}.
    Proposition \ref{QW} and condition (ii) also tell us that the induced $*$-homomorphisms $\mathcal{K}_{\Delta(j)}(X_{i,j}^\Delta(w)) \rightarrow \mathcal{L}_\mathcal{P}(\mathcal{X})$ and $\mathcal{K}_{\Delta(j)}(X_{i,j}^\Delta(w)) \rightarrow \mathscr{C}^{\max}$ have orthogonal images for different $j \in\mathscr{I}$ and thus define $*$-homomorphisms on $\mathcal{K}_{\mathcal{P}^\Delta}(\mathcal{X}^\Delta_i(w))$. The image of $\mathcal{K}_{\mathcal{P}^\Delta}(\mathcal{X}^\Delta_i(w))$ is the closure of the span of the elements of the form $x^\dagger p_{j,u}p_{j,v} (y^\dagger)^*$ for $x,y \in X_{i,j}(w), j\in \mathscr{I},u,v\in W$. Conclude by Proposition \ref{CWDELTA}. For the second representation, conclude by (iii) of Definition \ref{covariantrep}.
\end{proof}
For all $w\in W$ and $i\in \mathscr{I}$, denote by $K_{i,w} \cong \mathcal{K}_{\mathcal{P}^\Delta}(\mathcal{X}_i^\Delta(w))$ the image of the embedding $\mathcal{K}_{\mathcal{P}^\Delta}(\mathcal{X}_i^\Delta(w)) \hookrightarrow \mathcal{D}(i)$. We denote by $K_{i,w}^{\max} \cong \mathcal{K}_{\mathcal{P}^\Delta}(\mathcal{X}_i^\Delta(w))$ the image of the corresponding embedding in $\mathcal{D}^{\max}(i)$. We have $\Phi_i(K_{i,w}^{\max}) = K_{i,w}$. Moreover the restriction of $\Phi_i$ to $K_{i,w}^{\max}$ is injective by Proposition \ref{BOBO} and Proposition \ref{cw}. Notice that $p_{i,w}$ (resp. $\psi_i(P_{w})$) is the unit of $K_{i,w}$ (resp. $K^{\max}_{i,w}$).
\begin{lemma}\label{kuvv}
    For all $u,v \in W$ such that $u \vee v < \infty$ (for the right weak order) we have $K_{i,u}\cdot K_{i,v} \subset K_{i,u\vee v}$ and $K_{i,u}^{\max}\cdot K_{i,v}^{\max} \subset K_{i,u\vee v}^{\max}$. If $u \vee v$ does not exist then $K_{i,u}\cdot K_{i,v} = 0$ and $K_{i,u}^{\max}\cdot K_{i,v}^{\max} = 0$.
\end{lemma}
\begin{proof}
    Let $j,k \in \mathscr{I}$. Let $u,v\in W$. Let $x_1,y_1\in X_{i,j}(u)$ and $x_2,y_2 \in X_{i,k}(v)$. Define $f = x_1^\dagger(y_1^\dagger)^*x_2^\dagger(y_2^\dagger)^*$. Let $w\in W$ and $l \in \mathscr{I}$. It is clear that $f|_{X_{i,l}(w)} = 0$ if $w \ngeq u$ or $w\ngeq v$. In particular, if $u\vee v = \infty$ then $f=0.$ Assume that $u\vee v < \infty$. Let $l\in \mathscr{I}$. As $X_{i,l}(u \vee v)$ is of finite type over $B_l$, we have $f|_{X_{i,l}(u \vee v)} \in \mathcal{L}_{B_l}(X_{i,l}(u \vee v)) = \mathcal{K}_{B_l}(X_{i,l}(u \vee v))$. The first inclusion follows.
    Let $u,v\in W$. Let $k_1 \in K_{i,u}^{\max}$. Let $k_2 \in K_{i,v}^{\max}$. Lemma \ref{formulabasic} gives $k_1k_2 = k_1\psi_i(P_u)\psi_i(P_v)k_2$ which equals $0$ if $u\vee v = \infty$ and equals $k_1\psi_i(P_{u\vee v})k_2$ otherwise. In this case a direct computation using the covariance condition proves $k_1\psi_i(P_{u\vee v})k_2\in K_{i,u\vee v}^{\max}$. This proves the other desired inclusion.
\end{proof}
\begin{proposition}\label{Cspancompacts}
Let $i,j \in \mathscr{I}$.  Every path operator from $i$ to $j$ is a finite linear combination of operators of the form \[x^\dagger(y^\dagger)^*p_{j,w} \in \mathcal{L}_\mathcal{P}(\mathcal{X})\] for some $x\in X_{i,k}(u),y\in X_{j,k}(v),k \in \mathscr{I}, u,v,w\in W$. Every path operator of $\mathscr{C}^{\max}$ from $i$ to $j$ is a finite linear combination of operators of the form \[\rho(x)\rho(y)^*\psi_j(P_{w}) \in \mathscr{C}^{\max}\] for some $x\in X_{i,k}(u),y\in X_{j,k}(v),k \in \mathscr{I}, u,v,w\in W$.
\end{proposition}
\begin{proof}
    For the first assertion, use Lemma \ref{TWW} and the fact that each $X_{j,k}(w)$ is of finite type to write every path operator $f$ from $i$ to $j$ as a finite linear combination of elements of the form $x^\dagger (y^\dagger)^*p$, where $x\in X_{i,k}(u),y\in X_{j,k}(v), u,v \in W$, $p \in \Delta_0(j)$. We now prove the second assertion. Take an operator of the form $f = \rho(x)\rho(y)^*\psi_j(P_{w}) \in \mathscr{C}^{\max}$ for some $x\in X_{i,k}(u),y\in X_{j,k}(v),k \in \mathscr{I}, u,v,w\in W$. Let $z\in X_{k',i}(s), s\in S,k' \in \mathscr{I}$.
    In this case, we have $\rho(z)\rho(x)\rho(y)^* = \rho(zx)\rho(y)^*$. Let $z\in X_{i,k'}(s), s\in S,k' \in \mathscr{I}$. We then have $\rho(z)^*\rho(x) = \rho(z)^*\psi_i(P_s)\psi_i(P_u)\rho(x)$ so we may assume that $s\vee u < \infty$. We have assumed that the decomposition is of finite type, hence there exist families $x_q \in X_{i,l_q}(s\vee u), y_q \in X_{i,l_q}(s \vee u),l_q \in \mathscr{I}$ indexed by $1 \leq q \leq n$ such that \[\psi_i(P_{s\vee u}) = \sum_{1\leq q \leq n} \rho(x_q)\rho(y_q)^*.\] For all $1\leq q \leq n$, write $x_q = \sum_{l\in \mathscr{I}}a_{q,l}^1a_{q,l}^2$ for some $a_{q,l}^1 \in X_{i,l}(s), a_{q,l}^2 \in X_{l,l_q}(s(s\vee u))$ and $y_{q} = \sum_{l\in \mathscr{I}}b_{q,l}^1b_{q,l}^2$ for some $b_{q,l}^1 \in X_{i,l}(u), b_{q,l}^2 \in X_{l,l_q}(u^{-1}(s\vee u))$. We know that for all fixed $1 \leq q \leq n$ the families $(a^1_{q,l})_{l\in \mathscr{I}}$ and $(b^1_{q,l})_{l\in \mathscr{I}}$ have finite support by Proposition \ref{finito}. The covariance condition gives
\[
\begin{aligned}
\rho(z)^*\rho(x)
&= \rho(z)^* \psi_i(P_{s \vee u}) \rho(x) \\
&= \sum_{1 \leq q \leq n}
   \rho(z)^* \rho(x_q) \rho(y_q)^* \rho(x) \\
&= \sum_{1 \leq q \leq n, l,l'\in \mathscr{I}}
   \rho(z)^* \rho(a_{q,l}^1)\rho(a_{q,l}^2)
   \rho(b_{q,l'}^2)^* \rho(b_{q,l'}^1)^* \rho(x) \\
&= \sum_{1 \leq q \leq n, l,l' \in \mathscr{I}}
   \delta_{k',l}\delta_{k,l'}\rho(\langle z, a_{q,l}^1 \rangle)\,
   \psi_l(Q_s)\,
   \rho(a_{q,l}^2)\rho(b_{q,l'}^2)^*\,
   \rho_e(\langle b_{q,l'}^1, x \rangle)\psi_k(Q_u) \\
   &= \sum_{1 \leq q \leq n}
   \psi_{k'}(Q_s)\,
   \rho(c_{q,k'})\rho(d_{q,k}^2)^*\psi_k(Q_u).
\end{aligned}
\]
Here we have defined $c_{q,k'} = \langle z,a_{q,k'}^1\rangle a^2_{q,k'} \in X_{k',l_q}(s(s\vee u))$ and $d_{q,k}^2 = \langle x,b^1_{q,k}\rangle b^2_{q,k} \in X_{k,l_q}(u^{-1}(s\vee u))$. We can now conclude using (iii) and Lemma \ref{qwqw}.
\end{proof}
The last proposition gives the following more pleasant description of the covariance C$^*$-algebras:
\[
    \mathscr{C}(i) = \overline{\mathrm{span}}\{x^\dagger(y^\dagger)^*p_{i,w}|x\in X_{i,j}(u),y\in X_{i,j}(v),j \in \mathscr{I}, u,v,w\in W\}\] and  \[\mathscr{C}^{\max}(i) = \overline{\mathrm{span}}\{\rho_u(x)\rho_v(y)^*\psi_i(P_w)|x\in X_{i,j}(u),y\in X_{i,j}(v),j \in \mathscr{I}, u,v,w\in W\}.\]
    In particular, \[\mathcal{D}(i) = \overline{\mathrm{span}}\{x^\dagger(y^\dagger)^*p_{i,v}|x\in X_{i,j}(u), y\in X_{i,j}(u), j \in \mathscr{I}, u,v\in W\}\] and \[\mathcal{D}^{\max}(i) = \overline{\mathrm{span}}\{\rho_u(x)\rho_u(y)^*\psi_i(P_v)|x\in X_{i,j}(u),y\in X_{i,j}(u), j \in \mathscr{I},u,v\in W\}.\]
\begin{definition}
    A subset $\mathcal{F} \subset W$ is called $\vee$-closed if for all $u,v\in \mathcal{F}$ such that $u\vee v < \infty$, we have $u\vee v \in \mathcal{F}$.
\end{definition}
Let $i\in \mathscr{I}$. Let $\mathcal{F}\subset W$ be $\vee$-closed (for the right weak order on $W$). Let \[K_i(\mathcal{F}) = \overline{\sum_{w\in \mathcal
 F}K_{i,w}}\subset \mathcal{D}(i), K^{\max}_i(\mathcal{F}) = \overline{\sum_{w\in \mathcal
 F}K_{i,w}^{\max}}\subset \mathcal{D}^{\max}(i).\]
Lemma \ref{kuvv} proves that these are actually C$^*$-subalgebras. Whenever $\mathcal{F}_1 \subset \mathcal{F}_2$ we have $K_i(\mathcal{F}_1) \subset K_i(\mathcal{F}_2)$ and $K^{\max}_i(\mathcal{F}_1) \subset K^{\max}_i(\mathcal{F}_2)$. Proposition \ref{Cspancompacts} gives the following description of the diagonal subalgebras:

\[
\mathcal{D}(i) =\overline{\bigcup_{\substack{\mathcal{F} \subset W \\ \mathcal{F}\ \text{finite, }\vee\text{-closed}}} K_i(\mathcal{F})}, \qquad  \mathcal{D}^{\max}(i) = \overline{\bigcup_{\substack{\mathcal{F} \subset W \\ \mathcal{F}\ \text{finite, }\vee\text{-closed}}} K^{\max}_i(\mathcal{F})}.\]

For every finite $\vee$-closed subset $\mathcal{F}\subset W$, we define the projections $p_{i,\mathcal{F}} = \bigvee_{w\in \mathcal{F}}p_{i,w} \in \Delta_0(i)$ and $P_{\mathcal{F}} = \bigvee_{w\in \mathcal{F}}P_w \in C(\overline{(W,S)})$.
\begin{lemma}\label{compactinter}
   Let $\mathcal{F}\subset W$ be finite and $\vee$-closed. Let $w\in \mathcal{F}$ be of minimal length. Then $K_i(\mathcal{F}\setminus\{w\})$ and $K^{\max}_i(\mathcal{F}\setminus\{w\})$ are respectively C$^*$-ideals of $K_i(\mathcal{F})$ and of $K^{\max}_i(\mathcal{F})$. Moreover \[K_{i,w}\cap K_i(\mathcal{F}\setminus\{w\}) = K_{i,w}\cdot p_{i,\mathcal{F}\setminus\{w\}},\qquad K^{\max}_{i,w}\cap K^{\max}_i(\mathcal{F}\setminus\{w\}) = K_{i,w}^{\max}\cdot \psi_i(P_{\mathcal{F}\setminus\{w\}}).\] 
\end{lemma}
\begin{proof}
    The fact that $K_i(\mathcal{F}\setminus\{w\})$ is an ideal of $K_i(\mathcal{F})$ and that $K^{\max}_i(\mathcal{F}\setminus\{w\})$ is an ideal of $K^{\max}_i(\mathcal{F})$ is a consequence of Lemma \ref{kuvv}. The two equalities follow from the fact that $p_{i,\mathcal{F}\setminus\{w\}}$ (resp. $\psi_i(P_{\mathcal{F}\setminus\{w\}})$) is the unit of $K_i(\mathcal{F}\setminus\{w\})$ (resp. $K^{\max}_i(\mathcal{F}\setminus\{w\})$).
\end{proof}
\begin{proposition}
    $\Phi_i$ induces a $*$-isomorphism $\mathcal{D}^{\max}(i)\cong \mathcal{D}(i)$.
\end{proposition}
\begin{proof}
    It suffices to prove that $\Phi_i$ induces a $*$-isomorphism $K_i^{\max}(\mathcal{F}) \cong K_i(\mathcal{F})$ for every $\vee$-closed, finite subset of $W$. We do so by induction on $n = |\mathcal{F}|$. If $\mathcal{F} = \{w\}$ then $K_i^{\max}(\mathcal{F}) = K_{i,w}^{\max} \cong K_{i,w} = K_i(\mathcal{F})$. Let $n\geq 1$. Assume the result is true for every such $\mathcal{F}$ of cardinality $n$. Let $\mathcal{F}$ be a $\vee$-closed, finite subset of $W$ of cardinality $n+1$. Let $w\in \mathcal{F}$ be of minimal length. The last lemma implies that $K_i(\mathcal{F}\setminus \{w\}) \lhd K_i(\mathcal{F})$ and $K_i^{\max}(\mathcal{F}\setminus \{w\}) \lhd K_i^{\max}(\mathcal{F})$ are C$^*$-ideals. Moreover we have two short exact sequences and a commutative diagram\[\begin{tikzcd}
0 \arrow[r] & K_i^{\max}(\mathcal{F}\setminus \{w\}) \arrow[r] \arrow[d] & K_i^{\max}(\mathcal{F}) \arrow[r] \arrow[d] & K_{i,w}^{\max}/(K_{i,w}^{\max}\psi_i(P_{\mathcal{F}\setminus\{w\}})) \arrow[r] \arrow[d] & 0 \\
0 \arrow[r] & K_i(\mathcal{F}\setminus \{w\}) \arrow[r]                  & K_i(\mathcal{F}) \arrow[r]                  & K_{i,w}/(K_{i,w}p_{i,\mathcal{F}\setminus\{w\}}) \arrow[r]                     & 0
\end{tikzcd}
\]
where the vertical maps are induced by $\Phi_i$. We know that $\Phi_i$ restricts to an isomorphism $K_{i,w}^{\max} \cong K_{i,w} \cong \mathcal{K}_{\mathcal{P}^\Delta}(\mathcal{X}_i^\Delta(w))$. Hence the third vertical arrow in the above diagram is an isomorphism. The five lemma concludes the proof.
\end{proof}
\begin{proposition}
    The Fell bundle structure of $\mathscr{C}^{\max}(i)
=
\overline{\sum_{w \in W} \mathscr{C}^{\max}_w(i)}$ is a topological $W$-grading.
\end{proposition}
\begin{proof}
    We know that $\Phi_i$ induces a $*$-isomorphism $\mathcal{D}^{ \max}(i) \rightarrow \mathcal{D}(i)$.
Use this to define a conditional expectation $F^{\max}
=
\left(\Phi_i|_{\mathcal D^{\max}(i)}\right)^{-1}\circ F\circ \Phi_i$. By Proposition \ref{cgraded}, and Proposition \ref{fellbundlecmax} we have $F^{\max}(\mathscr{C}^{\max}_w(i)) = 0$ for all $w\in W\setminus\{e\}$. Conclude by \cite[Theorem 3.3]{Exe97}.
\end{proof}
 Now we are going to show the injectivity of $\Phi_i$ by using Exel's work on amenability of graded C$^*$-algebras. 
\begin{definition}
    Let $\Gamma$ be a discrete group. Let $\mathcal{A} = \overline{\bigoplus_{g\in \Gamma}\mathcal{A}_g}$ be a $\Gamma$-graded C$^*$-algebra. The grading is said to satisfy Exel's approximation property if there exists a net $(a_{\lambda})$ of functions $a_{\lambda} : \Gamma \rightarrow \mathcal{A}_e$ such that \[\sup_{\lambda}\left\| \sum_{g\in \Gamma}a_{\lambda}(g)^*a_{\lambda}(g)  \right\|< \infty,\] and for all $h\in \Gamma$, $x\in \mathcal{A}_h$, \[
\sum_{g\in \Gamma}
a_{\lambda}(hg)^*\,x\,a_{\lambda}(g)
\underset{\lambda\to\infty}{\longrightarrow}
x.
\]
\end{definition}
\begin{theorem}\cite{Exe97}\label{theoremexel}
    Let $\Gamma$ be a discrete group. Let $\mathcal{A} = \overline{\bigoplus_{g\in \Gamma}\mathcal{A}_g}$ be a topologically $\Gamma$-graded C$^*$-algebra with conditional expectation $\mathcal{E} : \mathcal{A} \rightarrow \mathcal{A}_e$. Assume the grading satisfies Exel's approximation property. Then $\mathcal{E}$ is faithful.
\end{theorem}
\begin{proposition}\label{amenable}
    The grading $\mathscr{C}^{\max}(i)
=
\overline{\sum_{w \in W} \mathscr{C}^{\max}_w(i)}$ satisfies Exel's approximation property.
\end{proposition}
\begin{proof}
By Theorem \ref{coxexact} we know there exists a net of continuous maps $m_{\lambda} : \overline{(W,S)} \rightarrow \mathrm{Prob}(W)$ such that for all $w\in W$ we have \[
\sup_{x \in \overline{(W,S)}} \bigl\| w \cdot m_{\lambda}^x - m_{\lambda}^{w \cdot x} \bigr\|_1 \xrightarrow[\lambda \to \infty]{} 0.
\]
For every index $\lambda$ and every \(u\in W\), define
$a_{\lambda}(u)\colon x\longmapsto \sqrt{m_{\lambda}^x(u)}$.
This defines a net
$a_{\lambda}\colon W\longrightarrow C(\overline{(W,S)}) \cong \Delta_0(i)\subset \mathcal D^{\max}(i).$
Moreover, for every index $\lambda$, we have
\[
\sum_{u\in W} a_{\lambda}(u)^*a_{\lambda}(u)
=
\left(x\longmapsto \sum_{u\in W} m_{\lambda}^x(u)\right)
=1.
\]
Let \(w\in W\), and let \(f\in \mathscr C_w^{\max}(i)\). Using point (iii) of Definition \ref{covariantrep} we obtain\[
\sum_{u \in W} a_{\lambda}(wu)^* f a_{\lambda}(u)
=
\Biggl(
\sum_{u \in W} a_{\lambda}(wu)^* (w \cdot a_{\lambda})(u)
\Biggr) f.
\] For all $x \in \overline{(W,S)}$,
\[
\begin{aligned}
\Biggl\lvert
\Big{(}\sum_{u \in W} a_{\lambda}(wu)^*(x)\,(w \cdot a_{\lambda})(u)(x)\Big{)} - 1
\Biggr\rvert
&=
\Biggl\lvert
\sum_{u \in W}
\sqrt{m_{\lambda}^x(wu)\,m_{\lambda}^{w^{-1}\cdot x}(u)}
- m_{\lambda}^x(wu)
\Biggr\rvert \\
&\leq
\sum_{u \in W}
\Biggl\lvert
\sqrt{m_{\lambda}^x(wu)}
\Bigl(
\sqrt{m_{\lambda}^{w^{-1}\cdot x}(u)}
- \sqrt{m_{\lambda}^x(wu)}
\Bigr)
\Biggr\rvert \\
&\leq
\Big{(}\sum_{u \in W}
\Bigl\lvert
\sqrt{m_{\lambda}^{w^{-1}\cdot x}(u)}
- \sqrt{m_{\lambda}^x(wu)}
\Bigr\rvert^2 \Big{)}^{1/2}\\
&\leq
\Big{(}\sum_{u \in W}
\Bigl\lvert
m_{\lambda}^{w^{-1}\cdot x}(u)
- m_{\lambda}^x(wu)
\Bigr\rvert \Big{)}^{1/2}
\;\xrightarrow[\lambda \to \infty]{}\; 0.
\end{aligned}
\]
\end{proof}
We can finally prove the main theorem of this section.
\begin{theorem}\label{bigbig}
Let $A = \overline{\sum_{w\in W}L_{i,j}(w)}$ be a unital C$^*$-algebra with a finite-type expected $W$-decomposition
of index set $\mathscr{I}$. Assume that $S$ is finite. Assume that the left action of $B_i$ on $\mathcal{X}_i(s)$ is
faithful for all $i \in \mathscr{I}$ and $s\in S$. Then for all $i\in \mathscr{I}$, the universal $*$-homomorphism $\Phi_i:\mathscr C^{\max}(i)\to \mathscr{C}(i)$ is a \(*\)-isomorphism.
\end{theorem}
\begin{proof}
 By Theorem \ref{theoremexel} and the last proposition we know that $F^{\max}$ is faithful. Let $x$ be a positive element of $\ker \Phi_i$. We have $\Phi_i(F^{\max}(x)) = F(\Phi_i(x)) = 0$ hence by injectivity of $\Phi_i$ on $\mathcal{D}^{\max}(i)$ we have $F^{\max}(x) = 0$, thus $x= 0$.
\end{proof}
\begin{example}
    Take $C^*_rW = \overline{\bigoplus_{w\in W}\mathbb{C}\cdot\lambda_w}$ as an expected Bruhat decomposition. The last theorem tells us that $C(\overline{(W,S)})\rtimes_r W$ is universal for representations on a unital C$^*$-algebra $C$ consisting of a unital $*$-homomorphism $\psi:C(\overline{(W,S)}) \rightarrow C$ together with a map $\rho : W \rightarrow C$ satisfying \begin{enumerate}[label=(\roman*)]
        \item $\forall u,v\in W, \rho_u \rho_v = \begin{cases}
        \rho_{uv} \text{ if } l(uv) = l(u) + l(v), \\
         0 \text{ otherwise.}
        \end{cases}$
        \item $\forall u,v\in W, \rho_u\psi(P_v) = \psi(\alpha_u(P_v))\rho_u$,
        \item $\forall w\in W, \rho_w^*\rho_w = \psi(Q_w), \rho_w\rho_w^* = \psi(P_w)$.
    \end{enumerate}
\end{example}
\begin{remark}
Assume for simplicity that we have an expected finite-type Bruhat decomposition $A = \overline{\sum_{w\in W}L(w)}$. One defines a natural left action of $\Delta = \overline{\operatorname{span}}\{p_w |w\in W\}$ on each $X^{\Delta}(w) = \overline{\operatorname{span}}\{x^\dagger p_v |x \in X(w),v\in W\}$ by adjointable operators. Moreover, for all $u,v \in W$ such that $l(uv) = l(u) +l(v)$ there is a $\Delta$-correspondence unitary isomorphism $\phi_{u,v}: X^{\Delta}(u)\otimes_\Delta X^{\Delta}(v) \cong X^{\Delta}(uv)$. When $l(uv)< l(u) + l(v)$ we let $\phi_{u,v} = 0$. $(X^{\Delta}(w))_{w\in W}$ with the maps $(\phi_{u,v})_{u,v \in W}$ becomes a subproduct system of correspondences over the C$^*$-algebra $\Delta$ and the group $W$ in the sense of \cite{CLSV11}. Regarding the main result of \cite{Has17}, it would be interesting to investigate to what extent our covariance C$^*$-algebra coincides with one of those introduced in \cite{CLSV11, Seh19}.
\end{remark}
\begin{remark}
    We do not think that a universal property for the global covariance C$^*$-algebra $\mathscr{C}$ exists. In order to establish a more general universality result, one must use C$^*$-categories. One may rewrite this entire article by replacing $A$, $A^r$, $\mathscr{C}$ and $\mathscr{C}^{\max}$ by C$^*$-categories whose object set is $\mathscr{I}$. We could have defined the notion of a $W$-decomposition of a C$^*$-category in order to fit Norledge's formalism of $W$-groupoids \cite{Nor17} (see Remark \ref{groupoids}). In this slightly more natural formalism we would have hom-spaces $A^r(i,j) \subset \mathscr{C}(i,j) \subset \mathcal{L}_\mathcal{P}(\bigoplus_{w\in W}\mathcal{X}_j(w),\bigoplus_{w\in W}\mathcal{X}_i(w))$ for all $i,j \in \mathscr{I}$. We would define $\mathscr{C}(i,j)$ as the closure of the span of path operators from $i$ to $j$. One would then redefine the notion of covariance representation (Definition \ref{covariantrep}) on a C$^*$-category instead of on a C$^*$-algebra. The proof carried out in this section tells us that the canonical $*$-functor $\mathscr{C}^{\max} \twoheadrightarrow \mathscr{C}$ is an isomorphism, meaning that for all $i,j \in \mathscr{I}$, the induced map $\mathscr{C}^{\max}(i,j) \twoheadrightarrow \mathscr{C}(i,j)$ is bijective.
\end{remark}
\subsection{Applications}
Let $(W,S)$ be a Coxeter system. We assume that $S$ is finite in this whole section.
\begin{theorem}\label{nucnucexex}
Let $A = \overline{\sum_{w\in W}L_{i,j}(w)}$ be an expected $W$-decomposition of a unital C$^*$-algebra. Assume that the decomposition is of finite type. Assume moreover that for all $i \in \mathscr{I}$ and $s\in S$, the left action of $B_i$ on $\mathcal{X}_i(s)$ is faithful. Let $(*)$ be either exactness or nuclearity. The following three conditions are equivalent \begin{enumerate}[label=(\roman*)]
\item For all $i \in \mathscr{I}$, $B_i$ has $(*)$.
\item For all $i \in \mathscr{I}$, $\mathscr{C}(i)$ has $(*)$.
\item There exists an $i \in \mathscr{I}$ such that $\mathscr{C}(i)$ has $(*)$.
\end{enumerate}
\end{theorem}
\begin{proof}
For all $i,j \in \mathscr{I}$, there is a conditional expectation from $\mathscr{C}(i)$ to $B_j$ given by the contraction of $\mathcal{L}_{\mathcal{P}}(\mathcal{X})$ onto the submodule $X_{j,j}(e)\otimes_{B_j}\mathcal{P} \cong B_j$. Hence it suffices to prove that (i) implies (ii) for both exactness and nuclearity.

Assume that $B_i$ is exact for all $i\in \mathscr{I}$. For all $j \in \mathscr{I},$ the C$^*$-algebra $\Delta(j)$ is isomorphic to $C(\overline{(W,S)})\otimes B_j$, hence it is exact. Hence for all $i\in \mathscr{I}$ and $w\in W,$ we know that $K_{i,w} \cong \bigoplus_{j\in \mathscr{I}}^{c_0}\mathcal{K}_{\Delta(j)}(X_{i,j}^{\Delta}(w))$ is exact. A direct induction on the cardinality of $\mathcal{F}$ together with the short exact sequence given by Lemma \ref{compactinter} gives exactness of $K_i(\mathcal{F})$ for all finite $\vee$-closed $\mathcal{F}$ and $i \in \mathscr{I}.$ Hence $\mathcal{D}(i)$ is exact. By \cite{DJ99} we know that $W$ is an exact group. We conclude by \cite[Proposition 25.12]{Exe17} that $\mathscr{C}(i)$ is exact.

The proof for nuclearity works the same, using Proposition \ref{amenable} (which requires the actions of each $B_i$ on $\mathcal{X}_i(s)$ to be faithful and $S$ to be finite) and \cite[Proposition 25.10.]{Exe17}.
\end{proof}
We now follow the approach of \cite{Has17} and use the universal property established in the last section to deduce the WEP and LLP for $\mathscr{C}(i)$. For any two C$^*$-algebras $C$ and $D$, denote by $C\odot D$ their algebraic tensor product. For any C$^*$-algebra $Q$ say that a map $C\odot D \rightarrow Q$ is min-bounded if it extends to a map defined on the minimal tensor product $C\otimes D \rightarrow Q$.
\begin{lemma}
    Let $P,Q$ be any two unital C$^*$-algebras. Let $\phi : \mathscr{C}^{\max} \rightarrow P$ be a unital $*$-homomorphism. If $\phi|_{B_i}\otimes \mathrm{id}_Q: B_i\odot Q \rightarrow P\otimes_{\max}Q$ is min-bounded for all $i\in \mathscr{I}$ then the $*$-homomorphism $\mathscr{C}(i)\odot Q \rightarrow P\otimes_{\max}Q$ obtained by composing $\phi|_{\mathscr{C}^{\max}(i)}\otimes \mathrm{id}_Q$ with the $*$-isomorphism $\mathscr{C}(i) \cong \mathscr{C}^{\max}(i)$ of Theorem \ref{bigbig} is min-bounded.
\end{lemma}
\begin{proof}
Write $\mathscr{C}' = \mathscr{C}(A\otimes Q).$ Let $i \in \mathscr{I}$. By Proposition \ref{covtensor}, we know that there is a $*$-isomorphism $\mathscr{C}'(i) \cong \mathscr{C}(i)\otimes Q$. Moreover, by Theorem \ref{bigbig}, there is a $*$-homomorphism from $\mathscr{C}'(i)$ to $(\mathscr{C}')^{\max}$. Hence we just have to build a $*$-homomorphism $(\mathscr{C}')^{\max} \rightarrow P\otimes_{\max} Q$ extending $\phi|_{B_i}\otimes \mathrm{id}_Q$. We do so by constructing a covariant representation of the expected $W$-decomposition $A\otimes Q = \overline{\sum_{w\in W}G_w}$ on $P\otimes_{\max}Q$. We denote the maps of this representation by $\alpha^w_{i,j} : Y_{i,j}(w) \rightarrow P\otimes_{\max}Q$ and $\chi_i : C(\overline{(W,S)}) \rightarrow P\otimes_{\max}Q$ for all $i,j \in \mathscr{I}$ and $w\in W$. Recall that the Borel C$^*$-algebras for this decomposition are given by $B_i\otimes Q$ and the Fock modules are given by $Y_{i,j}(w) = X_{i,j}(w)\otimes_{\mathrm{ext}} Q$. Let $\alpha_{i,i}^{e} = \phi|_{B_i}\otimes \mathrm{id}_Q$ for all $i \in \mathscr{I}$. Define $\alpha_{i,j}^{w}(x\otimes q) = \phi(\rho^w_{i,j}(x)) \otimes q$ for all $x\in X_{i,j}(w), w\neq e, q\in Q$ and $i,j \in \mathscr{I}$. Define $\chi_i : C(\overline{(W,S)}) \rightarrow P\otimes_{\max} Q$ by $\chi_i(P_w) = \phi(\psi_i(P_w))\otimes 1_Q$ for all $w\in W$ and $i \in \mathscr{I}$. One easily checks that the families $(\alpha_{i,j}^w)$ and $(\chi_i)$ do define a covariant representation of the expected $W$-decomposition of $A\otimes Q$.
\end{proof}
\begin{corollary}\label{wepllp}
Let $A = \overline{\sum_{w\in W}L_{i,j}(w)}$ be an expected $W$-decomposition of a unital C$^*$-algebra. Assume that the decomposition is of finite type. Assume that $S$ is finite. Assume moreover that for all $i,j \in \mathscr{I}$ and $s\in S$, the left action of $B_i$ on $\mathcal{X}_i(s)$ is faithful. Assume that for all $i \in \mathscr{I},$ we have $\mathscr{C}(i) = \mathscr{C}$ (e.g. $\mathscr{I}$ is a singleton or the decomposition comes from a group).
Then $\mathscr{C}$ has WEP (resp. LLP) if and only if for all $i \in \mathscr{I}$, $B_i$ has WEP (resp. LLP).
\end{corollary}
\begin{proof}
We use the classical result \cite[Proposition 1.1]{Kir93}. Assume that $B_i$ has WEP for all $i \in \mathscr{I}$. Let $Q = C^*(\mathbb{F}_{\infty})$. In the last lemma set $P = \mathscr{C}$ and let $\phi = \Phi$ be the canonical surjection. For all $i \in \mathscr{I}$, the map $\phi|_{B_i}\otimes \mathrm{id}_Q: B_i\odot Q \rightarrow \mathscr{C}\otimes_{\max}Q$ is min-bounded. Hence $\phi|_{\mathscr{C}}\otimes \mathrm{id}_Q: \mathscr{C}\odot Q \rightarrow \mathscr{C}\otimes_{\max}Q$ is min-bounded. We conclude that $\mathscr{C}$ has WEP. The proof for LLP is the same putting $Q = B(l^2\mathbb{N})$.
\end{proof}
Recall that the group $\mathrm{Aut}(W,S)$ of automorphisms of the Coxeter complex $(W,S)$ is the group of permutations $\sigma : S \rightarrow S$ such that $m_{\sigma(s),\sigma(t)} = m_{s,t}$ for all $s,t\in S.$ Such a $\sigma$ naturally extends to a group automorphism $\sigma : W \rightarrow W.$ Let $G$ be a discrete group. Let $X$ be a building of type $(W,S)$ with $W$-distance $\delta : X\times X \rightarrow W.$ An action of $G$ on the building $X$ is the data of an action of $G$ on the set $X$ and a group homomorphism $\sigma : G \rightarrow \mathrm{Aut}(W,S)$ such that for all $g\in G$ we have \[\forall x,y \in X, \delta(g\cdot x,g \cdot y)  =\sigma_g(\delta(x,y)).\]
By specializing to the example of a group and using Theorem \ref{convcrossed}, we get the following result.
\begin{corollary}\label{exactG}
    Assume that $S$ is finite. Let $G$ be a discrete group acting on a locally finite building $X$. For all $x\in X$, let $H_x$ be the $G$-stabilizer of $x$. Let $\Omega$ be the minimal combinatorial compactification of $X$. Let $(*)$ be one of the following properties: nuclearity, exactness, WEP, and LLP. Then the reduced crossed product $C(\Omega)\rtimes_rG$ has $(*)$ if and only if for all $x\in X$, $C^*_r(H_x)$ has $(*)$.
\end{corollary}
\begin{proof}
    As $S$ is finite we know that the subgroup $G_0=\ker(\sigma) < G$ has finite index. The action of $G_0$ on $X$ is $W$-isometric so by Proposition \ref{nondeggroup}, Theorem \ref{convcrossed}, Theorem \ref{nucnucexex} and Corollary \ref{wepllp}, we know that $C(\Omega)\rtimes_rG_0$ has $(*)$ if and only if $C^*_r(H_x\cap G_0)$ has $(*)$ for all $x\in X$. We conclude by \cite[Corollary~3.11]{FA18}.
\end{proof}
In particular, if each $H_x$ is amenable then the action $G \curvearrowright \Omega$ is amenable.
We also get a weaker version of a result of \cite{CF17}.
\begin{corollary}
Let $\Gamma$ be a simplicial graph with vertex set $V$. For all $v\in V$, let $A_v$ be a unital finite-dimensional C$^*$-algebra. Then the reduced graph product C$^*$-algebra of the family $(A_v)_{v\in V}$ is exact.
\end{corollary}
We finish by presenting an application to the strong condition (AO). Recall the following definition.
\begin{definition}[{\cite[Definition~2.6]{HI17}}]
Let $M$ be a von Neumann algebra and let $(M,H,J,P)$ be a standard form of $M$. We say that $M$ satisfies the strong condition (AO) if there exist unital $C^*$-subalgebras $A$ and $\mathscr{C}$ of $B(H)$ such that $A \subset M$, $ A \subset \mathscr{C}$ and 
\begin{enumerate}[label=(\roman*)] 
    \item $A$ is $\sigma$-weakly dense in $M$,
    \item $\mathscr{C}$ is nuclear,
    \item $\forall c\in \mathscr{C}, \forall a\in A, [c,JaJ] \in \mathcal{K}(H)$.
\end{enumerate}
\end{definition}
Assume that $S$ is finite. Let $G = \biguplus_{w\in W}C_{i,j}(w)$ be a discrete group with a $W$-decomposition of index set $\mathscr{I}$. Let $H_i = C_{i,i}(e)$ denote the Borel subgroup for all $i \in \mathscr{I}$. Let $A \cong C^*_r(G) = \overline{\bigoplus_{w\in W}F_{i,j}(w)} \subset B(l^2G)$ be the associated reduced C$^*$-algebra. The following theorem is inspired by \cite[Theorem 4.6]{Kli23}.
\begin{theorem} Assume that the action of $G$ on the underlying building $X$ is proper and that $X$ is locally finite.  If $(W,S)$ is small at infinity then $M = L(G)$ satisfies the strong condition (AO).
\end{theorem}
\begin{proof}
Let $j \in \mathscr{I}$. Consider the $*$-homomorphism \[\iota : \mathscr{C} \subset \mathcal{L}_\mathcal{P}(\mathcal{X}) \twoheadrightarrow \mathcal{L}_{B_j}(X_j)\hookrightarrow B(l^2G)\] obtained by composing the projection on $\mathcal{L}_{B_j}(X_j)$ with the embedding given by identifying $l^2G = X_j\otimes_{B_j}l^2H_j$ as a $G$-Hilbert space. It is easy to prove that $\iota$ restricted to $A \subset \mathcal{L}_\mathcal{P}(\mathcal{X})$ is just the left regular representation. We consider the C$^*$-algebras $\iota(A) = C^*_r(G) \subset M = L(G)\subset B(l^2G)$. For all $i\in \mathscr{I}$, $B_i$ is finite-dimensional so $\mathscr{C}$ is nuclear by Theorem \ref{nucnucexex} (recall that $\mathscr{C}(i) = \mathscr{C}$ by Theorem \ref{convcrossed}.) Hence $\iota(\mathscr{C})$ is nuclear as a quotient of $\mathscr{C}$. For all $g\in G$, we have $J\lambda(g)J = \rho(g)$, where $\rho(g)$ is the right multiplication operator on $l^2G$. All that is left to check is that $[c,\rho(g)] \in \mathcal{K}(l^2G)$ for all $u,v\in W,c\in \iota(\mathscr{C})$ and $g\in C_{i,j}(v)$. It suffices to check this condition on the elements of $C^*_r(G)$ and the projections of the form $\iota(p_{i,w})$ for arbitrary $i \in \mathscr{I}$ and $w\in W$. We have $[a, \rho(g)] = 0$ for every $a\in C^*_r(G)$ and $g\in G$. For all $w\in W$, denote by $\rho_w \in B(l^2W)$ the operator of right multiplication by $w$. Let $P_w$ be the projection on the submodule of $l^2W$ spanned by vectors $\delta_u$ such that $u \geq w$ for the right weak order. Let $u,v \in W$. As $(W,S)$ is small at infinity, \cite[Lemma 5.3.17]{BO08} gives $[P_{u},\rho_v] \in \mathcal{K}(l^2W)$. Set
\[
D(u,v)
=
\left\{
w\in W
\;\middle|\;
(w\geq u \text{ and } wv\ngeq u) \text{ or } (w\ngeq u \text{ and } wv\geq u)
\right\}.
\]
 For all $w\in W$ we have $[P_u,\rho_{v^{-1}}]\delta_w = (\mathbf{1}_{wv\geq u}-\mathbf{1}_{w\geq u})\delta_{wv}$. It follows from compactness that $D(u,v)$ is finite. Let $i \in \mathscr{I}$. Let $u \in W$. Under the identification $l^2G = \bigoplus_{w\in W}l^2C_{i,j}(w)$, $\iota(p_{i,u})$ is the projection onto the Hilbert subspace $\bigoplus_{w\geq u}l^2C_{i,j}(w)$. Let $g\in C_{j,j}(v)$ for some $v\in W$. Iterating the axiom (Wdec4), we know that there exists a finite set $F \subset W$ such that $\rho(g)$ sends every subspace of the form $l^2C_{i,j}(w)$ to a subspace of the form $\bigoplus_{z\in F}l^2C_{i,j}(wz).$ Consider the finite subset \[
D(u,F)= \bigcup_{z\in F}D(u,z) \subset W.\]
Let $w \notin D(u,F)$. Let $x \in C_{i,j}(w)$. There exists a $z \in F$ such that $xg^{-1} \in C_{i,j}(wz)$. Compute
\[\begin{aligned}
    \bigl[\iota(p_{i,u}),\rho(g)\bigr]\cdot \delta_x & =  \iota(p_{i,u})\rho(g)\delta_{x}
    - \rho(g) \iota(p_{i,u})\delta_{x} \\ & = (\mathbf{1}_{wz\geq u}-\mathbf{1}_{w\geq u})\delta_{xg^{-1}} \\ &= 0.
\end{aligned}\] This means that the operator $\bigl[\iota(p_{i,u}),\rho(g)\bigr]$ is supported on $\bigoplus_{w\in D(u,F)}l^2C_{i,j}(w)$. As $X$ is locally finite and each Borel subgroup $H_j$ is finite, we know that each $C_{i,j}(w)$ is finite. Hence the operator $\bigl[\iota(p_{i,u}),\rho(g)\bigr]$ is of finite rank.
\end{proof}

Notice that for $(W,S)$ to be small at infinity, $W$ has to be hyperbolic in the sense of Gromov by \cite[Theorem 0.3]{Kli23}.

\section*{Acknowledgements}
The author would like to thank his PhD supervisor Emmanuel Germain for support and comments on earlier versions of this paper. He also wishes to thank Jean Lécureux for helpful email exchanges, as well as Pierre Fima for bringing relevant references to his attention.


\begin{thebibliography}{CORIMB}

\bibitem[AB08]{AB08}
P. Abramenko and K.S. Brown,
\newblock \emph{Buildings: Theory and Applications}
\newblock {Graduate Texts in Mathematics, Vol.~248.}
\newblock {Springer, 2008.}

\bibitem[BB05]{BB05}
A. Bj\"orner, F. Brenti,
\newblock \emph{Combinatorics of Coxeter groups,}
\newblock {Graduate Texts in Mathematics, 231. Springer, New
York, 2005. xiv+363 pp.}


\bibitem[BH99]{BH99}
M.~R.~Bridson and A.~Haefliger,
\textit{Metric spaces of non-positive curvature},
Grundlehren der Mathematischen Wissenschaften, vol.~319,
Springer-Verlag, Berlin, 1999, xxii+643~pp.
ISBN: 3-540-64324-9.


\bibitem[BO08]{BO08}
N. Brown, N. Ozawa,
\newblock \emph{$C^*$-algebras and Finite-dimensional Approximations,}
\newblock Graduate Studies in Mathematics,
88. American Mathematical Society, Providence, RI, 2008. xvi+509 pp.
\bibitem[BT72]{BT72}
F.~Bruhat and J.~Tits,
\newblock \emph{Groupes r\'eductifs sur un corps local.
I. Donn\'ees radicielles valu\'ees,}
\newblock Publ. Math. Inst. Hautes \'Etudes Sci. \textbf{41} (1972), 5--251.

\bibitem[BT84]{BT84}
F.~Bruhat and J.~Tits,
\newblock \emph{Groupes r\'eductifs sur un corps local.
II. Sch\'emas en groupes. Existence d'une donn\'ee radicielle valu\'ee,}
\newblock Publ. Math. Inst. Hautes \'Etudes Sci. \textbf{60} (1984), 5--184.

\bibitem[CL11]{CL11}
P. E. Caprace, J. Lécureux,
\newblock \emph{Combinatorial and group-theoretic compactifications of buildings,}
\newblock Ann. Inst. Fourier (Grenoble) 61 (2011), no. 2, 619–672.

\bibitem[CLSV11]{CLSV11}
T.~M.~Carlsen, N.~S.~Larsen, A.~Sims, and S.~T.~Vittadello,
\textit{Co-universal algebras associated to product systems, and gauge-invariant uniqueness theorems},
Proc. Lond. Math. Soc. (3) \textbf{103} (2011), no.~4, 563--600.
\href{https://doi.org/10.1112/plms/pdq028}{\path{doi:10.1112/plms/pdq028}}.

\bibitem[CF17]{CF17}
M. Caspers, P. Fima,
\newblock \emph{Graph products of operator algebras,}
\newblock Journal of Noncommutative Geometry 11.1 (2017): 367-411.
\newblock \href{https://ems.press/content/serial-article-files/30688}{\path{https://ems.press/content/serial-article-files/30688}}.
\bibitem[DJ99]{DJ99}
A.~N.~Dranishnikov and T.~Januszkiewicz,
\textit{Every Coxeter group acts amenably on a compact space},
Topology Proc. \textbf{24} (1999), 135--141.

\bibitem[Dyk04]{Dyk04}
K.~J.~Dykema,
\textit{Exactness of reduced amalgamated free product C$^*$-algebras},
Forum Math. \textbf{16} (2004), no.~2, 161--180.

\bibitem[Exe97]{Exe97}
R.~Exel,
\newblock Amenability for Fell bundles,
\newblock {\em J. Reine Angew. Math.}, 492 (1997), 41--74.

\bibitem[Exe17]{Exe17}
R.~Exel,
\textit{Partial Dynamical Systems, Fell bundles and Applications},
Mathematical Surveys and Monographs, vol.~224,
American Mathematical Society, Providence, RI, 2017.

\bibitem[FA18]{FA18}
M.~Forough and M.~Amini,
\newblock \emph{Hilbert \(C^*\)-bimodules of finite index and
approximation properties of \(C^*\)-algebras,}
\newblock Glasgow Math. J. \textbf{60} (2018), no.~2, 321--331.
\newblock
\href{https://doi.org/10.1017/S001708951700012X}
{\path{doi:10.1017/S001708951700012X}}.
\bibitem[Has17]{Has17}
K. Hasegawa,
\newblock \emph{Noncommutative Bass–Serre trees and their applications,}
\newblock PhD thesis
\newblock \href {https://catalog.lib.kyushu-u.ac.jp/opac_download_md/1866255/math0220.pdf}{\path{https://catalog.lib.kyushu-u.ac.jp/opac_download_md/1866255/math0220.pdf}}.

\bibitem[Has19]{Has19}
K. Hasegawa,
\newblock \emph{Bass-Serre trees of amalgamated free product $C^{\ast}$-algebras,}
\newblock Int. Math. Res. Not. IMRN 2019, no.
21, 6529–6553.
\newblock \href {https://arxiv.org/abs/1609.08837} {\path{https://arxiv.org/abs/1609.08837}}.

\bibitem[HI17]{HI17}
C.~Houdayer and Y.~Isono,
\textit{Unique prime factorization and bicentralizer problem for a class of type~III factors},
Adv. Math. \textbf{305} (2017), 402--455.

\bibitem[KS91]{KS91}
G.~Kasparov and G.~Skandalis,
\textit{Groups acting on buildings, operator $K$-theory, and Novikov's conjecture},
K-Theory \textbf{4} (1991), no.~4, 303--337.

\bibitem[KLT87]{KLT87}
W.~M.~Kantor, R.~A.~Liebler, and J.~Tits,
\newblock \emph{On discrete chamber-transitive automorphism groups of affine buildings,}
\newblock Bull. Amer. Math. Soc. (N.S.) \textbf{16} (1987), no.~1, 129--133.
\newblock
\href{https://doi.org/10.1090/S0273-0979-1987-15487-5}
{\path{doi:10.1090/S0273-0979-1987-15487-5}}.

\bibitem[Kat04]{Kat04}
T.~Katsura,
\textit{On $C^*$-algebras associated with $C^*$-correspondences},
J. Funct. Anal. \textbf{217} (2004), no.~2, 366--401.
\href{https://doi.org/10.1016/j.jfa.2004.03.010}{\path{doi:10.1016/j.jfa.2004.03.010}}.

\bibitem[Kir93]{Kir93}
E. Kirchberg,
\newblock \emph{On nonsemisplit extensions, tensor products and exactness of
group $C^{\ast}$-algebras,}
\newblock Invent. Math. 112 (1993), no. 3, 449--489.
\newblock \href{https://doi.org/10.1007/BF01232444}{\path{doi:10.1007/BF01232444}}.

\bibitem[Kli22]{Kli22}
M.~Klisse,
\emph{The Structure of Hecke Operator Algebras},
PhD thesis, Delft University of Technology, 2022.
\url{https://doi.org/10.4233/uuid:0f0982ee-3e3b-4364-ab6d-c8dd03435c0e}.

\bibitem[Kli23]{Kli23}
M.~Klisse,
\textit{Topological boundaries of connected graphs and Coxeter groups},
J. Operator Theory \textbf{89} (2023), no.~2, 429--476.
\href{https://arxiv.org/abs/2010.03414}{\path{https://arxiv.org/abs/2010.03414}}.

\bibitem[Kli25]{Kli25}
M. Klisse,
\newblock  \emph{Universal $C^{\ast}$-Algebras from Graph Products: Structure and Applications,}
\newblock arXiv preprint, 2025.
\newblock \href{https://arxiv.org/abs/2507.12271} {\path{https://arxiv.org/abs/2507.12271}}.

\bibitem[Lec09]{Lec09}
J. Lécureux,
\newblock{\em Amenability of actions on the boundary of a building, Int. Math. Res. Not. IMRN 2010 no. 17, 3265-3302.}
\newblock \href {https://arxiv.org/abs/0907.2033} {\path{https://arxiv.org/abs/0907.2033}}.

\bibitem[MRV24]{MRV24}
S.~A.~Mutter, A.-C.~Radu, and A.~Vdovina,
\textit{C$^*$-algebras of higher-rank graphs from groups acting on buildings, and explicit computation of their $K$-theory},
Publ. Mat. \textbf{68} (2024), no.~1, 187--217.
\href{https://doi.org/10.5565/PUBLMAT6812408}{\path{doi:10.5565/PUBLMAT6812408}}.

\bibitem[Nor17]{Nor17}
W.~Norledge,
\newblock \emph{Stacky buildings,}
\newblock arXiv preprint arXiv:1710.06968, 2017,
\newblock
\href{https://arxiv.org/abs/1710.06968}
{\path{https://arxiv.org/abs/1710.06968}}.

\bibitem[Nor21]{Nor21}
W.~Norledge,
\newblock \emph{Quotients of buildings by groups acting freely on chambers,}
\newblock J. Pure Appl. Algebra \textbf{225} (2021), no.~11, Article 106730.
\newblock
\href{https://doi.org/10.1016/j.jpaa.2021.106730}
{\path{doi:10.1016/j.jpaa.2021.106730}}.

\bibitem[Rob00]{Rob00}
G.~Robertson,
\textit{Boundary Actions for Affine Buildings and Higher Rank Cuntz--Krieger Algebras},
in: J.~Cuntz and S.~Echterhoff (eds.),
\textit{$C^*$-Algebras},
Springer, Berlin, Heidelberg, 2000, pp.~182--196.
\href{https://doi.org/10.1007/978-3-642-57288-3_10}{\path{doi:10.1007/978-3-642-57288-3_10}}.

\bibitem[RS99]{RS99}
G.~Robertson and T.~Steger,
\textit{Affine buildings, tiling systems and higher rank Cuntz--Krieger algebras},
J. Reine Angew. Math. \textbf{513} (1999), 115--144.
\bibitem[Ron09]{Ron09}
M.~Ronan,
\emph{Lectures on Buildings},
updated and revised ed.,
University of Chicago Press, Chicago, 2009.
\bibitem[Seh19]{Seh19}
C.~F.~Sehnem,
\textit{On $C^*$-algebras associated to product systems},
J. Funct. Anal. \textbf{277} (2019), no.~2, 558--593.
\href{https://doi.org/10.1016/j.jfa.2018.10.012}{\path{doi:10.1016/j.jfa.2018.10.012}}.

\bibitem[Ser77]{Ser77}
J.-P.~Serre,
\textit{Arbres, amalgames, $\mathrm{SL}_2$},
Ast\'erisque, no.~46,
Soci\'et\'e Math\'ematique de France, Paris, 1977.

\bibitem[Sol07]{Sol07}
M.~S.~Solleveld,
\textit{Periodic cyclic homology of affine Hecke algebras},
Ph.D. thesis, Universiteit van Amsterdam, Amsterdam, 2007.

\bibitem[Voi85]{Voi85}
D. Voiculescu,
\newblock{\em Symmetries of some reduced free product C$^*$
-algebras,}
\newblock Operator algebras and their connections
with topology and ergodic theory (Busteni, 1983), 556–588. Lecture Notes in Math., 1132 Springer-Verlag,
Berlin, 1985.

\end{thebibliography}
\end{document}